\documentclass{article}
\usepackage{amsmath,amssymb,graphicx,subfig,mathrsfs,amsthm}
\newcommand{\norm}[1]{\left\Vert #1 \right\Vert}
\def\Re{\ensuremath{\mathrm{Re}}\,}
\def\Im{\ensuremath{\mathrm{Im}}\,}
\def\sgn{\ensuremath{\mathrm{sgn}}\,}
\newtheorem{theorem}{Theorem}
\newtheorem{lemma}[theorem]{Lemma}
\newtheorem{proposition}[theorem]{Proposition}

\begin{document}
\title{Sums, series, and products in Diophantine approximation}
\author{Jordan Bell \\
jordan.bell@gmail.com \\
Department of Mathematics, University of Toronto\\
Toronto, Ontario, Canada}
\date{\today}

\maketitle

\begin{abstract}
There is not much that can be said for all $x$ and for all $n$ about the sum \[ \sum_{k=1}^n \frac{1}{|\sin k\pi x|}. \] However, for this and similar sums, series, and products, we can establish results for almost all $x$ using the tools of continued fractions. We present in detail the appearance of these sums in the singular series for the circle method. One particular interest of the paper is the detailed proof of a striking result of Hardy and Littlewood, whose compact proof, which delicately uses analytic continuation, has not been written freshly anywhere since its original publication. This story includes various parts of late 19th century and early 20th century mathematics.
\end{abstract}

\tableofcontents

\section{Introduction}
In this paper we survey a class of estimates for sums, series, and products that involve Diophantine approximation.
We both sort out a timeline of the literature on these questions and give careful proofs of many lesser known results.
Rather than being an open-pit mine for the history of Diophantine approximation, this paper follows one vein as deep as it goes.

For $x \in \mathbb{R}$, we write $\norm{x}=\min_{n \in \mathbb{Z}} |x-n|$, the distance from $x$ to a nearest integer. 
In this paper we give a comprehensive presentation of estimates for sums whose $j$th term involves $\norm{jx}$ and
determine the abscissa of convergence and radius of convergence respectively of Dirichlet series and power series whose
$j$th term involves $\norm{jx}$. We also give a detailed proof of a result of Hardy and Littlewood that 
\[
\lim_{n \to \infty} \left( \prod_{k=1}^n |\sin k\pi x| \right)^{1/n} = \frac{1}{2}
\]
for almost all $x$.

We either prove or state in detail and give references for all the material on continued fractions and measure theory that we use in this paper. 
Many of the results we prove in this paper do not have detailed proofs written in any books, and the proofs we give for results
that do have proofs in books are often written significantly more  meticulously here than anywhere else; in some cases the proofs
in the literature are so sketchy that the proof we give is written from scratch, for example Lemma \ref{h0}.
Our presentation of Hardy and Littlewood's estimate for
$\prod_{k=1}^n |\sin \pi kx|$  makes clear exactly what results in Diophantine approximation one needs for the proof.

In the next section we introduce the Bernoulli polynomials, the Euler-Maclaurin summation formula, and Euler's constant, which we shall use in a few places.
Because the calculations are similar to what we do in the rest of this paper and because we want to be comfortable using Bernoulli polynomials, 
we work things out from scratch rather than merely stating results as known.
In the section after  that we summarize various problems that involve sums of the type we are talking about in this paper.

\section{The Bernoulli polynomials, the Euler-Maclaurin summation formula, and Euler's constant}
For $k \geq 0$, the \textbf{Bernoulli polynomial} $B_k(x)$ is defined by
\begin{equation}
\frac{ze^{xz}}{e^z-1} = \sum_{k=0}^\infty B_k(x) \frac{z^k}{k!},\qquad |z|<2\pi.
\label{bernoullipolynomials}
\end{equation}
The \textbf{Bernoulli numbers} are $B_k = B_k(0)$, the constant terms of the Bernoulli polynomials.
For any $x$, using L'Hospital's rule the left-hand side of \eqref{bernoullipolynomials} tends to $1$ as $z \to 0$, and the right-hand side
tends to $B_0(x)$, hence $B_0(x)=1$. 
Differentiating \eqref{bernoullipolynomials} with respect to $x$,
\[
\sum_{k=0}^\infty B_k'(x) \frac{z^k}{k!} = \frac{z^2 e^{xz}}{e^z-1} = \sum_{k=0}^\infty B_k(x) \frac{z^{k+1}}{k!}
=\sum_{k=1}^\infty B_{k-1}(x) \frac{z^k}{(k-1)!},
\]
so $B_0'(x) = 0$ and for $k \geq 1$ we have $\frac{B_k'(x)}{k!} = \frac{B_{k-1}(x)}{(k-1)!}$, i.e.
\[
B_k'(x)=k B_{k-1}(x).
\]
Furthermore, integrating \eqref{bernoullipolynomials} with respect
to $x$ on $[0,1]$ gives, since $\int e^{xz} dx = \frac{e^z-1}{z}$,
\[
1 = \sum_{k=0}^\infty \left( \int_0^1 B_k(x) dx \right) \frac{z^k}{k!},\qquad |z|<2\pi,
\]
hence $\int_0^1 B_0(x) dx =1$ and for $k \geq 1$,
\[
\int_0^1 B_k(x) dx = 0.
\] 
The first few Bernoulli polynomials are
\[
B_0(x)=1,\quad B_1(x) = x-\frac{1}{2},\quad B_2(x) = x^2-x+\frac{1}{6},
\quad B_3(x)=x^3-\frac{3}{2}x^2+\frac{1}{2}x.
\]

The Bernoulli polynomials  satisfy the following:
\begin{align*}
\sum_{k=0}^\infty B_k(x+1) \frac{z^k}{k!}&=\frac{ze^{(x+1)z}}{e^z-1}\\
&=\frac{ze^{xz} (e^z-1+1)}{e^z-1}\\
&=ze^{xz}+\frac{ze^{xz}}{e^z-1}\\
&=\sum_{k=0}^\infty \frac{x^k z^{k+1}}{k!} + \sum_{k=0}^\infty B_k(x) \frac{z^k}{k!}\\
&=\sum_{k=1}^\infty \frac{x^{k-1} z^k}{(k-1)!} +  \sum_{k=0}^\infty B_k(x) \frac{z^k}{k!},
\end{align*}
hence
\begin{equation}
B_k(x+1) = kx^{k-1} + B_k(x),\qquad k \geq 1, x \in \mathbb{R}.
\label{bernoulliplus}
\end{equation}
In particular, for $k \geq 2$,
$B_k(1)=B_k(0)$.
The  identity  \eqref{bernoulliplus} yields \textbf{Faulhaber's formula}, for $k \geq 0$ and positive integers $a < b$,
\begin{equation}
\sum_{a \leq m \leq b} m^k = \frac{B_{k+1}(b+1)-B_{k+1}(a)}{k+1}.
\label{faulhaber}
\end{equation}

The following identity is  the \textbf{multiplication formula for the Bernoulli polynomials}, found by Raabe \cite[pp.~19--24,
\S 13]{raabe}.

\begin{lemma}
For $k \geq 0$, $q \geq 1$, and $x \in \mathbb{R}$,
\[ 
q B_k(qx) = q^k \sum_{j=0}^{q-1} B_k\left(x+\frac{j}{q}\right).
\]
\label{bernoullisum}
\end{lemma}
\begin{proof}
Using \eqref{bernoulliplus} with $x=\frac{qn+j}{q}=n+\frac{j}{q}$,
\[
(qn+j)^k  = \frac{q^k}{k+1} \left( B_{k+1} \left(n+\frac{j}{q}+1\right) - B_{k+1}\left(n+\frac{j}{q}\right)\right),
\]
 thus
\begin{align*}
\sum_{m=q}^{Nq-1} m^k& = \sum_{n=1}^{N-1} \sum_{j=0}^{q-1} (nq+j)^k \\
&= \frac{q^k}{k+1} \sum_{j=0}^{q-1} \sum_{n=1}^{N-1} \left( B_{k+1} \left(n+\frac{j}{q}+1\right) - B_{k+1}\left(n+\frac{j}{q}\right)\right)\\
&= \frac{q^k}{k+1} \sum_{j=0}^{q-1} \left( B_{k+1} \left(N+\frac{j}{q}\right)-B_{k+1}\left(1+\frac{j}{q}\right)\right).
\end{align*}
Then by \eqref{faulhaber},
\[
B_{k+1}(qN)-B_{k+1}(q) = q^k \sum_{j=0}^{q-1} \left( B_{k+1} \left(N+\frac{j}{q}\right)-B_{k+1}\left(1+\frac{j}{q}\right)\right).
\]
Let $F_{k,q}(x) = B_{k+1}(qx) - q^k \sum_{j=0}^{q-1} B_{k+1} (x+ j/q)$, 
which is a polynomial of degree $\leq k+1$. The above means 
that for $N \geq 1$, $F_{k,q}(N)=F_{k,q}(1)$, which implies that the polynomial $F_{k,q}(x)-F_{k,q}(1)$ is identically $0$
and, a fortiori, $F_{k,q}'(x)$ is identically $0$:
\[
q B_{k+1}'(qx) - q^k \sum_{j=0}^{q-1} B_{k+1}'\left(x+\frac{j}{q} \right) = 0.
\]
Using $B_{k+1}'(x) = (k+1) B_k(x)$,
\[
q B_k(qx) = q^k \sum_{j=0}^{q-1} B_k\left(x+\frac{j}{q}\right).
\]
\end{proof}

We remark that for prime $p$ and for $k \geq 0$, one uses Lemma \ref{bernoullisum} to prove that
there is a unique $p$-adic distribution $\mu_{B,k}$ on the $p$-adic integers $\mathbb{Z}_p$ such that
$\mu_{B,k}(a+p^N \mathbb{Z}_p) = p^{N(k-1)} B_k(a/p^N)$ \cite[p.~35, Chapter II, \S 4]{koblitz},
called a \textbf{Bernoulli distribution}.

For $x \in \mathbb{R}$, let $[ x ]$ be the greatest integer $\leq x$, and let $R(x)=x-[x]$, called the \textbf{fractional part of $x$}.
Write $\mathbb{T}=\mathbb{R}/\mathbb{Z}$ and define the 
\textbf{periodic Bernoulli functions} $P_k:\mathbb{T} \to \mathbb{R}$ by
\[
P_k = B_k \circ R.
\] 
For $k \geq 2$, because $B_k(1)=B_k(0)$, the function $P_k$ is continuous. 
For $f:\mathbb{T} \to \mathbb{C}$ define its \textbf{Fourier series} $\hat{f}:\mathbb{Z} \to \mathbb{C}$ by
\[
\hat{f}(n) = \int_\mathbb{T} f(t) e^{-2\pi int} dt,\qquad n \in \mathbb{Z}.
\]
For $k \geq 1$, one calculates $\widehat{P}_k(0)=0$ and using integration by parts, $\widehat{P}_k(n)= - \frac{1}{(2\pi in)^k}$ for
$n \neq 0$.
Thus for $k \geq 1$, the Fourier series of $P_k$ is
\[
P_k(t) \sim \sum_{n \in \mathbb{Z}} \widehat{P}_k(n) e^{2\pi int} = -\frac{1}{(2\pi i)^k} \sum_{n \neq 0} n^{-k} e^{2\pi int}.
\]
For $k \geq 2$, $\sum_{n \in \mathbb{Z}} |\widehat{P}_k(n)|<\infty$, from which it follows that $\sum_{|n| \leq N}  \widehat{P}_k(n) e^{2\pi int}$
converges to $P_k(t)$ uniformly for $t \in \mathbb{T}$.
Furthermore, for $t \not \in \mathbb{Z}$  \cite[p.~499, Theorem B.2]{multiplicative},
\[
P_1(t) = - \frac{1}{\pi} \sum_{n=1}^\infty \frac{1}{n} \sin 2\pi nt.
\]
Thus for example,
\[
B_1\left(\frac{1}{2\pi} \right) = P_1\left(\frac{1}{2\pi} \right) = - \frac{1}{\pi} \sum_{k=1}^\infty \frac{\sin k}{k}.
\]

The \textbf{Euler-Maclaurin summation formula} is the following \cite[p.~500, Theorem B.5]{multiplicative}. If $a<b$ are real numbers,
$K$ is a positive integer, and $f$ is a $C^K$ function on an open set that contains $[a,b]$, then 
\begin{align*}
\sum_{a<m \leq b} f(m)&=\int_a^b f(x) dx + \sum_{k=1}^K \frac{(-1)^k}{k!} (P_k(b) f^{(k-1)}(b)-P_k(a) f^{(k-1)}(a))\\
&-\frac{(-1)^K}{K!} \int_a^b P_K(x)  f^{(K)}(x) dx.
\end{align*}

Applying the Euler-Maclaurin summation formula with $a=1, b=n, K=2, f(x)=\log x$ yields \cite[p.~503, Eq. B.25]{multiplicative}
\[
\sum_{1 \leq m \leq n} \log n = n \log n - n  + \frac{1}{2} \log n + \frac{1}{2} \log 2\pi + O(n^{-1}).
\]
Using $e^{1+O(n^{-1})} = 1+O(n^{-1})$, taking the exponential of the above equation gives \textbf{Stirling's approximation},
\[
n! = n^n e^{-n} \sqrt{2\pi n} (1+O(n^{-1})).
\]

Write $a_n  = - \log n+ \sum_{1 \leq m \leq n} \frac{1}{m}$. Because $x \mapsto \log(1-x)$ is concave,
\[
a_n - a_{n-1} = \frac{1}{n} + \log\left(1-\frac{1}{n} \right) \leq 1 + 1 - \frac{1}{n}=0,
\] 
which means that the sequence $a_n$ is nonincreasing. For $f(x)=\frac{1}{x}$, because
$f$ is positive and nonincreasing,
\[
\sum_{1 \leq m \leq n} f(m) \geq \int_1^{n+1} f(x) dx = \log(n+1) > \log n,
\]
hence $a_n>0$. Because the sequence $a_n$ is positive and nonincreasing, there exists some nonnegative limit,
$\gamma$, called \textbf{Euler's constant}. 
Using the Euler-Maclaurin summation formula with $a=1, b=n, K=1, f(x)=\frac{1}{x}$, as $P_1(x) = [x]-\frac{1}{2}$,
\[
\sum_{1 < m \leq n} \frac{1}{m} = \log n + \frac{1}{2n} - \frac{1}{2} + \frac{1}{2} \int_1^n \frac{1}{x^2} dx - \int_1^n R(x) \frac{1}{x^2} dx,
\]
which is
\[
\sum_{1<m \leq n} \frac{1}{m} = \log n - \int_1^\infty \frac{R(x)}{x^2} dx + \int_n^\infty \frac{R(x)}{x^2} dx.
\]
As $0 \leq R(x) x^{-2} \leq x^{-2}$, the function $x \mapsto R(x)x^{-2}$ is integrable
on $[1,\infty)$; let $C=1-\int_1^\infty R(x)x^{-2}$.
Since $0 \leq \int_n^\infty R(x)x^{-2} dx \leq \int_n^\infty x^{-2} dx = n^{-1}$, 
\[
\sum_{1 \leq m \leq n} \frac{1}{m} = \log n + C + O(n^{-1})
\]
f. But $-\log n + \sum_{1 \leq m \leq n} \frac{1}{m} \to \gamma$ as $n \to \infty$, from which it follows that
$C=\gamma$ and thus
\[
\sum_{1 \leq m \leq n} \frac{1}{m} = \log n + \gamma + O(n^{-1}).
\]

\section{Background}
\label{background}
For $x \in \mathbb{R}$, let $[ x ]$ be the greatest integer $\leq x$ and let $R(x)=x-[x]$.
It will be handy to review some properties of $x \mapsto [x]$.
For $n \in \mathbb{Z}$ it is immediate that $[x+n]=[x]$. 
For $x,y \in \mathbb{R}$ we have
$0 \leq x-[x] + y - [y] < 2$, which means
$0 = [0] \leq [x-[y]+y-[y]] < [2] = 2$, and using $[x+n]=[x]$ this is $0 \leq [x+y]-[x]-[y] < 2$, therefore
\[
[x] + [y] \leq [x+y] \leq [x]+[y]+1.
\]
For $m,n \in \mathbb{Z}$, $n>1$, and $x \in \mathbb{R}$,
\[
\left[ \frac{x+m}{n} \right] = \left[ \frac{[x]+m}{n} \right].
\]
For $n$ a positive integer and for real $x$,
\[
[x]=\left[x+\frac{0}{n}\right] \leq \left[x+\frac{1}{n}\right] \leq \cdots \left[x+\frac{n-1}{n} \right] \leq \left[x+\frac{n}{n} \right]
=[x]+1.
\]
There is a unique $\nu$, $1 \leq \nu \leq n$, such that $\left[x+\frac{\nu-1}{n} \right]= [x]$ and
$\left[x+\frac{\nu}{n} \right] = [x]+1$, 
and therefore $\left[R(x) + \frac{\nu-1}{n} \right] = 0$ and 
$\left[ R(x) + \frac{\nu}{n} \right]=1$, consequently
$R(x)+\frac{\nu-1}{n}<1$ and $R(x)+\frac{\nu}{n} \geq 1$, which means
$1-\frac{\nu}{n} \leq R(x) < 1-\frac{\nu-1}{n}$, from which finally we get
$n \leq [nx]-n[x]+\nu<n+1$ and so $n=[nx]-n[x]+\nu$.
But
\[
\sum_{k=0}^{n-1} \left[x+\frac{k}{n} \right] = \sum_{k=0}^{\nu-1} [x] + \sum_{k=\nu}^{n-1} ([x]+1)
=\nu [x] + (n-\nu)([x]+1)
=n[x]+n-\nu,
\]
and using $\nu=n-[nx]+n[x]$,
\[
\sum_{k=0}^{n-1} \left[x+\frac{k}{n} \right]  = n[x]+n-(n-[nx]+n[x]) = [nx].
\] 
This identity is proved by Hermite \cite[pp.~310--315, \S V]{hermite}.

The \textbf{Legendre symbol} is defined in the following way. Let $p$ be an odd prime, and let $a$ be an integer that is not a multiple of $p$.
If there is an integer $b$ such that $a \equiv b^2 \pmod{p}$ then
$\left( \frac{a}{p} \right)=1$, and otherwise $\left( \frac{a}{p} \right)=-1$. In other words, for an integer $a$ that is relatively prime to $p$, if $a$ is a square mod $p$ then
$\left( \frac{a}{p} \right)=1$, and if $a$ is not a square mod $p$ then $\left( \frac{a}{p} \right)=-1$. For example, one checks that there is no integer $b$ such that
$b^2 \equiv 6 \pmod{7}$, and hence $\left( \frac{6}{7} \right)=-1$, while $3^2 \equiv 2 \pmod{7}$, and so $\left( \frac{2}{7} \right)=1$.

If $p$ and $q$ are distinct odd primes, define integers $u_k$, $1 \leq k \leq \frac{p-1}{2}$, by
\[
kq=p \left[ \frac{kq}{p} \right] + u_k;
\]
namely, $u_k$ is the remainder of $kq$ when divided by $p$.
We have $1 \leq u_k \leq p-1$. Let $\mu(q,p)$ be the number of $k$ such that $u_k > \frac{p-1}{2}$. 
It can be shown that
\cite[p.~74, Theorem 92]{wright} 
\[
\left( \frac{q}{p} \right)=(-1)^{\mu(q,p)};
\]
this fact is called \textbf{Gauss's lemma}.
For example, for $p=13$ and $q=3$ we work out that
\[
u_1=3, \quad u_2=6, \quad u_3 = 9, \quad u_4=12, \quad u_5=2, \quad u_6=5,
\]
and hence $\mu(3,13)=2$, so $(-1)^{\mu(3,13)}=1$; on the other hand, $4^2 \equiv 3 \pmod{13}$, hence $\left( \frac{3}{13} \right)=1$.
With
\[
S(q,p)=\sum_{j=1}^{\frac{p-1}{2}} \left[ \frac{jq}{p} \right],
\]
it
is known that  \cite[pp.~77-78, \S 6.13]{wright} 
\[
S(q,p) \equiv \mu(q,p) \pmod{2}.
\]
And it can  be shown that \cite[p.~76, Theorem 100]{wright} 
\begin{equation}
S(q,p)+S(p,q)=\frac{p-1}{2} \cdot \frac{q-1}{2}.
\label{Sqp}
\end{equation}
Thus
\begin{eqnarray*}
\left( \frac{p}{q} \right) \left( \frac{q}{p} \right)&=&(-1)^{\mu(p,q)+\mu(q,p)}\\
&=&(-1)^{S(p,q)+S(q,p)}\\
&=&(-1)^{\frac{p-1}{2} \cdot \frac{q-1}{2}}.
\end{eqnarray*}
This is Gauss's third proof of the \textbf{law of quadratic reciprocity} in the numbering \cite[p.~50, \S 20]{werke}. This proof
was published in Gauss's 1808 ``Theorematis arithmetici demonstratio nova'', which is translated in \cite[pp.~112--118]{smith}.
Dirichlet \cite[pp.~65--72, \S\S 42--44]{dirichlet} gives a presentation of the proof.
Eisenstein's streamlined version of Gauss's third proof is presented with historical remarks in \cite{eisenstein}.
Lemmermeyer \cite{lemmermeyer} gives a comprehensive history of the law of quadratic reciprocity, and in particular writes about Gauss's third proof  \cite[pp.~9--10]{lemmermeyer}.
The formula \eqref{Sqp} resembles the reciprocity formula for Dedekind sums \cite[p.~4, Theorem~1]{rademacher}.

Gauss obtains \eqref{Sqp} from the following \cite[p.~116, \S 5]{smith}:
if $x$ is irrational and $n$ is a positive integer, then
\begin{equation}
\sum_{k=1}^n [kx]+\sum_{k=1}^{[nx]} \left[ \frac{k}{x} \right]=n[nx],
\label{gauss}
\end{equation}
which he proves as follows.
If $\left[\frac{j}{x}\right] < k \leq \left[\frac{j+1}{x}\right]$, then $[kx]=j$. Therefore
\begin{eqnarray*}
\sum_{k=1}^n [kx]&=&\sum_{j=1}^{[nx]} j \left( \left[\frac{j+1}{x} \right]- \left[\frac{j}{x} \right] \right)\\
&=&n[nx]-\sum_{j=1}^{[nx]}  \left[\frac{j}{x} \right].
\end{eqnarray*}
Bachmann \cite[pp.~654--658, \S 4]{bachmann1904} surveys later work on sums similar to \eqref{gauss}; see also Dickson \cite[Chapter X]{dicksonI}.
If $m$ and $n$ are relatively prime, then
\[
\left\{R\left(\frac{m}{n}\right), R\left(\frac{2m}{n}\right), \ldots, R\left(\frac{(n-1)m}{n}\right)\right\} =\left\{\frac{1}{n},\frac{2}{n},\ldots,\frac{n-1}{n}\right\},
\]
and so
\[
\sum_{k=1}^{n-1} R\left( \frac{km}{n} \right)=\sum_{k=1}^{n-1} \frac{k}{n}=\frac{1}{n} \cdot \frac{(n-1)n}{2}=\frac{n-1}{2}.
\]
Hence
\begin{equation}
\sum_{k=1}^{n-1} \left[ \frac{km}{n} \right] = \sum_{k=1}^{n-1} \frac{km}{n} - \sum_{k=1}^{n-1} R\left( \frac{km}{n} \right)=\frac{(m-1)(n-1)}{2}.
\label{stern}
\end{equation}
There is also a simple lattice point counting argument \cite[p.~113, No.~18]{polyaII} that gives \eqref{stern}.

In 1849, Dirichlet \cite{mittleren} shows that 
\[
\sum_{k=1}^n d(k) = \sum_{k=1}^n \left[  \frac{n}{k} \right],
\]
where $d(n)$ denotes the number of positive divisors of an integer $n$. (This equality is Dirichlet's ``hyperbola method''.) He then proves that
\[
\sum_{k=1}^n \left[  \frac{n}{k} \right] = n \log n + (2\gamma -1)n +O(\sqrt{n}).
\]
Hardy and Wright \cite[pp.~264--265, Theorem 320]{wright} give a proof of this.
Finding the best possible error term in the estimate for $\sum_{k=1}^n d(k)$ is ``Dirichlet's divisor problem''.  Dirichlet cites  
the end of Section V Gauss's {\em Disquisitiones Arithmeticae}
as precedent for determining average magnitudes of arithmetic functions. (In Section V, Articles 302--304, of the {\em Disquisitiones Arithmeticae},
Gauss writes about averages of class numbers of binary quadratic forms, cf. \cite[Chapter VI]{dicksonIII}.)

Define $(x)$ to be $0$ if $x  \in \mathbb{Z}+\frac{1}{2}$; if $x \not \in \mathbb{Z}+\frac{1}{2}$ then there is
an integer $m_x$ for which $|x-m_x| < |x-n|$ for all integers $n \neq m_x$, and we define $(x)$ to be $x-m_x$. 
Riemann \cite[p.~105, \S 6]{riemann} defines 
\[
f(x)=\sum_{n=1}^\infty \frac{(nx)}{n^2};
\]
for any $x$, the series converges absolutely because $|(-nx)|<\frac{1}{2}$. 
Riemann states that if  $p$ and $m$ are relatively prime and $x=\frac{p}{2m}$,  then
\[
f(x^+)=\lim_{h \to 0^+} f(x+h) = f(x)-\frac{\pi^2}{16m^2}, \qquad f(x^-)=\lim_{h \to 0^-} f(x+h) = f(x)+\frac{\pi^2}{16m^2},
\]
thus
\[
f(x^-)-f(x^+)=\frac{\pi^2}{8m^2},
\]
and hence that $f$ is discontinuous at such points, and says that at all other points $f$ is continuous; see Laugwitz \cite[\S 2.1.1, pp. 183-184]{laugwitz},
Neuenschwander \cite{neuenschwander}, and Pringsheim \cite[p.~37]{pringsheim} about Riemann's work on pathological functions.
The role of pathological functions in the development of set theory is explained by Dauben \cite[Chapter 1, p. 19]{dauben} and Ferreirós \cite[Chapter V, \S 1, p. 152]{Labyrinth}.

For any interval $[a,b]$ and any $\sigma>0$, it is apparent from the above that there are only finitely many $x \in [a,b]$ for which
$f(x^-)-f(x^+)>\sigma$, and Riemann deduces from this that $f$ is  Riemann integrable on $[a,b]$; cf. Hawkins  \cite[p.~18]{hawkins} on the history
of Riemann integration.
Later in the same paper  \cite[p.~129, \S 13]{riemann}, Riemann states that the function
\[
x \mapsto \sum_{n=1}^\infty \frac{(nx)}{n},
\] 
is not Riemann integrable in any interval.

In 1897, Ces\`aro \cite{rosace} asks the following question (using the pseudonym, and anagram, ``Rosace'' \cite[p.~331]{perna}).
Let $\epsilon(x)=x-[x]-\frac{1}{2}$. Is the series
\begin{equation}
\sum_{n=1}^\infty \frac{\epsilon(nx)}{n}
\label{Bseries}
\end{equation}
convergent for all non-integer $x$? This is plausible because the expected value of $\epsilon(x)$ is 0. Landau \cite{landau} answers this question in 1901. Landau proves that if there is some $g$ such that 
\[
\sum_{k=1}^n [kx]=\frac{n(n+1)x}{2}-\frac{n}{2}+O(g(n))
\]
where $g$ is a nonnegative function such $g(n)=o(n)$ and such that $\sum_{n=1}^\infty \frac{g(n)}{n(n+1)}$ converges, then \eqref{Bseries} converges. And he proves that if $x$ is rational then \eqref{Bseries} diverges.
We
return to this series in \S \ref{dirichlet}.

Also in 1898, Franel \cite{franel1260} asks whether for irrational $x$ and for
$\epsilon>0$ we have
\[
\sum_{k=1}^n [kx]=\frac{n(n+1)x}{2}-\frac{n}{2}+O(n^\epsilon).
\]
Then in 1899, Franel \cite{franel1547} asks if we can do better than this: is the error term in fact $O(1)$?
Ces\`aro and Franel each contributed many problems to {\em L'Inter\-m\'ediaire des math\-\'em\-aticiens}, the periodical in which they posed their questions. Information
about Franel is given in \cite{kollros}.

Lerch \cite{lerch} answers Franel's questions in 1904. If $x$ is irrational and
$\frac{p}{q}$ is a convergent of $x$ (which we will define in \S \ref{preliminaries}), then using 
Theorem \ref{bestapproximation} (from \S \ref{preliminaries}) we can show that if $1 \leq k \leq q$ then $[k x]=[\frac{k p}{q}]$.
Lerch uses this and \eqref{stern} to show that if $x$ is irrational and $\frac{p}{q}$ is a convergent of $x$ then
\[
\sum_{k=1}^q [k x]=\frac{q(q+1)x}{2}-\frac{q}{2}+R, \qquad 0<R<\frac{1}{2}.
\]
Lerch states that if the continued fraction expansion of $x$ has bounded partial quotients (defined in \S \ref{preliminaries}) then, for any positive integer $n$,
\[
\sum_{k=1}^n [kx]=\frac{n(n+1)x}{2}-\frac{n}{2}+O(\log n).
\]
Lerch only gives a brief indication of the proof of this. This result is proved by Hardy and Littlewood in 1922 \cite[p.~24, Theorem B3]{latticeI},
and also in 1922 by Ostrowski \cite[p.~81]{ostrowski1922}.
On the other hand, Lerch also constructs examples of $x$ such that, for some positive integer $c$,
\[
\left| \sum_{k=1}^n [kx]-\frac{n(n+1)x}{2}+\frac{n}{2} \right| \gg n^{1-\frac{1}{c}}.
\]
Nevertheless, in 1909 Sierpinski \cite{sierpinski} proves that if $x$ is irrational then
\[
\sum_{k=1}^n [kx]=\frac{n(n+1)x}{2}-\frac{n}{2}+o(n).
\] 
A bibliography of Lerch's works is given in \cite{skrasek}. Lerch had written
earlier papers on Gauss sums, Fourier series, theta functions, and the class number; many of his papers are in Czech, but some of them are in French, several of which
were published in the Paris {\em Comptes rendus}.
Several of Lerch's papers are discussed in Cresse's survey  of the class number of binary quadratic forms \cite[Chapter VI]{dicksonIII}.

In 1899, a writer using the pseudonym ``Quemquaeris''  \cite{quemquaeris} (``quem quaeris'' = ``whom you seek'') asks if we can characterize $\phi(n)$ such that for
all  irrational $\theta$
the series
\[
\sum_{n=1}^\infty \frac{\phi(n)}{\sin n\pi \theta}
\]
converges.
 In particular, the writer asks if $\phi(n)=\frac{1}{n!}$ satisfies this. In the same year, de la Vall\'ee-Poussin \cite{valleepoussin}
answers this question. (There are also responses following de la Vall\'ee-Poussin's by Borel and Fabry.) For a given function $\phi(n)$,
de la Vall\'ee-Poussin shows that if we have $a_n>\frac{1}{\phi(q_{n-1})}$ for all $n$,  for $a_n$ the $n$th partial quotient of $\theta$ and $q_n$ the denominator of the $n$th convergent of $\theta$,
then the series
\[
\sum_{n=1}^\infty \frac{\phi(n)}{\sin n\pi \theta}
\]
will diverge. Hardy and Littlewood prove numerous results on similar series, e.g. for $\phi(n)=n^{-r}$ for real $r>1$ and for certain classes of $\theta$, in their papers on Diophantine approximation \cite{collectedpapers}. 
In 1931, Walfisz  \cite[p.~570, Hilfssatz 4]{walfisz} shows, following work of Behnke \cite[p.~289, \S 16]{behnke1924}, that for almost all irrational $x\in [0,1]$, if $\epsilon>0$ then
\[
\sum_{k=1}^n \frac{1}{\norm{k\theta}} = O(n (\log n)^{2+\epsilon}),
\]
where $\norm{x}=\min(R(x),1-R(x))$.
Walfisz's paper includes many results on related sums.

In 1916, Watson \cite{watson}
finds the following asymptotic series for
$S_n=\sum_{m=1}^{n-1} \csc\left( \frac{m\pi}{n} \right)$:
\[
\pi S_n \sim 2n\log(2n)+2n(\gamma-\log \pi)+\sum_{j=1}^\infty (-1)^j \frac{B_{2j}^2 ( 2^{2j}-2)\pi^{2j}}{j(2j)! n^{2j-1}},
\]
where $\gamma$ is Euler's constant and $B_j$ are the Bernoulli numbers.
Truncating the asymptotic series and rewriting
gives
\[
S_n=\frac{2n\log n}{\pi}+n \cdot \frac{2\log 2+2\gamma-2\log \pi}{\pi}+O\left(\frac{1}{n}\right).
\]
For example, computing $S_{1000}$ directly we get $S_{1000}=4477.593932\ldots$, and computing the right-hand side of the above formula  without the error term we obtain
$4477.594019\ldots$
A cleaner derivation of the asymptotic series using the Euler-Maclaurin summation formula is given
later by Williams \cite{williams}.

Early surveys of Diophantine approximation are given by Bohr and Cram\'er \cite[pp.~833--836, \S 39]{bohr} and Koksma \cite[pp.~102--110]{koksma}.
Hlawka and Binder \cite{hlawka} present the history of the initial years of the theory of uniform distribution.
Narkiewicz \cite[pp.~82--95, \S 2.5 and pp.~175--183 \S 3.5]{narkiewicz} gives additional historical references on Diophantine approximation.
The papers of Hardy and Littlewood on Diophantine approximation are reprinted in \cite{collectedpapers}. Perron \cite{perron}
and Brezinski \cite{brezinski} give historical references on continued fractions,
and there is  reliable material on the use of continued fractions by 17th century mathematicians in Whiteside \cite{whiteside}.
Fowler \cite{fowler}
presents a  prehistory of continued fractions in Greek mathematics.

\section{Preliminaries on continued fractions}
\label{preliminaries}
Let 
\[
\norm{x}=\min_{k \in \mathbb{Z}} |x-k|=\min(R(x),1-R(x)).
\]
Let $\mu$ be Lebesgue measure on $[0,1]$, and let $\Omega=[0,1] \setminus \mathbb{Q}$.

For positive integers $a_1,\ldots,a_n$, we define
\[
[a_1,\ldots,a_n]=\cfrac{1}{a_1+\cfrac{1}{\cdots+\cfrac{1}{a_{n-1}+\cfrac{1}{a_n}}}}.
\]
For example, $[1,1,1]=\frac{2}{3}$.

Let $\mathbb{N}$ be the set of positive integers. We call $a \in \mathbb{N}^\mathbb{N}$ a \textbf{continued fraction}, and we call
$a_n$ the \textbf{$n$th partial quotient} of $a$. If there is some $K>0$ such that $a_n \leq K$ for all $n$ then we say that $a$ has bounded partial quotients.
We call
$[a_1,\ldots,a_n]$ the \textbf{$n$th convergent} of $a$. For $n \geq 1$ let
\[
\frac{p_n}{q_n}=[a_1,\ldots,a_n],
\]
with $p_n,q_n$ positive integers that are relatively prime, and set
\[
p_0=0,\qquad q_0=1.
\]
One can show by induction  \cite[p.~70, Lemma 3.1]{einsiedler} that for $n \geq 1$ we have
\begin{equation}
\begin{pmatrix}p_n&p_{n-1}\\q_n&q_{n-1}\end{pmatrix}=\begin{pmatrix}0&1\\1&0\end{pmatrix}\begin{pmatrix}a_1&1\\1&0\end{pmatrix}
\cdots \begin{pmatrix}a_n&1\\1&0\end{pmatrix}.
\label{matrixformula}
\end{equation}
We have
\[
p_1=1, \qquad q_1=a_1,
\]
 and from \eqref{matrixformula} we get for all $n \geq 1$ that
\[
p_{n+1}=a_{n+1}p_n+p_{n-1},
\]
and
\[
q_{n+1}=a_{n+1}q_n+q_{n-1}.
\]
Since the $a_n$ are positive integers we get by induction that for all $n \geq 1$,
\begin{equation}
p_n \geq 2^{(n-1)/2}, \qquad q_n \geq 2^{(n-1)/2}.
\label{qnbound}
\end{equation}
In fact, setting $F_1=1, F_2=1$, $F_{n+1}=F_n+F_{n-1}$ for $n \geq 2$, with $F_n$ the \textbf{$n$th
Fibonacci number},
as $a_n \geq 1$ we check by induction that
\[
p_n \geq F_n,\qquad q_n \geq F_{n+1}.
\] 
Taking determinants of \eqref{matrixformula} gives us for all $n \geq 1$ that
\begin{equation}
p_nq_{n-1}-p_{n-1}q_n=(-1)^{n+1},
\label{determinants}
\end{equation}
and then by induction we have for all $n \geq 1$,
\[
\frac{p_n}{q_n}=\sum_{k=1}^n \frac{(-1)^{k+1}}{q_{k-1}q_k}.
\]
For any $a \in \mathbb{N}^\mathbb{N}$, as $n \to \infty$ this sequence of sums converges and we denote its limit by $v(a)$.
We have for all $n \geq 1$,
\[
v(a)-\frac{p_n}{q_n}=\sum_{k=n+1}^\infty  \frac{(-1)^{k+1}}{q_{k-1}q_k}.
\]
Since the right hand side is an alternating series we obtain for $n \geq 1$,
\begin{equation}
\left|  v(a)-\frac{p_n}{q_n} \right| < \frac{1}{q_n q_{n+1}},
\label{upperbound}
\end{equation}
and 
\begin{equation}
\left|  v(a)-\frac{p_n}{q_n} \right| > \frac{1}{q_nq_{n+1}}-\frac{1}{q_{n+1}q_{n+2}}=\frac{a_{n+2}}{q_n q_{n+2}} \geq \frac{1}{q_n(q_{n+1}+q_n)},
\label{lowerbound}
\end{equation}
and
\begin{equation}
\frac{p_2}{q_2}<\frac{p_4}{q_4}<\cdots<v(a)<\cdots<\frac{p_3}{q_3}<\frac{p_1}{q_1}.
\label{signs}
\end{equation}

For $x \in \Omega$, we say that $\frac{p}{q} \in \mathbb{Q}$, $q>0$, is a \textbf{best approximation to $x$}
if $\norm{qx} = |qx-p|$ and $\norm{q' x} > \norm{qx}$ for $1 \leq q'<q$.
The following theorem shows in particular that the convergents of a continued fraction $a$
are  best approximations to $v(a)$  \cite[p.~22, Chapter 2, \S 3, Theorem 1]{rockett}.

\begin{theorem}[Best approximations]
Let $a \in \mathbb{N}^{\mathbb{N}}$. For any $n \geq 1$,
\[
\norm{q_n v(a)} = |p_n-q_n v(a)|.
\]
If $1 \leq q < q_{n+1}$, then for any $p \in \mathbb{Z}$,
\[
|qv(a)-  p| \geq |p_n-q_n v(a)|.
\]
\label{bestapproximation}
\end{theorem}
\begin{proof}
By \eqref{upperbound}, $|q_n v(a) - p_n| < \frac{1}{q_{n+1}}$. 
But
\[
q_{n+1} \geq q_2 = a_2q_1+q_0 \geq q_1 + q_0 = a_1 + 1 \geq 2.
\]
So $|q_n v(a) - p_n| < \frac{1}{2}$, which means 
$\norm{q_n v(a)} = |q_n v(a) - p_n|$.  

Write $x=v(a)$ and
let $A = \begin{pmatrix} p_{n+1}&p_n\\ q_{n+1}&q_n \end{pmatrix}$. Applying
\eqref{determinants},
\[
\det A = p_{n+1}q_n - p_n q_{n+1} = (-1)^n. 
\] 
Let 
\[
\begin{pmatrix}\mu \\ \nu \end{pmatrix} = A^{-1} \begin{pmatrix} p\\ q \end{pmatrix}
=\frac{1}{\det A} \begin{pmatrix} q_n&-p_n\\ -q_{n+1}&p_{n+1} \end{pmatrix}  \begin{pmatrix} p\\ q \end{pmatrix}
=(-1)^n  \begin{pmatrix} q_n&-p_n\\ -q_{n+1}&p_{n+1} \end{pmatrix}  \begin{pmatrix} p\\ q \end{pmatrix}.
\]
Then
\[
q x - p = (\mu q_{n+1} + \nu q_n)x - (p_{n+1} \mu + p_n \nu)
=\mu (q_{n+1} x - p_{n+1}) + \nu(q_n x - p_n)
\]
and
\[
\begin{pmatrix}p\\ q\end{pmatrix}
=A \begin{pmatrix}\mu\\ \nu\end{pmatrix}
= \begin{pmatrix} p_{n+1}&p_n\\ q_{n+1}&q_n \end{pmatrix}\begin{pmatrix}\mu\\ \nu\end{pmatrix}
=\begin{pmatrix}p_{n+1} \mu + p_n \nu\\ q_{n+1}\mu+q_n \nu\end{pmatrix}.
\]
In particular, $\mu,\nu \in \mathbb{Z}$.
Suppose by contradiction that $\nu=0$. Then $(q-\mu q_{n+1})x = p-\mu p_{n+1}$, and as $x \not \in \mathbb{Q}$ it must then be that 
$q=\mu q_{n+1}$ and $p=\mu p_{n+1}$. But $q=\mu q_{n+1}$, $1 \leq q < q_{n+1}$, and $\mu \in \mathbb{Z}$ are together a contradiction. Therefore
$\nu \neq 0$. 
Either $\mu=0$ or $\mu \neq 0$. For $\mu =0$,
\[
|qx-p| = |\nu| |q_nx-p_n| \geq |q_nx-p_n|,
\]
which is the claim. For $\mu \neq 0$ we use the fact
$q = q_{n+1} \mu + q_n \nu$ and $1 \leq q < q_{n+1}$. If $\mu,\nu>0$ then $q<q_{n+1}$ is contradicted,
and if $\mu,\nu<0$ then $q \geq 1$ is contradicted. Therefore $\mu$ and $\nu$ have different signs, say $\mu=(-1)^N |\mu|$ and 
$\nu=(-1)^{N+1} |\nu|$.
Furthermore, we get from \eqref{signs} that
\[
\sgn (q_n x - p_n) = (-1)^n,
\qquad \sgn (q_{n+1} x - p_{n+1}) = (-1)^{n+1}.
\]
Therefore
\begin{align*}
qx-p&=\mu (q_{n+1} x - p_{n+1}) + \nu(q_n x - p_n)\\
&=(-1)^N |\mu| \cdot (-1)^{n+1} |q_{n+1}x-p_{n+1}|
+ (-1)^{N+1} |\nu| \cdot (-1)^n |q_nx-p_n|,
\end{align*}
hence
\[
|qx-p| =  |\mu| |q_{n+1}x-p_{n+1}| +  |\nu| |q_nx-p_n| \geq
  |\nu| |q_nx-p_n| \geq |q_nx-p_n|,
\]
which is the claim.
\end{proof}

The above theorem says, a fortiori, that the convergents of $a$ are best approximations to $v(a)$. It can also be proved
that if $\frac{p}{q} \in \mathbb{Q}$, $q>0$, is a best approximation to $v(a)$ then $\frac{p}{q}$ is a convergent of $a$ \cite[p.~9, Theorem 6]{MR0209227}.
Cassels \cite[p.~2, Chapter I]{cassels} works out the theory of continued 
fractions according to this point of view. Similarly, Milnor \cite[p.~234, Appendix C]{milnor} works out the theory of continued fractions in the language of rotations of the unit circle.

We define the \textbf{Gauss transformation} $T:\Omega \to \Omega$ by
 $T(x)=R\left(\frac{1}{x}\right)$ for $x \in \Omega$, and we define $\Phi:\Omega \to \mathbb{N}^\mathbb{N}$ by
\[
(\Phi(x))_n=\left[ \frac{1}{T^{n-1}(x)} \right],\qquad n \geq 1.
\]
One can check that if $a \in \mathbb{N}^\mathbb{N}$ then $v(a) \in \Omega$ \cite[p.~73, Lemma 3.2]{einsiedler}. (Namely, the value of a nonterminating continued fraction is irrational.)
One can prove that $v:\mathbb{N}^\mathbb{N} \to \Omega$ is injective \cite[p.~75, Lemma 3.4]{einsiedler}, and
 for $x \in \Omega$ that \cite[p.~78, Lemma 3.6]{einsiedler}
 \[
 (v \circ \Phi)(x)=x.
 \]
 Therefore $\Phi:\Omega \to \mathbb{N}^\mathbb{N}$ is a bijection. Moreover,
 $\Phi$ is a homeomorphism, when $\mathbb{N}$ has discrete topology, $\mathbb{N}^\mathbb{N}$ has the product topology, and $\Omega$ has the subspace topology
 inherited from $\mathbb{R}$ \cite[p.~86, Exercise 3.2.2]{einsiedler}. That $\mathbb{N}^\mathbb{N}$ and $\Omega$ are homeomorphic
 can also be proved without using continued fractions \cite[p.~106, Theorem 3.68]{hitchhiker}.
 In descriptive set theory,
 the topological space $\mathscr{N}=\mathbb{N}^\mathbb{N}$
 is  called the \textbf{Baire space}, and the Alexandrov-Urysohn theorem states that $\mathscr{N}$ has the universal property  that any nonempty
 Polish space that is zero-dimensional (there is a basis of clopen sets for the topology) and all of whose compact subsets have empty
 interior is homeomorphic to $\mathscr{N}$  \cite[p.~37, Theorem 7.7]{kechris}.
 Some of Baire's work on $\mathscr{N}$ is described in \cite[pp.~119--120]{johnson} and \cite[pp.~349, 372]{arboleda}.

For $I=[0,1]$ and for $T:I \to I$, $T(x)=R(1/x)$ for $x>0$ and $T(0)=0$. For $k \geq 1$ let 
$I_k = \left( \frac{1}{k+1},\frac{1}{k} \right)$, so if $x \in I_k$ then $T(x) = x^{-1}-k$. Then for
$x \in I=[0,1]$,
\[
T(x) = \sum_{k=1}^\infty 1_{I_k}(x) (x^{-1}-k).
\]
For $S = \{0\} \cup \{k^{-1} : k \geq 1\}$, $I \setminus S = \bigcup_{k \geq 1} I_k$, and for $x \in I \setminus S$,
\[
T'(x) = - \sum_{k=1}^\infty 1_{I_k}(x) x^{-2},
\]
and for $x \in I_k$, $k^2 < |T'(x)| < (k+1)^2$. Differentiability and dynamical properties of the Gauss transformation are worked out
by Cornfeld, Fomin and Sinai \cite[pp.~165--177, Chapter 7, \S 4]{cornfeld},
as an instance of \textbf{piecewise monotonic} transformations.

For each $n \geq 1$ we define $a_n:\Omega \to \mathbb{N}$ by $a_n(x)=(\Phi(x))_n$.
For example, $e-2 \in \Omega$, and it is known \cite[p.~74, Theorem 2]{MR0209227} 
that for $k \geq 1$,
\[
a_{3k}(e-2)=a_{3k-2}(e-2)=1 \quad \textrm{and} \quad a_{3k-1}(e-2)=2k.
\]
The pattern for the continued fraction expansion of $e$ seems first to have been worked out by
Roger Cotes in 1714 \cite{fowler1992}, and was  later proved by Euler using a method involving the Riccati equation \cite{cretney}.

For $n \geq 1$ and $i \in \mathbb{N}^n$, 
let 
\[
I_n(i) = \{\omega \in \Omega: a_k(x) = i_k, 1 \leq k \leq n\}.
\]
For $x \in I_n(i)$, 
\[
\frac{p_n(x)}{q_n(x)} = [i_1,\ldots,i_n],
\qquad \frac{p_{n-1}(x)}{q_{n-1}(x)} = [i_1,\ldots,i_{n-1}].
\]
The following is an expression for the sets $I_n(i)$ \cite[p.~18, Theorem 1.2.2]{iosifescu}.

\begin{theorem}
Let $n \geq 1$, $i \in \mathbb{Z}_{\geq 1}^n$,
and for $p_n=p_n(x)$, $q_n=q_n(x)$, $x \in I_n(i)$,
\[
u_n(i) = \begin{cases}
\frac{p_n+p_{n-1}}{q_n+q_{n-1}}&\textrm{$n$ odd}\\
\frac{p_n}{q_n}&\textrm{$n$ even}
\end{cases}
\]
and
\[
v_n(i) = \begin{cases}
\frac{p_n}{q_n}&\textrm{$n$ odd}\\
\frac{p_n+p_{n-1}}{q_n+q_{n-1}}&\textrm{$n$ even}.
\end{cases}
\]
Then
\[
I_n(i) = \Omega \cap (u_n(i),v_n(i)).
\]
\end{theorem}

It follows from the above that for $i \in \mathbb{N}^n$, $n \geq 1$,
\[
\mu(I_n(i)) = \frac{1}{q_n(q_n+q_{n-1})}.
\]

\section{Diophantine conditions}
For real $\tau,\gamma>0$ let
\begin{align*}
\mathcal{D}(\tau,\gamma) &=\bigcap_{q \in \mathbb{Z}_{\geq 1}, p \in \mathbb{Z}} \left\{x \in [0,1]: \left|x-\frac{p}{q}\right| \geq \gamma q^{-\tau} \right\}\\
&=\bigcap_{q \in \mathbb{Z}_{\geq 1}} \left\{x \in [0,1]: \norm{qx} \geq \gamma q^{-\tau+1} \right\},
\end{align*}
and let
\[
\mathcal{D}(\tau) = \bigcup_{\gamma>0} \mathcal{D}(\tau,\gamma).
\]

We relate the sets $\mathcal{D}(\tau)$ and continued fractions expansions  \cite[p.~241, Lemma C.6]{milnor}, cf.  \cite[p.~130, Proposition 2.4]{yoccoz}.

\begin{lemma}
For $\tau>0$ and $x \in \Omega$, $x \in \mathcal{D}(\tau)$ if and only if there is some
$C=C(x)>0$ such that $q_{n+1}(x) \leq C q_n(x)^{\tau-1}$ for all $n \geq 1$.
\label{Dtau}
\end{lemma}
\begin{proof}
For $x \in \Omega$, write $q_n=q_n(x)$.
By \eqref{lowerbound}, $\norm{q_n x} > \frac{1}{q_{n+1}+q_n}$, 
and by \eqref{upperbound}, $\norm{q_n x} < \frac{1}{q_{n+1}}$.
Suppose $x \in \mathcal{D}(\tau)$, so there is some $\gamma>0$ such that $x \in \mathcal{D}(\tau,\gamma)$. 
Then
\[
q_{n+1} < \frac{1}{\norm{q_n x}} \leq \gamma^{-1} q_n^{\tau-1}.
\]

Suppose $q_{n+1} \leq C q_n^{\tau-1}$ for all $n \geq 1$. For $q \in \mathbb{Z}_{\geq 1}$, take $q_n \leq q < q_{n+1}$.
Using Theorem \ref{bestapproximation},
\[
\norm{qx} \geq \norm{q_n x} > \frac{1}{q_{n+1}+q_n} > \frac{1}{2q_{n+1}} \geq \frac{1}{2} C^{-1} q_n^{-\tau+1}
\geq \frac{1}{2C} \cdot q^{-\tau+1},
\]
which means that $x \in \mathcal{D}(\tau,\frac{1}{2C})$. 
\end{proof}

For $K$ a positive integer, let 
\[
\mathcal{B}_K = \{x \in \Omega: \textrm{$a_n(x) \leq K$ for all $n \geq 1$}\},
\]
so $\mathcal{B}=\bigcup_{K \geq 1} \mathcal{B}_K$ is the set of those $x \in \Omega$ with bounded partial quotients.

\begin{lemma}
For $x \in \Omega$, $x \in \mathcal{D}(2)$ if and only if $x \in \mathcal{B}$.
\end{lemma}
\begin{proof}
Write $a_n=a_n(x)$ and $q_n=q_n(x)$.
If $x \in \mathcal{D}(2)$ then there is some $\gamma>0$ such that $x \in \mathcal{D}(2,\gamma)$, hence for $n \geq 1$,
\[
q_{n+1}<\frac{1}{\norm{q_n x}} \leq \gamma^{-1} q_n.
\]
Now, $q_{n+1} = a_{n+1} q_n + q_{n-1}$ for $n \geq 1$, so
\[
a_{n+1}  < q_{n+1} q_n^{-1}  < q_n^{-1} \cdot \gamma^{-1} q_n
=\gamma^{-1},
\]
which shows that $x$ has bounded partial quotients.

If $x \in \mathcal{B}_K$, let $q \in \mathbb{Z}_{\geq 1}$ and let $q_n \leq q < q_{n+1}$. 
Using Theorem \ref{bestapproximation} and then \eqref{lowerbound},
\[
\norm{qx} \geq \norm{q_n x} > \frac{1}{q_{n+1}+q_n}
=\frac{1}{a_{n+1} q_n + q_{n-1} + q_n}
>\frac{1}{(a_{n+1}+2)q_n}.
\]
As $x \in \mathcal{B}_K$,
\[
\norm{qx} > \frac{1}{(K+2) q_n} \geq \frac{1}{K+2} q^{-1},
\]
which means that $x \in \mathcal{D}(2,\frac{1}{K+2})$. 
\end{proof}

A complex number $\alpha$ is called an \textbf{algebraic number of degree $d$}, $d \geq 0$, if there is some polynomial $f \in \mathbb{Z}[x]$ with degree $d$ such that 
$f(\alpha)=0$ and if $g \in \mathbb{Z}[x]$ has degree $<d$ and $g(\alpha)=0$ then $g=0$. 
An algebraic number number of degree $2$ is called a \textbf{quadratic irrational}. 
Let $\alpha \in \Omega$. It was proved by Euler \cite[p.~144, Theorem 176]{wright} that if there is some $p>0$ and some $L$ such that $a_{l+p}(\alpha)=a_l(\alpha)$ for all $l \geq L$ then $\alpha$ is a quadratic irrational The converse of this was proved by Lagrange \cite[p.~144, Theorem 177]{wright},
namely that a quadratic irrational has eventually periodic 
partial quotients.
For example, $\alpha=\sqrt{11}-3 \in \Omega$ is a quadratic irrational, being a root of 
$x^2+6x-2$,
 and one works out
that $a_1(\alpha)=3, a_2(\alpha)=6$, and that $a_{l+2}(\alpha)=a_l(\alpha)$ for $l \geq 1$. 
In particular, if $\alpha \in \Omega$ is a quadratic irrational, then $\alpha$ has bounded partial quotients.

Liouville \cite[p.~161, Theorem 191]{wright} proved that if $x \in \Omega$ is an algebraic number of degree $d \geq 2$,
then $x \in \mathcal{D}(d)$. 
The \textbf{Thue-Siegel-Roth theorem} \cite[p.~55, Theorem 1.23]{feldman} states that if $x \in \Omega$ is an algebraic number,
then for any
$\delta>0$ there  is some $q_\delta \in \mathbb{Z}_{\geq 1}$ such that for all $q \geq q_\delta$,
\[
\norm{qx} \geq q^{-1-\delta}.
\]
See Schmidt  \cite[p.~195, Theorem 2B]{schmidt}.

\section{Sums of reciprocals}
\label{reciprocalsection}
We are interested in getting bounds on the sum $\sum_{j=1}^m \frac{1}{\norm{jx}}$.
This is an appealing question because the terms $\frac{1}{\norm{jx}}$ are unbounded.

Rather than merely stating that $\sum_{k=1}^\infty \frac{1}{k}=\infty$, we give more information by giving the
 estimate 
 \[
 \sum_{k=1}^n \frac{1}{k} = \log n + \gamma + O(n^{-1}),
 \]
 where $\gamma$ is Euler's constant. Likewise, rather than merely stating that there are infinitely many primes, we state more
 information with \cite[p.~102, \S 28]{handbuch}
 \[
 \sum_{p \leq x} \frac{1}{p} = \log \log x  + B + O\left( \frac{1}{\log x} \right),
 \]
 for a certain constant $B$ (namely ``Merten's constant''),
 or with \cite[p.~226, \S 61]{handbuch}
 \[
 \sum_{p \leq x} p = \frac{x^2}{2\log x}+O\left(\frac{x^2}{(\log x)^2}\right).
 \]

Because $|\sin (\pi x)| = \sin(\pi \norm{x}) 
\leq \pi \norm{x}$ and
\[
|\sin (\pi x)| = \sin(\pi \norm{x}) \geq \frac{2}{\pi} \pi \norm{x}=2\norm{x},
\]
we have
\begin{equation}
\frac{1}{\pi} \sum_{j=1}^m \frac{1}{\norm{jx}} 
\leq \sum_{j=1}^m \frac{1}{|\sin \pi  j x |} 
\leq \frac{1}{2} \sum_{j=1}^m  \frac{1}{\norm{jx}} .
\label{sinebound}
\end{equation}
Thus, getting bounds on $\sum_{j=1}^m \frac{1}{\norm{jx}}$ will give us bounds on $\sum_{j=1}^m \frac{1}{|\sin \pi  j x|}$.

Let $\psi$ be a nondecreasing function defined on the positive integers such that $\psi(h)>0$ for $h  \geq 1$ (for example, $\psi(h)=\log(2h)$). Following Kuipers and Niederreiter \cite[p.~121, Definition 3.3]{kuipers}, we say that 
an irrational number $x$ is \textbf{of type $<\psi$} if $\norm{hx} \geq \frac{1}{h\psi(h)}$ for all  integers $h \geq 1$.
If $\psi$ is a constant function, then we say that
\textbf{$x$ is of constant type}.

\begin{lemma}
$x \in \Omega$ is of constant type if and only if it has bounded partial quotients.
\label{constanttype}
\end{lemma}
\begin{proof}
Suppose that $x \in \Omega$ is of constant type. So there is some $K>0$ such that $\norm{hx} \geq \frac{1}{hK}$ for all integers $h \geq 1$. For $n \geq 2$
we have $q_n=a_n q_{n-1}+q_{n-2}$, and hence, by \eqref{upperbound},
\[
a_n=\frac{q_n}{q_{n-1}}-\frac{q_{n-2}}{q_{n-1}}
<\frac{q_n}{q_{n-1}}
\leq q_n \cdot K\norm{q_{n-1}x}
< K,
\]
showing that $x$ has bounded partial quotients.

Suppose that $x \in \Omega$ has bounded partial quotients, say $a_n \leq K$ for all $n \geq 1$. Let $h$ be a positive integer and take $q_n \leq h
< q_{n+1}$. Then first by Theorem \ref{bestapproximation} and then by \eqref{lowerbound},
\[
\norm{hx} \geq \norm{q_n x} > \frac{1}{q_n+q_{n+1}} > \frac{1}{2q_{n+1}}
=\frac{1}{2(a_{n+1}q_n+q_{n-1})}
>\frac{1}{2(a_{n+1}q_n+q_n)},
\]
and so
\[
\norm{hx} > \frac{1}{2(K+1)} \frac{1}{q_n} \geq \frac{1}{2(K+1)}  \frac{1}{h},
\]
showing that $x$ is of constant type.
\end{proof}

However, almost all $x$ do not have bounded partial quotients \cite[p.~60, Theorem 29]{MR1451873}. Shallit \cite{MR1175525} gives a survey on numbers with bounded partial quotients.

We  state and prove a result of Khinchin's \cite[p.~69, Theorem 32]{MR1451873}  that we  then use. 

\begin{lemma}
Let $f$ be a positive function on the positive integers. If
\[
\sum_{q=1}^\infty f(q)<\infty,
\]
then for almost all $x \in \Omega$ there are only finitely many $q$ such that $\norm{qx}<f(q)$.
\label{khinchin}
\end{lemma}
\begin{proof}
For each positive integer $q$, let $E_q=\{t\in \Omega: \norm{qt} < f(q)\}$. If $t \in E_q$, then there is some integer $p$ with $0 \leq p  \leq q$ such that $\left| t-\frac{p}{q} \right| < \frac{f(q)}{q}$. It follows that
\[
E_q \subset \left(0,\frac{f(q)}{q} \right) \cup \left(1-\frac{f(q)}{q}, 1 \right)  \cup \bigcup_{p=1}^{q-1} \left(\frac{p}{q}-\frac{f(q)}{q},\frac{p}{q}+\frac{f(q)}{q} \right).
\]
Therefore
\[
\sum_{q=1}^\infty \mu(E_q)  \leq  \sum_{q=1}^\infty 
2q \cdot \frac{f(q)}{q}=2\sum_{q=1}^\infty f(q) < \infty.
\]
Let $E=\limsup_{q \to \infty} E_q$, i.e. $E=\{t \in \Omega : t \in E_q \, \textrm{for infinitely many $q$}\}$. Then by the Borel-Cantelli lemma \cite[p.~59, Theorem 4.3]{billingsley} we have that $\mu(E)=0$. Therefore, for almost all $t \in \Omega$ there are only finitely many $q$ such that $t \in E_q$, i.e., for almost all $t \in \Omega$ there are only finitely many $q$ such that $\norm{qt} < f(q)$.
\end{proof}

The above lemma is  proved in
Benedetto and Czaja \cite[p.~183, Theorem 4.3.3]{benedetto} using
the fact that a function of bounded variation is differentiable almost everywhere. We outline the proof. Define $F:[0,1] \to \mathbb{R}_{>0}$ by 
$F(x) = \frac{f(q)}{q}$ if $x=\frac{a}{q}$, $\gcd(a,q)=1$, $0 \leq a \leq q$, and $F(x)=0$ if $x \in \Omega$. 
Writing $\mathscr{A}_q = \left\{\frac{a}{q}: 0 \leq a \leq q, \gcd(a,q)=1\right\}$,
for $0=t_0<t_1<\cdots<t_N=1$,
\[
\sum_{j=0}^N F(t_j) =   \sum_{q=1}^\infty \sum_{j=0}^N F(t_j) \cdot 1_{\mathscr{A}_q}(t_j) = 
\sum_{q=1}^\infty \frac{f(q)}{q} \sum_{j=0}^N 1_{\mathscr{A}_q}(t_j) \leq \sum_{q=1}^\infty f(q).
\]
It follows that the total variation of $F$ is
$\leq 2 \sum_{q=1}^\infty f(q)$ and hence $F$ is a function of \textbf{bounded variation}. Because $F$ has bounded variation, the set $D_F$ of
$x \in [0,1]$ at which $F$ is differentiable is a Borel set with $\lambda(D_F)=1$. Check that $F'(x)=0$ for $x \in D_F \setminus \mathbb{Q}$, and using this,
if
$\frac{a_n}{q_n} \to x$ with $\gcd(a_n,q_n)=1$ and $0 \leq a_n \leq q_n$ then for some $N$, if $n \geq N$ then
$\left|x-\frac{a_n}{q} \right| \geq \frac{F(q_n)}{q_n}$.

We use the above lemma to prove the following result.

\begin{lemma}
Let $\epsilon>0$. For almost all $x \in \Omega$, there is some $K>0$ such that $x$ is of type $<K(\log h)^{1+\epsilon}$.
\label{aatype}
\end{lemma}
\begin{proof}
Let
\[
E =\left\{t \in \Omega: \norm{ht} < \frac{1}{h(\log h)^{1+\epsilon}} \, \textrm{for infinitely many $h$}\right\}.
\]
Since $\sum_{h=1}^\infty \frac{1}{h(\log h)^{1+\epsilon}}$ converges, we have by Lemma \ref{khinchin} that
$\mu(E)=0$. 
Let $t \in \Omega \setminus E$.  Then $\norm{ht} \geq \frac{1}{h (\log h)^{1+\epsilon}}$ for all sufficiently large $h$.
It follows that there is some $K$ such that $t$ is of type $<K(\log h)^{1+\epsilon}$.
\end{proof}

The following technical lemma is from Kuipers and Niederreiter \cite[p.~130, Exercise 3.9]{kuipers}; cf. Lang \cite[p.~39, Lemma]{MR0209227}.

\begin{lemma}
Let $x \in \Omega$ be of type $< \psi$. If $n \geq 0$ and if $0 \leq h_0 < q_{n+1}$, then
\[
\sum_{\stackrel{1 \leq j \leq q_n}{j+h_0<q_{n+1}}} \frac{1}{\norm{(j+h_0)x}} < 6 q_n(\psi(q_n)+\log q_n).
\]
\label{h0}
\end{lemma}
\begin{proof}
Since $p_n$ and $q_n$ are relatively prime, 
the remainders of $j p_n$, $j=1,\ldots,q_n$, when divided by $q_n$ are all distinct.
Then also, the remainders of $j p_n+h_0 p_n$, $j=1,\ldots,q_n$, when divided by $q_n$ are all distinct. Let $\lambda_j$, $j=1,\ldots,q_n$, be the remainder of $jp_n+
h_0p_n$  when divided by $q_n$. We have $\{\lambda_j: 1 \leq j \leq q_n\}=\{0,\ldots,q_n-1\}$; let $\lambda_{j_1}=0$, $\lambda_{j_2}=1$, and $\lambda_{j_3}=q_n-1$.

Write $x=\frac{p_n}{q_n}+\frac{\delta_n}{q_n q_{n+1}}$; by \eqref{upperbound} we have $|\delta_n|<1$. If $j+h_0<q_{n+1}$ then by
Theorem \ref{bestapproximation} we have $\norm{(j+h_0)x} \geq \norm{q_n x}$, 
and since $x$ is of type $<\psi$ we have
\[
\norm{(j+h_0)x} \geq \norm{q_n x} \geq \frac{1}{q_n \psi(q_n)}.
\] 
Let, for $i=1,2,3$,
\[
A_i=\begin{cases}
1,&\textrm{if $j_i+h_0<q_{n+1}$},\\
0,&\textrm{if $j_i+h_0 \geq q_{n+1}$}.
\end{cases}
\]
If $j+h_0<q_{n+1}$ and $j \neq j_1,j_2,j_3$, then 
\[\norm{\frac{\lambda_j}{q_n}+\frac{(j+h_0)\delta_n}{q_n q_{n+1}}} \geq \min \left\{ \norm{\frac{\lambda_j}{q_n}+\frac{1}{q_n}},\norm{\frac{\lambda_j}{q_n}-\frac{1}{q_n}}  \right\}.
\]
It follows that
\begin{eqnarray*}
\sum_{\stackrel{1 \leq j \leq q_n}{j+h_0<q_{n+1}}} \frac{1}{\norm{(j+h_0)x}}
&=&
\frac{A_1}{\norm{(j_1+h_0)x}}
+
\frac{A_2}{\norm{(j_2+h_0)x}}
+\frac{A_3}{\norm{(j_3+h_0)x}}\\
&&+
\sum_{\stackrel{1 \leq j \leq q_n}{\stackrel{j+h_0<q_{n+1}}{j \neq j_1, j_2, j_3}}} \frac{1}{\norm{\frac{\lambda_j}{q_n}+\frac{(j+h_0)\delta_n}{q_n q_{n+1}}}}\\
&\leq&3q_n\psi(q_n)+\sum_{\stackrel{1 \leq j \leq q_n}{\stackrel{j+h_0<q_{n+1}}{j \neq j_1, j_2, j_3}}} \frac{1}{\norm{\frac{\lambda_j}{q_n}+\frac{(j+h_0)\delta_n}{q_n q_{n+1}}}}\\
&\leq&3q_n\psi(q_n)+\sum_{\stackrel{1 \leq j \leq q_n}{\stackrel{j+h_0<q_{n+1}}{j \neq j_1, j_2, j_3}}} \frac{1}{\norm{\frac{\lambda_j}{q_n}+\frac{1}{q_n}}}\\
&&+\sum_{\stackrel{1 \leq j \leq q_n}{\stackrel{j+h_0<q_{n+1}}{j \neq j_1, j_2, j_3}}} \frac{1}{\norm{\frac{\lambda_j}{q_n}-\frac{1}{q_n}}}\\
&<&3q_n\psi(q_n)+2\sum_{k=1}^{q_n-1} \frac{1}{\norm{\frac{k}{q_n}}}.
\end{eqnarray*}

But $\frac{1}{\norm{y}}<\frac{1}{R(y)}+\frac{1}{1-R(y)}$ for $y \not \in \mathbb{Z}$, so
\begin{eqnarray*}
\sum_{k=1}^{q_n-1} \frac{1}{\norm{\frac{k}{q_n}}}&<&\sum_{k=1}^{q_n-1} \frac{1}{R\left( \frac{k}{q_n} \right)}
+\frac{1}{1-R\left( \frac{k}{q_n} \right)}\\
&=&\sum_{k=1}^{q_n-1} \frac{1}{\frac{k}{q_n}}+\frac{1}{1-\frac{k}{q_n}}\\
&=&2q_n \sum_{k=1}^{q_n-1} \frac{1}{k}\\
&<&3 q_n \log q_n;
\end{eqnarray*}
the last inequality is because, for all $m \geq 1$,
\[
\sum_{k=1}^m \frac{1}{k} < \frac{3}{2}\log(m+1).
\]
\end{proof}

We now use Lemma \ref{h0} to obtain a bound on $\sum_{j=1}^m \frac{1}{\norm{jx}}$  in terms of the type of $x$. This is from Kuipers and Niederreiter \cite[p.~131, Exercise 3.11]{kuipers};
cf. Lang \cite[p.~39, Theorem 2]{MR0209227}.

\begin{theorem}
If $x \in \Omega$ is of type $<\psi$, then for all $m \geq 1$ we have
\[
\sum_{j=1}^m \frac{1}{\norm{jx}}<12m(\psi(m)+\log m).
\]
\label{jalpha}
\end{theorem}
\begin{proof}
We shall prove the claim by induction. Because $x$ is of type $<\psi$, we have
\[
\frac{1}{\norm{x}} \leq \psi(1)<12\psi(1),
\]
so the claim is true for $m=1$. Take $m>1$ and assume that the claim is true for all $1 \leq m'<m$. We shall show that it is true for $m$.

Let $q_n \leq m < q_{n+1}$. Either $m < 2q_n$ or $m \geq 2q_n$.
In the first case, using Lemma \ref{h0}  we have
\begin{eqnarray*}
\sum_{j=1}^m \frac{1}{\norm{jx}}&=&\sum_{j=1}^{q_n} \frac{1}{\norm{jx}}+\sum_{j=1}^{m-q_n} \frac{1}{\norm{(j+q_n)x}}\\
&<&12 q_n(\psi(q_n)+\log q_n)\\
&<&12m(\psi(m)+\log m).
\end{eqnarray*}

In the second case, using the induction assumption (with $m'=m-q_n$) and Lemma \ref{h0} we have, because $q_n<m-q_n$,
\begin{eqnarray*}
\sum_{j=1}^m \frac{1}{\norm{jx}}&=&\sum_{j=1}^{m-q_n} \frac{1}{\norm{jx}}+\sum_{j=1}^{q_n} \frac{1}{\norm{(j +m-q_n)x}}\\
&<&12(m-q_n)\left(\psi(m-q_n)+\log(m-q_n)\right)+6q_n(\psi(q_n)+\log q_n)\\
&<&12(m-q_n)\left(\psi(m)+\log m \right)+12q_n(\psi(q_n)+\log q_n)\\
&=&12m(\psi(m)+\log m)-12q_n\left(\psi(m)-\psi(q_n)+\log m - \log q_n\right)\\
&\leq&12m(\psi(m)+\log m).
\end{eqnarray*}

The claim is true in both cases, which completes the proof by induction.
\end{proof}

We can now establish for almost all $x \in \Omega$ a tractable upper bound on the sum $\sum_{j=1}^m \frac{1}{\norm{jx}}$, and thus
by \eqref{sinebound} also on $\sum_{j=1}^m \frac{1}{|\sin \pi  j x|}$; cf. Lang \cite[p.~44, Theorem 3]{MR0209227}.

\begin{theorem}
Let $\epsilon>0$. For almost all $x \in \Omega$, we have
\[
\sum_{j=1}^m \frac{1}{\norm{jx}} = O\left(m \left(\log m\right)^{1+\epsilon}\right),
\]
while if $x$ has bounded partial quotients then
\[
\sum_{j=1}^m \frac{1}{\norm{jx}} = O\left(m \log m \right).
\]
\label{reciprocaltheorem}
\end{theorem}
\begin{proof}
Let $\epsilon>0$. By Lemma \ref{aatype}, for almost all $x \in \Omega$ there is some $K$ such that $x$ is of type $<K(\log h)^{1+\epsilon}$.
For
such an $x$, it follows from Theorem \ref{jalpha} that for all $m \geq 1$,
\[
\sum_{j=1}^m \frac{1}{\norm{jx}} < 12m(K(\log m)^{1+\epsilon}+\log m)=O\left(m \left(\log m\right)^{1+\epsilon}\right).
\]

If $x$ has bounded partial quotients then by Lemma \ref{constanttype} it is of constant type, say $\psi(m)=K$ for some $K$. It follows from Theorem \ref{jalpha} that for all $m \geq 1$,
\[
\sum_{j=1}^m \frac{1}{\norm{jx}} < 12m(K+\log m)=O\left(m \log m \right).
\]
\end{proof}

For example, take $x=\frac{-1+\sqrt{5}}{2}$, for which $a_n(x)=1$ for all $n \in \mathbb{N}$, and so in particular $x$ has bounded partial quotients.
In Figure \ref{m5000_40000_g} we plot $\frac{1}{m \log m} \cdot \sum_{j=1}^m \frac{1}{\norm{jx}}$ for $m=5000,\ldots,40000$.  
 These computations suggest that there is some constant $C$ for which $\sum_{j=1}^m \frac{1}{\norm{jx}} > C m \log m$ for all $m$.
 In Theorem \ref{omegabound} we shall prove that for almost all $x \in \Omega$ there is such a $C(x)$.

\begin{figure}
\includegraphics[width=\textwidth]{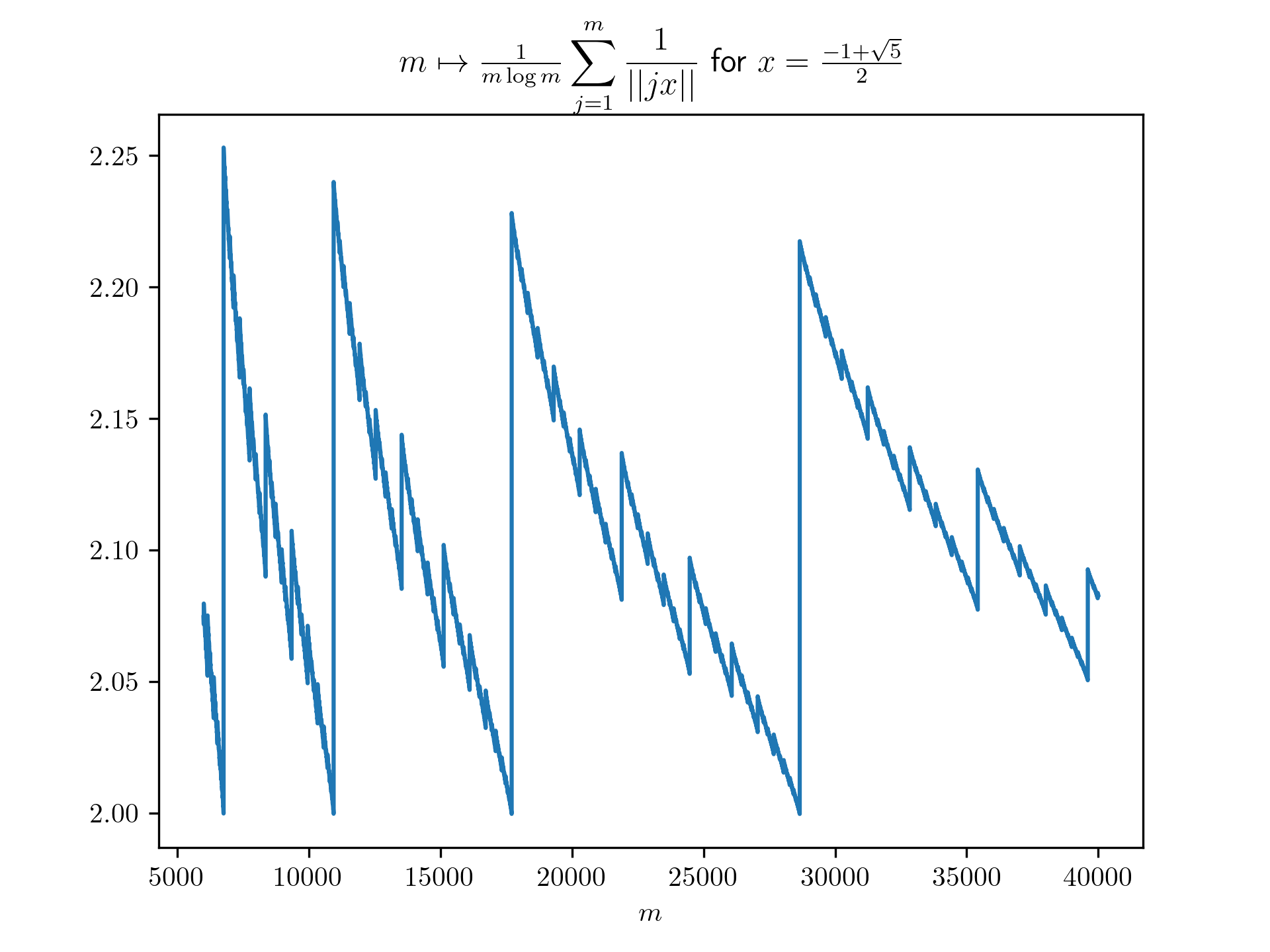}
\caption{$\frac{1}{m\log m} \sum_{j=1}^m \frac{1}{\norm{jx}}$ for $x=\frac{-1+\sqrt{5}}{2}$, for $m=5000,\ldots,40000$}
\label{m5000_40000_g}
\end{figure}

However, the estimate in Theorem \ref{reciprocaltheorem} is not true for all $x \in \Omega$.
Define $a \in \mathbb{N}^\mathbb{N}$ as follows. Let $a_1$ be any element of $\mathbb{N}$. Then inductively, define
$a_{n+1}$ to be any element of $\mathbb{N}$ that is $>q_n^{n-1}$. Then for any $n \in \mathbb{N}$,  using \eqref{upperbound} and
$q_{n+1}=a_{n+1}q_n+q_{n-1}>a_{n+1}q_n$ we get
\[
| q_n v(a) - p_n| < \frac{1}{q_{n+1}} <\frac{1}{a_{n+1}q_n}<\frac{1}{q_n^n},
\] 
hence $\norm{q_n v(a)} < \frac{1}{q_n^n}$, and then
\[
\sum_{j=1}^{q_n} \frac{1}{\norm{j v(a)}} > \frac{1}{\norm{q_n v(a)}} > q_n^n.
\]
Using $\epsilon=1$, it is then straightforward to check that there is no constant $C$ such that $\sum_{j=1}^m \frac{1}{\norm{jv(a)}} \leq 
C m(\log m)^2$ for all $m$. 

We will need the following lemma  \cite[p.~324, Lemma 3]{billingsley}  to prove a theorem;
cf. Khinchin \cite[p. 63, Theorem 30]{MR1451873}. 

\begin{lemma}
If $\phi$ is a function defined on the positive integers such that $\phi(n) \geq 1$ for all $n$ and 
\[
\sum_{n=1}^\infty \frac{1}{\phi(n)}<\infty,
\]
then for almost all $x \in \Omega$ there are only finitely many $n$ such that $a_n(x) \geq \phi(n)$.
\label{partialquotients}
\end{lemma}
\begin{proof}
For a measurable set $A \subset \Omega$, define
\[
\gamma(A)=\frac{1}{\log 2}\int_A \frac{1}{1+x} d\mu(x),
\]
in other words $d\gamma(x) = \frac{1}{(1+x)\log 2} d\mu(x)$.
Thus for a measurable set $A \subset \Omega$ we have
\[
\gamma(A) \leq \frac{1}{\log 2}\int_A dx = \frac{1}{\log 2}\mu(A)
\]
and
\[
\gamma(A) \geq \frac{1}{\log 2}\int_A \frac{1}{2}dx=\frac{1}{2\log 2}\mu(A).
\]

We will use that $\gamma$ is an invariant measure for the Gauss transformation $T:\Omega \to \Omega$  \cite[p.~77, Lemma 3.5]{einsiedler}, i.e., if $A \subset \Omega$ is a measurable set
then
\[
(T_*\gamma)(A) = \gamma(T^{-1}(A))=\gamma(A).
\]

Let $A_n=\{x \in \Omega: a_n(x) \geq \phi(n)\}$, $n \geq 1$. 
As 
\[
a_n(x)=\left[ \frac{1}{T^{n-1}(x)} \right],
\]
we have
\begin{eqnarray*}
A_n&\subset&\{x \in \Omega: \frac{1}{T^{n-1}(x)} > \phi(n) \}\\
&=&\left\{x \in \Omega: T^{n-1}(x) < \frac{1}{\phi(n)} \right\}\\
&=&(T^{n-1})^{-1}\left( \left[0,\frac{1}{\phi(n)}\right) \setminus \mathbb{Q} \right).
\end{eqnarray*}
Hence
\begin{eqnarray*}
\mu(A_n)&\leq&\mu\left((T^{n-1})^{-1}\left( \left[0,\frac{1}{\phi(n)}\right) \setminus \mathbb{Q} \right)\right)\\
&\leq& 2\log 2 \cdot \gamma\left((T^{n-1})^{-1}\left( \left[0,\frac{1}{\phi(n)}\right) \setminus \mathbb{Q} \right)\right)\\
&=&2\log 2 \cdot \gamma \left( \left[0,\frac{1}{\phi(n)}\right) \setminus \mathbb{Q} \right)\\
&\leq&2 \log 2 \cdot \frac{1}{\log 2} \cdot \mu \left( \left[0,\frac{1}{\phi(n)}\right) \setminus \mathbb{Q} \right)\\
&=&\frac{2}{\phi(n)}.
\end{eqnarray*}

It follows that 
\[
\sum_{n=1}^\infty \mu(A_n)<\infty,
\]
and thus by the Borel-Cantelli lemma \cite[p.~59, Theorem 4.3]{billingsley} we have
\[
\mu(\limsup_{n \to \infty} A_n)=0.
\]
\end{proof}

Let $\lambda$ be Lebesgue measure on $I=[0,1]$, let
$d\gamma(x) = \frac{1}{(1+x)\log 2} d\lambda(x)$, and let $T:I \to I$ be the Gauss transformation,
$T(x)=x^{-1}-[x]^{-1}$ for $x >0$ and $T(0)=0$, for which $T_*\gamma = \gamma$   \cite[p.~77, Lemma 3.5]{einsiedler}. 
Suppose that $\nu$ is a Borel probability measure on $[0,1]$ such that the pushforward measure $T_* \nu$ is absolutely continuous with respect
to $\nu$. For $f \in L^1(\nu)$,
define $d\nu_f = f d\nu$, and define
$P_\nu:L^1(\nu) \to L^1(\nu)$ by
\[
P_\nu f = \frac{d(T_* \nu_f)}{d\nu},\qquad f \in L^1(\nu).
\]
Thus, for $g \in L^\infty(\nu)$, using the change of variables formula,
\[
\int_I g \cdot P_\nu f d\nu = \int_I g d(T_* \nu_f) = 
\int_I g \circ T d\nu_f
=\int_I (g \circ T) \cdot f d\nu,
\]
in particular,
\[
\int_I P_\nu f d\nu = \int_I f d\nu.
\]
We call $P_\nu:L^1(\nu) \to L^1(\nu)$ a \textbf{Perron-Frobenius operator for $T$}. It is a fact that
if $f \geq 0$ then $P_\nu f \geq 0$ \cite[p.~57, Proposition 2.1.1]{iosifescu}, namely $P_\nu \geq 0$.
It can be proved  that for $f \in L^1(\gamma)$, for almost all $x \in I$ \cite[p.~59, Proposition 2.1.2]{iosifescu},
\[
(P_\gamma f)(x) = \sum_{k=1}^\infty \frac{x+1}{(x+k)(x+k+1)}\cdot f\left(\frac{1}{x+k}\right),
\]
and for $f \in L^1(\lambda)$, for almost all $x \in I$ \cite[p.~60, Corollary 2.1.4]{iosifescu},
\[
(P_\lambda f)(x) = \sum_{k=1}^\infty \frac{1}{(x+k)^2} \cdot f\left(\frac{1}{x+k}\right),
\]
and with $g(x)=(x+1)f(x)$, for $n \geq 1$ it holds for almost all $x \in I$ that
 $(P_\lambda^n f)(x) = \frac{(P_\gamma^n g)(x)}{x+1}$.
 Iosifescu and Kraaikamp \cite[Chapter 2]{iosifescu} give a detailed presentation of Perron-Frobenius operators for the Gauss
 map. We make the final remark that
 $P_\nu 1_I = 1_I$ is equivalent with
 $\int_I 1_E  d\nu = \int_I 1_{T^{-1}(E)} d\nu$ for all Borel sets $E$ in $I$, 
 i.e. $\nu(E) = \nu(T^{-1}(E))$, which in turn means $T_* \nu = \nu$, cf. Markov operators \cite[Chapter 5, \S 5.1, pp. 177-186]{Krengel}.
An object similar to Perron-Frobenius operators for the Gauss transformation
is the \textbf{zeta-function for the Gauss transformation}, for which see Lagarias \cite[p.~58, \S 3.3]{zetafunctions}.

The following theorem gives a lower bound on the sum $\sum_{j=1}^m \frac{1}{\norm{jx}}$, cf. \cite[p.~4, Theorem 3.1]{2007arXiv0709.2882V}.

\begin{theorem}
For almost all $x \in \Omega$ there is some $C>0$ such that
\[
\sum_{j=1}^m \frac{1}{\norm{jx}} > C m \log m.
\]
\label{omegabound}
\end{theorem}
\begin{proof}
For all $x \in \Omega$, if $n \geq 1$ then $q_n \geq 2^{\frac{n-1}{2}}$, by \eqref{qnbound}. Take $\phi(n)=2^{\frac{n-2}{2}}$. The series $\sum_{n=1}^\infty \frac{1}{\phi(n)}$ converges, so by Lemma \ref{partialquotients}, for almost all $x \in \Omega$ there are only finitely many $n$ such that  $a_n \geq \phi(n)$. That is, for almost all $x \in \Omega$ there is some $n_0$ such that if $n \geq n_0$ then
\[
a_n<\phi(n) =2^{\frac{n-2}{2}} \leq q_{n-1}.
\]
Hence, if $n \geq n_0$ then
\[
q_n = a_n q_{n-1}+q_{n-2} < q_{n-1}^2+q_{n-2}< 2 q_{n-1}^2.
\]
It follows that for almost all $x \in \Omega$ there is some $K$ such that
\begin{equation}
q_{n+1} < K q_n^2
\label{qnsquared}
\end{equation}
for all $n \geq 0$.

For such an $x$, let $m$ be a positive integer and let $q_n \leq m < q_{n+1}$. For $1 \leq j \leq m$ we have by \eqref{upperbound},
\[
\norm{jx-\frac{jp_n}{q_n}} \leq \left| jx-\frac{jp_n}{q_n} \right|=j  \left| x-\frac{p_n}{q_n} \right| < \frac{j}{q_n q_{n+1}}
< \frac{1}{q_n}.
\]
Therefore for $1 \leq j \leq m$ we have
\[
\norm{jx} \leq \norm{jx-\frac{jp_n}{q_n}}+ \norm{\frac{jp_n}{q_n}}<\frac{1}{q_n}+\norm{\frac{jp_n}{q_n}}.
\]
Let $L=[ \frac{m}{q_n} ]$, so $Lq_n \leq m$. Then,
\begin{eqnarray*}
\sum_{j=1}^m \frac{1}{\norm{jx}}&>&\sum_{j=1}^m \frac{1}{\frac{1}{q_n}+\norm{\frac{jp_n}{q_n}}}\\
&\geq&\sum_{l=0}^{L-1} \sum_{h=1}^{q_n} \frac{1}{\frac{1}{q_n}+\norm{\frac{(lq_n+h)p_n}{q_n}}}\\
&=&q_n \sum_{l=0}^{L-1} \sum_{h=1}^{q_n} \frac{1}{1+q_n\norm{\frac{hp_n}{q_n}}}\\
&=&Lq_n  \sum_{h=1}^{q_n} \frac{1}{1+q_n\norm{\frac{hp_n}{q_n}}}\\
&=&Lq_n \sum_{k=0}^{q_n-1} \frac{1}{1+q_n\cdot \frac{k}{q_n}}\\
&>&Lq_n \log q_n.
\end{eqnarray*}

But if $y\geq 1$ then $[ y ] > \frac{y}{2}$, so $L=[ \frac{m}{q_n} ] > \frac{m}{2q_n}$. Hence by \eqref{qnsquared},
\[
\sum_{j=1}^m \frac{1}{\norm{jx}} > \frac{m}{2} \log q_n > \frac{m}{2} \log \sqrt{\frac{q_{n+1}}{K}} > 
\frac{m}{2} \log \sqrt{\frac{m}{K}},
\]
and thus there is some $C>0$ such that $\sum_{j=1}^m \frac{1}{\norm{jx}} > C m \log m$ for all $m \geq 1$.
\end{proof}

The following is from Kuipers and Niederreiter \cite[p.~131, Exercise 3.12]{kuipers}.

\begin{theorem}
If $x \in \Omega$ is of type $<\psi$, then for all $m \geq 1$ we have
\[
\sum_{j=1}^m \frac{1}{j\norm{jx}}<24\left((\log m)^2 + \psi(m)+\sum_{j=1}^m \frac{\psi(j)}{j} \right).
\]
\label{hhalpha}
\end{theorem}
\begin{proof}
\textbf{Summation by parts} is the following identity, which can be easily checked:
\[
\sum_{n=1}^N a_n(b_{n+1}-b_n)=a_{N+1}b_{N+1}-a_1 b_1 - \sum_{n=1}^N b_{n+1}(a_{n+1}-a_n).
\]

Let $a_j=\frac{1}{j}$, let $s_j=\sum_{h=1}^j \frac{1}{\norm{h x}}$, let $b_1=0$, and let $b_j=s_{j-1}$ for $j  \geq 2$. Doing summation by parts gives
\[
\sum_{j=1}^m  \frac{1}{j\norm{jx}}=\frac{1}{m+1} s_m-\sum_{j=1}^m s_j \left(\frac{1}{j+1}-\frac{1}{j}\right)
=\frac{1}{m+1} s_m+\sum_{j=1}^m s_j \frac{1}{j(j+1)}.
\]
As $x$ is of type $<\psi$, we can use Theorem \ref{jalpha} to get $s_j<12j(\psi(j)+\log j)$ for each $j \geq 1$. Therefore
\begin{eqnarray*}
\sum_{j=1}^m  \frac{1}{j\norm{jx}}&<&\frac{1}{m+1}12m(\psi(m)+\log m)+\sum_{j=1}^m \frac{12(\psi(j)+\log j)}{j+1}\\
&<&12(\psi(m)+\log m)+12 (\log m)^2+12\sum_{j=1}^m \frac{\psi(j)}{j+1}\\
&<&24(\log m)^2 + 12\psi(m)+12\sum_{j=1}^m \frac{\psi(j)}{j+1}.
\end{eqnarray*}
\end{proof}

Erd\H os \cite{remarks} proves that for almost all $x$,
\[
\sum_{j=1}^m \frac{1}{j\norm{jx}}=(1+o(1))(\log m)^2.
\]
Kruse \cite{MR0207642} gives a comprehensive investigation of the sums $\sum_{j=1}^m \frac{1}{j^s \norm{jx}^t}$, $s,t \geq 0$. The results depend on whether $s$ and $t$ are are less than, equal, or greater than $1$, and on whether $t<s$. One of the theorems proved by Kruse is the following \cite[p.~260, Theorem 7]{MR0207642}. If $t>1$ and $0 \leq s \leq t$, and if $\epsilon>0$, then for almost all $x$ we have
\[
\sum_{j=1}^m \frac{1}{j^s \norm{jx}^t} = O\left(m^{t-s} (\log m)^{(1+\epsilon)t} \right).
\]

Haber and Osgood  \cite[p.~387, Theorem 1]{haber} prove that for real $t \geq 1$, $A>1$, $M>0$, $r>0$, there is some $C=C(t,A,M,r)>0$ such that for
all $x \in \Omega$ satisfying $q_{n+1}(x)<M q_n(x)^r$, for all positive integers $K$,
\[
\sum_{n=K+1}^{[AK]} \norm{nx}^{-t} > \begin{cases}
CK \log K&t=1\\
CK^{1+(t-1)/r}&t>1.
\end{cases}
\]
We remind ourselves that according to Theorem \ref{Dtau}, the elements of $\mathcal{D}(r+1)$ are those
$x \in \Omega$ for which there is some $c(x)>0$ such that $q_{n+1}(x) \leq C(x) q_n(x)^r$ for all $n \geq 1$.

For $x \in \mathbb{Z}+\frac{1}{2}$, define $\{\{x\}\}=\frac{1}{2}$. If $x \not \in \mathbb{Z}+\frac{1}{2}$, then there is 
an integer $m_x$ for which $|x-m_x| < |x-n|$ for all integers $n \neq m_x$, and we define $\{\{x\}\}=x-m_x$. 
Sinai and Ulcigrai \cite[p.~96, Proposition 2]{fixed} prove that if $\alpha$ has bounded partial quotients, then there is some
$C(\alpha)$ such that for all $M$,
\[
\left| \sum_{m=1}^M \frac{1}{\{\{m\alpha\}\}} \right| \leq C(\alpha)M.
\]

\section{Weyl's inequality, Vinogradov's estimate, Farey fractions, and the circle method}
Write
\[
\mathscr{A} = \{(a,q) \in \mathbb{Z}^2: \gcd(a,q)=1, q \geq 1\}.
\]
We first prove four estimates following Nathanson  \cite[pp.~104--110, Lemmas 4.8--4.11]{nathanson} that we will use in what follows; cf. Vinogradov \cite[p.~26, Chapter I,
Lemma 8b]{vinogradov}.

\begin{lemma}
There is some $C$ such that if $\alpha \in \mathbb{R}$, $(a,q) \in \mathscr{A}$, and
\[
\left|\alpha - \frac{a}{q} \right| \leq \frac{1}{q^2},
\]
then
\[
\sum_{1 \leq r \leq q/2} \frac{1}{\norm{\alpha r}} \leq C q\log q.
\]
\label{lemma48}
\end{lemma}
\begin{proof}
For $q=1$, $\sum_{1 \leq r \leq q/2} \frac{1}{\norm{\alpha r}}=0$. For $q \geq 2$, let
$1 \leq r \leq \frac{q}{2}$. 
As $\gcd(a,q)=1$ and $r \not \equiv 0 \pmod{q}$, $ar \not \equiv 0 \pmod{q}$. So for $\mu_r = \left[ \frac{ar}{q} \right]$, there is some
$1 \leq \sigma_r \leq q-1$ such that $ar=\mu_r q + \sigma_r$. Then
\[
\norm{\frac{ar}{q}} = \norm{\frac{\sigma_r}{q}} \in \left\{ \frac{\sigma_r}{q},1-\frac{\sigma_r}{q}\right\}
=\left\{\frac{\sigma_r}{q},\frac{q-\sigma_r}{q}\right\}.
\]
Put $\frac{s_r}{q} =  \norm{\frac{ar}{q}}$, so (i) $s_r = \sigma_r$ or (ii) $s_r = q-\sigma_r$.
In case (i), $\frac{s_r}{q} = \frac{ar}{q}-\mu_r$.
In case (ii), $\frac{s_r}{q} = 1-\left( \frac{ar}{q} - \mu_r \right)$.
In case (i) let $\epsilon_r = 1, m_r=\mu_r$, and in case (ii) let $\epsilon_r=-1,
m_r = \mu_r+1$. Thus, whether (i) or (ii) holds we have
\[
\frac{s_r}{q} = \epsilon_r \left(\frac{ar}{q}-m_r\right),\quad \frac{s_r}{q} = \norm{\frac{ar}{q}}, \quad 1 \leq s_r \leq \frac{q}{2}.
\]
Write 
\[
\alpha - \frac{a}{q} = \frac{\theta}{q^2},
\]
for some real $\theta$, $|\theta| \leq 1$. For $\theta_r = \frac{2r}{q} \theta$, which satisfies
$|\theta_r| \leq |\theta| \leq 1$,
\[
\alpha r = \frac{ar}{q} + \frac{r\theta}{q^2} = \frac{ar}{q} + \frac{\theta_r}{2q}.
\]
Then
\begin{align*}
\norm{\alpha r}&=\norm{\frac{ar}{q}+\frac{\theta_r}{2q}}\\
&=\norm{\epsilon_r \frac{s_r}{q} + m_r + \frac{\theta_r}{2q}}\\
&\geq \norm{\epsilon_r \frac{s_r}{q} + m_r} - \norm{\frac{\theta_r}{2q}}\\
&=\frac{s_r}{q} - \left|\frac{\theta_r}{2q}\right|\\
&\geq \frac{s_r}{q} - \frac{1}{2q}.
\end{align*}

Take $1 \leq r_1, r_2 \leq \frac{q}{2}$ and suppose that $s_{r_1}=s_{r_2}$. So 
\[
 \epsilon_{r_1} \left(\frac{ar_1}{q}-m_{r_1}\right) =  \epsilon_{r_2} \left(\frac{ar_2}{q}-m_{r_2}\right)
\]
hence $ar_1 \equiv \epsilon_{r_1} \epsilon_{r_2} ar_2 \pmod{q}$.  As
$\gcd(a,q)=1$, $r_1 \equiv \epsilon_{r_1} \epsilon_{r_2} r_2 \pmod{q}$. 
Because $1 \leq r_1, r_2 \leq q$, 
if $r_1 \equiv r_2 \pmod{q}$ then $r_1=r_2$ and if $r_1 \equiv -r_2 \pmod{q}$ then $r_1 = \frac{q}{2}$
and $r_2=\frac{q}{2}$, so in any case $r_1=r_2$. 
Therefore
\[
\left\{ \frac{s_r}{q} : 1 \leq r \leq \frac{q}{2} \right\}
=\left\{\frac{s}{q} : 1 \leq s \leq \frac{q}{2} \right\}.
\]

Using the two things we have established, 
\begin{align*}
\sum_{1 \leq r \leq q/2} \frac{1}{\norm{\alpha r}}&\leq \sum_{1 \leq r \leq q/2} \frac{1}{\frac{s_r}{q}-\frac{1}{2q}}\\
&=\sum_{1 \leq s \leq q/2} \frac{1}{\frac{s}{q}-\frac{1}{2q}}\\
&= 2q \sum_{1 \leq s \leq q/2}  \frac{1}{2s-1}\\
&\leq 2q \sum_{1 \leq s \leq q/2} \frac{1}{s}\\
&\leq 2q \left( \log \frac{q}{2} + \gamma + O(q^{-1}) \right)\\
&=O(q \log q).
\end{align*}
\end{proof}

\begin{lemma}
There is some $C$ such that if $\alpha \in \mathbb{R}$, $(a,q) \in \mathscr{A}$, and
\[
\left|\alpha - \frac{a}{q} \right| \leq \frac{1}{q^2},
\]
then for any positive real $V$ and nonnegative integer $h$,
\[
\sum_{r=1}^q \min\left(V,\frac{1}{\norm{\alpha(hq+r)}}\right) \leq C(V+q\log q).
\]
\label{lemma49}
\end{lemma}
\begin{proof}
Write
\[
\alpha = \frac{a}{q}+\frac{\theta}{q^2},
\]
which satisfies $|\theta| \leq 1$, and for $1 \leq r \leq q$ define
\[
\delta_r = R(\theta h)+\frac{\theta r}{q},
\]
which satisfies $-1 \leq \delta_r < 2$. Then
\begin{align*}
\alpha(hq+r)&= \left(\frac{a}{q}+\frac{\theta}{q^2}\right)(hq+r)\\
&=ah+\frac{ar}{q}+\frac{\theta h}{q}+\frac{\theta r}{q^2}\\
&=ah+\frac{ar}{q} + \frac{R(\theta h)+[\theta h]}{q}+\frac{\theta r}{q^2}\\
&=ah+\frac{ar+[\theta h]+\delta_r}{q}.
\end{align*}
For $m_r = \left[ \frac{ar+[\theta h]+\delta_r}{q^2} \right]$, 
\[
R(\alpha(hq+r)) = R\left( \frac{ar+[\theta h]+\delta_r}{q}\right) = \frac{ar+[\theta h]+\delta_r}{q} - m_r.
\]

Suppose that $t \in \left[0, 1-\frac{1}{q}\right]$ and that $t \leq R(\alpha(hq+r)) \leq t+\frac{1}{q}$.
Then
\[
qt \leq ar+[\theta h]+\delta_r - qm_r \leq qt+1.
\]
This implies, as $\delta_r \geq -1$,
\[
ar-qm_r \leq qt+1 - [\theta h] - \delta_r \leq qt+1 -[\theta h] + 1 = qt-[\theta h]+2
\]
and, as $\delta_r < 2$,
\[
ar-qm_r \geq qt - [\theta h]-\delta_r > qt - [\theta h] - 2,
\]
so $ar-qm_r  \in J_t$, writing
\[
J_t=(qt-[\theta h]-2,qt-[\theta h]+2].
\] 
For $1 \leq r_1,r_2 \leq  q$, if $ar_1-qm_{r_1} = ar_2-qm_{r_2}$ then $ar_1 \equiv ar_2 \pmod{q}$, and 
$\gcd(a,q)=1$ implies $r_1 \equiv r_2 \pmod{q}$; and $1 \leq r_1,r_2 \leq q$ so
$r_1=r_2$. 
For
$t \in \left[0, 1-\frac{1}{q}\right]$,
four integers belong to $J_t$, hence
\[
\{1 \leq r \leq q: ar-qm_r \in J_t\}
\]
has at most four elements. But 
\[
\left\{1 \leq r \leq q: R(\alpha(hq+r)) \in \left[t,t+\frac{1}{q}\right]\right\}
\subset \{1 \leq r \leq q: ar-qm_r \in J_t\}.
\]
Now,
\[
\begin{split}
&\left\{1 \leq r \leq q: \norm{\alpha(hq+r)} \in \left[t,t+\frac{1}{q}\right]\right\}\\
=&\left\{1 \leq r \leq q: R(\alpha(hq+r)) \in \left[t,t+\frac{1}{q}\right]\right\}\\
&\cup \left\{1 \leq r \leq q: 1-R(\alpha(hq+r)) \in \left[t,t+\frac{1}{q}\right]\right\}\\
=&\left\{1 \leq r \leq q: R(\alpha(hq+r)) \in \left[t,t+\frac{1}{q}\right]\right\}\\
&\cup \left\{1 \leq r \leq q: R(\alpha(hq+r)) \in \left[1-\frac{1}{q}-t,1-t\right]\right\},
\end{split}
\]
whence
\begin{align*}
\left\{1 \leq r \leq q: \norm{\alpha(hq+r)} \in \left[t,t+\frac{1}{q}\right]\right\} &\subset 
\{1 \leq r \leq q: ar-qm_r \in J_t\}\\
&\cup \{1 \leq r \leq q: ar-qm_r \in J_{1-\frac{1}{q}-t}\}.
\end{align*}
This shows that if $t \in \left[0,1-\frac{1}{q}\right]$ then
\[
\left\{1 \leq r \leq q: \norm{\alpha(hq+r)} \in \left[t,t+\frac{1}{q}\right]\right\}
\]  
has at most eight elements. For   $0 \leq k < \frac{q}{2}$, writing
\[
I_k = \left[\frac{k}{q},\frac{k}{q}+\frac{1}{q}\right],
\] 
the set $\{1 \leq r \leq q: \norm{\alpha(hq+r)} \in I_k\}$ has at most eight elements. Therefore
\begin{align*}
\sum_{1 \leq r \leq q} \min\left(V,\frac{1}{\norm{\alpha(hq+r)}}\right)&=
\sum_{0 \leq k < q/2} \sum_{1 \leq r \leq q, \norm{\alpha(hq+r)} \in I_k} \min\left(V,\frac{1}{\norm{\alpha(hq+r)}}\right)\\
&\leq 8V + \sum_{1 \leq k < q/2} \sum_{1 \leq r \leq q, \norm{\alpha(hq+r)} \in I_k} \frac{1}{\norm{\alpha(hq+r)}}\\
&\leq 8V + \sum_{1 \leq k < q/2}  8 \cdot \frac{q}{k}\\
&=O(V+q \log q).
\end{align*}
\end{proof}

\begin{lemma}
There is some $C$ such that if $\alpha \in \mathbb{R}$, $(a,q) \in \mathscr{A}$, 
\[
\left|\alpha - \frac{a}{q} \right| \leq \frac{1}{q^2},
\]
 $U \geq 1$ is a real number, and $n$ is a positive integer, then
\[
\sum_{1 \leq k \leq U} \min\left( \frac{n}{k}, \frac{1}{\norm{\alpha k}} \right) \leq C
\left(\frac{n}{q}+U+q\right) \log 2qU.
\]
\label{lemma410}
\end{lemma}
\begin{proof}
For $1 \leq k \leq U$ there is some $0 \leq h_k < \frac{U}{q}$ and
$1 \leq r_k \leq q$ such that $k=q h_k  + r_k$, and then
\begin{align*}
\sum_{1 \leq k \leq U} \min\left( \frac{n}{k}, \frac{1}{\norm{\alpha k}} \right)
&\leq \sum_{0 \leq h < U/q} \sum_{1 \leq r \leq q} \min\left(\frac{n}{qh+r},\frac{1}{\norm{\alpha(hq+r)}}\right)\\
&\leq \sum_{1 \leq r \leq q/2} \frac{1}{\norm{\alpha r}}+\sum_{q/2<r \leq q}  \min\left(\frac{n}{r},\frac{1}{\norm{\alpha r}}\right)\\
&+\sum_{1 \leq h < U/q} \sum_{1 \leq r \leq q}\min\left(\frac{n}{qh+r},\frac{1}{\norm{\alpha(hq+r)}}\right)\\
&\leq C_1  q \log q 
+\sum_{q/2<r \leq q}  \min\left(\frac{n}{r},\frac{1}{\norm{\alpha r}}\right)\\
&+\sum_{1 \leq h < U/q} \sum_{1 \leq r \leq q}\min\left(\frac{n}{qh+r},\frac{1}{\norm{\alpha(hq+r)}}\right),
\end{align*}
the last inequality by Lemma \ref{lemma48}.
If $\frac{q}{2} < r \leq q$ then 
$\frac{1}{r} < \frac{2}{q} = \frac{2}{(h+1)q}$ for $h=0$, and if
$1 \leq h < \frac{U}{q}$ and $1 \leq r \leq q$ then $h \geq \frac{h+1}{2}$ so
$hq+r > hq \geq \frac{(h+1)q}{2}$ and hence
$\frac{1}{hq+r} < \frac{2}{(h+1)q}$, whence
\[
\begin{split}
&\sum_{q/2<r \leq q}  \min\left(\frac{n}{r},\frac{1}{\norm{\alpha r}}\right)+
\sum_{1 \leq h < U/q} \sum_{1 \leq r \leq q}\min\left(\frac{n}{qh+r},\frac{1}{\norm{\alpha(hq+r)}}\right)\\
\leq&2\sum_{q/2<r \leq q}  \min\left( \frac{n}{(h+1)q},\frac{1}{\norm{\alpha r}}\right)+
2\sum_{1 \leq h < U/q}  \sum_{1 \leq r \leq q} \min\left( \frac{n}{(h+1)q},\frac{1}{\norm{\alpha(hq+r)}}\right).
\end{split}
\]
Consquently 
\[
\sum_{1 \leq k \leq U} \min\left( \frac{n}{k}, \frac{1}{\norm{\alpha k}} \right)
\leq C_1 q \log q + 2 \sum_{0 \leq h < U/q} \sum_{1 \leq r \leq q} \min\left( \frac{n}{(h+1)q},
\frac{1}{\norm{\alpha(hq+r)}}\right).
\]
Lemma \ref{lemma49} with $V=\frac{n}{(h+1)q}$ says
\[
\sum_{1 \leq r \leq q} \min\left( \frac{n}{(h+1)q},\frac{1}{\norm{\alpha(hq+r)}}\right) \leq C_2  \left( \frac{n}{(h+1)q}+q \log q\right),
\]
therefore
\begin{align*}
\sum_{1 \leq k \leq U} \min\left( \frac{n}{k}, \frac{1}{\norm{\alpha k}} \right)&\ll q \log q + \sum_{0 \leq h < U/q}  \left(\frac{n}{(h+1)q}+q \log q\right)\\
&\ll q \log q + \frac{n}{q} \sum_{1 \leq h <\frac{U}{q}+1} \frac{1}{h} + q (\log q) \left( \frac{U}{q}+1\right)\\
&\ll q \log q + \frac{n}{q} \log \left(\frac{U}{q}+1\right) + U \log q\\
&\ll U \log 2qU + q \log 2qU + \frac{n}{q} \log \left(\frac{U}{q}+1\right).
\end{align*}
If $U \leq q$ then $\frac{U}{q}+1 \leq 2 \leq 2qU$, and if $U > q$ then $\frac{U}{q} +1 \leq U+1 \leq 2U \leq 2qU$, hence
\[
\sum_{1 \leq k \leq U} \min\left( \frac{n}{k}, \frac{1}{\norm{\alpha k}} \right)
\ll U \log 2qU + q \log 2qU + \frac{n}{q} \log 2qU.
\]
\end{proof}

\begin{lemma}
There is some $C$ such that if $\alpha \in \mathbb{R}$, $(a,q) \in \mathscr{A}$, 
\[
\left| \alpha - \frac{p}{q} \right| \leq \frac{1}{q^2},
\]
and $U,V \geq 1$ are real numbers, then 
\[
\sum_{1 \leq k \leq U} \min\left(V,\frac{1}{\norm{\alpha k}} \right) \leq C\left(q+U+V+\frac{UV}{q}\right) \max\{1,\log q\}.
\]
\label{lemma411}
\end{lemma}
\begin{proof}
For $1 \leq k \leq U$ there is some $0 \leq h_k < \frac{U}{q}$ and $1 \leq r_k \leq q$ such that
$k=qh_k+r_k$, and then, as in the proof of Lemma \ref{lemma410},
\begin{align*}
\sum_{1 \leq k \leq U} \min\left(V,\frac{1}{\norm{\alpha k}}\right)&\leq 
\sum_{0 \leq h<U/q} \sum_{1 \leq r \leq q} \min\left(V,\frac{1}{\norm{\alpha(hq+r)}}\right)\\
&\leq C_1 q \log q +2 \sum_{0 \leq h < U/q} \sum_{1 \leq r \leq q} \min\left(V,\frac{1}{\norm{\alpha(hq+r)}}\right).
\end{align*}
Using Lemma \ref{lemma49}, 
\begin{align*}
\sum_{1 \leq k \leq U} \min\left(V,\frac{1}{\norm{\alpha k}}\right)&\leq C_1 q \log q + 2C_2 \sum_{0 \leq h < U/q} (V+q\log q)\\
&\ll q \log q + (V+q\log q)\left(\frac{U}{q}+1\right)\\
&\ll q \log q + \frac{UV}{q}+V+U\log q.
\end{align*}
\end{proof}

\textbf{Weyl's inequality} \cite[p.~114, Theorem 4.3]{nathanson} is the following.
For $k \geq 2$ and $\epsilon>0$, there is some $C(k,\epsilon)$ such that if
$\alpha \in \mathbb{R}$, $f(x)$ is a real polynomial with highest degree term $\alpha x^k$, 
$(a,q) \in \mathscr{A}$, and
\[
\left|\alpha-\frac{a}{q}\right| \leq \frac{1}{q^2},
\]
then, writing $S_N(f)=\sum_{j=1}^N e^{2\pi i f(j)}$ and $K=2^{k-1}$,
\[
|S_N(f)| \leq C(k,\epsilon) \cdot N^{1+\epsilon}(N^{-1}+q^{-1}+N^{-k}q)^{\frac{1}{K}}.
\]
Weyl's inequality is proved using Lemma \ref{lemma411}.

Montgomery \cite[Chapter 3]{montgomery} gives a similar but more streamlined presentation of Weyl's inequality.
Chandrasekharan \cite{MR0369277} gives a historical survey of exponential sums.

\textbf{Vinogradov's estimate}  \cite[p.~26, Theorem 3.1]{vaughan}  states that
there is some $C$ such that for $n \geq 2$, $1 \leq q \leq n$, $\gcd(a,q)=1$, and $\left|\alpha - \frac{a}{q}\right| \leq q^{-2}$, then
\begin{equation}
|f_n(\alpha)| \leq C (nq^{-1/2} + n^{4/5} + n^{1/2} q^{1/2}) (\log n)^4,
\label{vinogradov}
\end{equation}
where  $f_n(\alpha) = \sum_{p \leq n} (\log p)e^{2\pi i\alpha p}$;
cf. Nathanson \cite[p.~220, Theorem 8.5]{nathanson} and Vinogradov \cite[p.~131, Chapter IX, Theorem 1]{vinogradov}.
This is proved using Lemma \ref{lemma410}.

Fix $B>0$ and let $P_n = (\log n)^B$. For $1 \leq a \leq q \leq P_n$ and $\gcd(a,q)=1$, let
\[
\mathfrak{M}_n(q,a) = \left\{ \alpha \in \mathbb{R} : \left|\alpha - \frac{a}{q} \right| \leq P_n n^{-1} \right\},
\]
called a \textbf{major arc}. One checks that there is some $n_B$ such that if $n \geq n_B$, then
$\mathfrak{M}_n(q,a)$ and $\mathfrak{M}_n(q',a')$ are disjoint when $(q,a) \neq (q',a')$. 
Let
\[
\mathfrak{M}_n = \bigcup_{1 \leq a \leq q \leq P_n, \gcd(a,q)=1} \mathfrak{M}_n(q,a).
\]

The \textbf{Farey fractions of order $N$} are
\[
\mathfrak{F}_N = \left\{ \frac{h}{k} : 0 \leq h \leq k \leq N, \gcd(h,k)=1\right\}.
\]
Cf. the \textbf{Stern-Brocot tree} \cite[\S 4.5]{concrete}.
For early appearances of  Farey
fractions, see Dickson \cite[pp.~155--158, Chapter V]{dicksonI}.
It is proved by Cauchy that if $h/k$ and $h'/k'$ are successive elements of $\mathfrak{F}_N$, then
$kh'-hk'=1$ \cite[p.~23, Theorem 28]{wright}.
Let
\[
\phi(m) = |\{1 \leq k \leq m: \gcd(k,m)=1\}|,
\]
the \textbf{Euler phi function}, and write $\Phi(N) = \sum_{1 \leq m \leq N} \phi(m)$. 
One sees that 
$|\mathfrak{F}_N| = 1+\Phi(N)$, and
 it was proved by Mertens  \cite[p.~268]{wright} that
\[
\Phi(N) = \frac{3N^2}{\pi^2} + O(N\log N).
\]

Let $\lambda$ be Lebesgue measure on $\mathbb{R}$.
For $n \geq n_B$, because the major arcs are pairwise disjoint,
\begin{align*}
\lambda(\mathfrak{M}_n)&=\sum_{1 \leq a \leq q \leq P_n, \gcd(a,q)=1} 2P_n n^{-1}\\
&=\left( \sum_{m=1}^{P_n} \phi(m) \right) 2P_n n^{-1}\\
&=\frac{6}{\pi^2} P_n^3 n^{-1} + O(P_n^2 n^{-1} \log P_n).
\end{align*} 

Let $I_n = (P_nn^{-1},1+P_nn^{-1}]$, which
for $n>2P_n^2$
 contains  $\mathfrak{M}_n$. Let
\[
\mathfrak{m}_n=I_n \setminus \mathfrak{M}_n,
\]
called the \textbf{minor arcs}.

With $f_n(\alpha) = \sum_{p \leq n} (\log p) e^{2\pi i\alpha p}$ and 
\[
R(n) = \sum_{p_1+p_2+p_3=n} (\log p_1) (\log p_2) (\log p_3),
\]
we have
\[
R(n)=\int_{I_n} f_n(\alpha)^3 e^{-2\pi in\alpha} d\alpha
=\int_{\mathfrak{M}_n}  f_n(\alpha)^3 e^{-2\pi in\alpha} d\alpha
+\int_{\mathfrak{m}_n}  f_n(\alpha)^3 e^{-2\pi in\alpha} d\alpha.
\]
Using \eqref{vinogradov}, it can be proved that for $A>0$ with $B \geq 2A+10$ \cite[p.~29, Theorem 3.2]{vaughan},
\[
\int_{\mathfrak{m}_n} |f_n(\alpha)|^3 d\alpha = O(n^2 (\log n)^{-A}).
\]
Writing 
\[
\mathfrak{S}(n) = \left( \prod_{p \nmid n} (1+(p-1)^{-3}) \right) \prod_{p \mid n} (1-(p-1)^{-2}),
\]
called the \textbf{singular series}, it is proved, using the Siegel-Walfisz theorem on primes in arithmetic progressions \cite[p.~381, Corollary
11.19]{multiplicative}, that 
for $A>0$ with $B \geq 2A$ \cite[p.~31, Theorem 3.3]{vaughan},
\[
\int_{\mathfrak{M}_n} f_n(\alpha)^3 e^{-2\pi in\alpha} d\alpha = \frac{1}{2} n^2 \mathfrak{S}(n) + O(n^2(\log n)^{-A}).
\]
Thus
\[
R(n) = \frac{1}{2} n^2 \mathfrak{S}(n) + O(n^2(\log n)^{-A}),
\]
and it follows from this  that there is some $n_0$ such that if $n \geq n_0$ is odd then there are primes
$p_1,p_2,p_3$ such that $n=p_1+p_2+p_3$. 

For integers $a,b$ with $\gcd(a,b)=1$, the \textbf{Ford circle} $C(a,b)$ is the circle in $\mathbb{C}$ 
that touches the line $\Im z=0$ at $z=\frac{a}{b}$ and has radius $\frac{1}{2b^2}$; in other words, $C(a,b)$ is the circle in $\mathbb{C}$
with center $\frac{a}{b}+\frac{i}{2b^2}$ and radius
$\frac{1}{2b^2}$.
It is straightforward to prove that if $C(a,b)$ and $C(c,d)$ are Ford circles, then they are tangent if and only if $(bc-ad)^2=1$, and
otherwise they are disjoint \cite[p.~100, Theorem 5.6]{apostol}. It is also straightforward to prove \cite[p.~101, Theorem 5.7]{apostol} that if 
$\frac{h_1}{k_1} < \frac{h}{k} < \frac{h_2}{k_2}$ are successive elements of $\mathfrak{F}_N$, then 
$C(h_1,k_1)$ and $C(h,k)$ touch at
\[
\frac{h}{k} - \frac{k_1}{k(k^2+k_1^2)} + \frac{i}{k^2+k_1^2}
\]
and  $C(h,k)$ and $C(h_2,k_2)$ touch at
\[
\frac{h}{k} + \frac{k_2}{k(k^2+k_2^2)} + \frac{i}{k^2+k_2^2}.
\]
Bonahon \cite[pp.~207 ff., Chapter 8]{bonahon} explains Ford circles in the language of hyperbolic geometry.

We remind ourselves that $|\mathfrak{F}_N|=1+\Phi(N) = 1+\sum_{m=1}^N \phi(m)$ and let
 $\rho_{0,N} <  \cdots < \rho_{\Phi(N),N}$ be the elements of $\mathfrak{F}_N$. In particular,
$\rho_{0,N}=0$ and $\rho_{\Phi(N),N}=1$.  For $\rho_{n,N} = \frac{h_{n,N}}{k_{n,N}}$ with
$\gcd(h_{h,N},k_{n,N})=1$,
write
$C_{n,N} = C(h_{n,N},k_{n,N})$. 

Let $A_{0,N}$ be the clockwise arc of $C_{0,N}$ from $i$ to the point at which $C_{0,N}$ and $C_{1,N}$ touch. 
For $0<n<\Phi(N)$, let $A_{n,N}$ be the clockwise arc of $C_{n,N}$ from the point at which $C_{n-1,N}$ and $C_{n,N}$ touch
to the point at which $C_{n,N}$ and $C_{n+1,N}$ touch. Finally, let $A_{\Phi(N),N}$ be the clockwise arc of $C_{\Phi(N),N}$ from the point at which
$C_{\Phi(N)-1,N}$ and $C_{\Phi(N),N}$ touch to $i+1$.
Let $A_N$ be the composition of the arcs
\begin{equation}
A_{0,N},\ldots,A_{\Phi(N),N},
\label{arc_composition}
\end{equation}
which is a contour from $i$ to $i+1$. 

Write $H = \{\tau \in \mathbb{C} : \Im \tau > 0\}$. 
The \textbf{Dedekind eta function} $\eta:H \to \mathbb{C}$ is defined by
\[
\eta(\tau) = e^{\pi i \tau/12} \prod_{m=1}^\infty (1-e^{2\pi im\tau}).
\]
It is straightforward to check that $\eta$ is analytic and that
$\eta(\tau) \neq 0$ for all $\tau \in H$ \cite[pp.~17--18, \S 1.44]{titchmarsh}. For 
$(h,k) \in \mathscr{A}$, let
\[
s(h,k) = \sum_{r=1}^{k-1} \frac{r}{k} \left( \frac{hr}{k} - \left[ \frac{hr}{k} \right] - \frac{1}{2}\right)
=\sum_{r=1}^{k-1} \frac{r}{k} P_1(hr/k),
\]
called a \textbf{Dedekind sum}; $P_1$ is the periodic Bernoulli function.
Also, for $\begin{pmatrix}a&b\\c&d\end{pmatrix} \in SL_2(\mathbb{Z})$ write
\[
\epsilon(a,b,c,d)  = \exp\left( \pi i\left(\frac{a+d}{12c} + s(-d,c)\right) \right).
\]
The \textbf{functional equation for the Dedekind eta function} \cite[p.~52, Theorem 3.4]{apostol} is 
\[
\eta\left( \frac{a\tau +b}{c\tau+d} \right) = \epsilon(a,b,c,d) (-i(c\tau+d))^{1/2} \eta(\tau),
\quad \begin{pmatrix}a&b\\c&d\end{pmatrix} \in SL_2(\mathbb{Z}), \tau \in H.
\]

Let $p(n)$ be the number of ways of writing $n$ as a sum of positive integers where the order does not matter, called the
\textbf{partition function}. For example, $4,3+1,2+2,2+1+1,1+1+1+1$ are the partitions of $4$, so $p(4)=5$. 
Denoting by $D(0,1)$ the open disc with center $0$ and radius $1$, define $F:D(0,1) \to \mathbb{C}$ by
\[
F(z) = \prod_{m=1}^\infty (1-z^m)^{-1} =  \sum_{n=0}^\infty p(n) z^n;
\]
that the product and the  series are equal was found by Euler.
$F$ is analytic. On the one hand, $p(n) =\frac{F^{(n)}(0)}{n!} $, and on the other hand, by Cauchy's integral formula  \cite[p.~82, Theorem 2.41]{titchmarsh}, if $C$ is a circle with center $0$ and radius $0<R<1$ then 
\[
F^{(n)}(0) = \frac{n!}{2\pi i} \int_C \frac{F(z)}{z^{n+1}} dz.
\]
Taking $C$ to be the circle with center $0$ and
radius $e^{-2\pi}$ and doing the change of variable $z = e^{2\pi i\tau}$, 
\begin{align*}
p(n) & = \frac{1}{2\pi i} \int_i^{i+1} \frac{F(e^{2\pi i\tau})}{e^{2\pi i (n+1)\tau}} \cdot 2\pi i e^{2\pi i\tau}
d\tau\\
&=\int_i^{i+1} F(e^{2\pi i\tau}) e^{-2\pi in\tau} d\tau\\
&=\int_{A_N} F(e^{2\pi i\tau}) e^{-2\pi in\tau} d\tau,
\end{align*}
where we remind ourselves that $A_N$ is the contour \eqref{arc_composition} from $i$ to $i+1$.
Using this and the functional equation for the Dedekind eta function, 
Rademacher \cite[p.~104, Theorem 5.10]{apostol} proves that for $n \geq 1$,
\[
p(n) = \frac{1}{\pi \sqrt{2}} \sum_{k=1}^\infty A_k(n) k^{1/2} \frac{d}{dn}  \frac{ \sinh \left( \pi \sqrt{\frac{2}{3}} \cdot \frac{1}{k} \sqrt{n-\frac{1}{24}}\right)}{\sqrt{n-\frac{1}{24}}},
\]
for $A_k(n) = \sum_{0 \leq h < k, \gcd(h,k)=1} e^{\pi i s(h,k) - 2\pi inh/k}$.

We make a final remark about the Farey fractions. Writing the elements of $\mathfrak{F}_N$ as $\rho_{0,N}<\cdots<\rho_{\Phi(N),N}$,
 let $\eta_{n,N} = \rho_{n,N} - \frac{n}{\Phi(N)}$ for $1 \leq n \leq N$. 
For example,  $\Phi(5)=10$ and
\[
\left\{\rho_{1,5},\ldots,\rho_{10,5}\right\}=
 \left\{\frac{1}{5},\frac{1}{4},\frac{1}{3},\frac{2}{5},\frac{1}{2},\frac{3}{5},\frac{2}{3},\frac{3}{4},\frac{4}{5},1\right\},
\]
and 
\[
\{\eta_{1,5},\ldots,\eta_{10,5}\} = \left\{\frac{1}{10},\frac{1}{20},\frac{1}{30},0,0,0,-\frac{1}{30},
-\frac{1}{20},-\frac{1}{10},0\right\}.
\]
Landau \cite{farey}, following work of Franel, proves that the \textbf{Riemann hypothesis} is true if and only if
for every $\epsilon>0$,
\[
\sum_{n=1}^{\Phi(N)} |\eta_{n,N}| = O(N^{\frac{1}{2}+\epsilon}).
\]
For example, 
for $N=5$, the left-hand side is $\frac{11}{30}$. 
See Narkiewicz \cite[p.~40, \S 2.2.3]{narkiewicz}.

\section{Discrepancy and exponential sums}
\label{discrepancy}
Discrepancy and Diophantine approximation are covered by Kuipers and Niederreiter \cite[Chapters 1--2 ]{kuipers} and by 
Drmota and Tichy \cite[\S\S 1.1--1.4]{drmota}, especially \cite[pp.~48--66, \S 1.4.1]{drmota}.

Let $\omega=(x_n)$, $n \geq 1$, be a sequence of real numbers, and let $E \subset [0,1)$. For a positive integer $N$, let
$A(E;N;\omega)$ be the number of $x_n$, $1 \leq n \leq N$, such that $R(x_n) \in E$. 
We say that the sequence $\omega$ is \textbf{uniformly distributed modulo $1$} if we have for all $a$ and $b$ with $0 \leq a < b \leq 1$ that 
\[
\lim_{N \to \infty} \frac{A([a,b);N;\omega)}{N}=b-a.
\]
It can be shown \cite[p. 3, Corollary 1.1]{kuipers} that a sequence $x_n$ is uniformly distributed modulo $1$ if and only if for every 
Riemann integrable function $f:[0,1] \to \mathbb{R}$ we have 
\begin{equation}
\lim_{N \to \infty} \frac{1}{N} \sum_{n=1}^N f(R(x_n))= \int_0^1 f(t) dt.
\label{riemann}
\end{equation}
Thus, if a sequence is uniformly distributed then the integral of any Riemann integrable function on $[0,1]$ can be approximated by sampling according to this sequence.
This approximation can be quantified using the notion of discrepancy.

It can be proved \cite[p.~7, Theorem 2.1]{kuipers} that a sequence $x_n$ is uniformly distributed modulo $1$ if and only if 
for all nonzero integers $h$ we have
\[
\lim_{N \to \infty} \frac{1}{N} \sum_{n=1}^N e^{2\pi ihx_n}=0;
\]
this is called \textbf{Weyl's criterion}.
But if $x \in \Omega$, then
\begin{equation}
\left| \sum_{n=1}^N e^{2\pi i hnx} \right| = \left| \frac{1-e^{2\pi i hNx}}{1-e^{2\pi ihx}} \right|
\leq \frac{2}{|1-e^{2\pi ihx}|}=\frac{1}{|\sin \pi h x|}.
\label{irratunif}
\end{equation}
We thus obtain the following theorem.

\begin{theorem}
If $x \in \Omega$ then the sequence $nx$ is uniformly distributed modulo $1$.
\label{uniformtheorem}
\end{theorem}

The \textbf{discrepancy} of a sequence $\omega$ is defined, for $N$ a positive integer, by
\[
D_N(\omega)=\sup_{0 \leq a<b \leq 1} \left| \frac{A([a,b);N;\omega)}{N}-(b-a)\right|
\]
One proves that the sequence $\omega$ is uniformly distributed modulo $1$ if and only if
$D_N(\omega) \to 0$ as $N \to \infty$ \cite[p.~89, Theorem 1.1]{kuipers}.

For $f:[0,1] \to \mathbb{R}$, let $V(f)$ denote the total variation of $f$.
Koksma's inequality \cite[p.~143, Theorem 5.1]{kuipers} states that for any sequence $\omega=(x_n)$,  
for any $f:[0,1] \to \mathbb{R}$ of bounded variation, and for any positive integer $N$, we have
\begin{equation}
\left| \frac{1}{N} \sum_{n=1}^N f(R(x_n))-\int_0^1 f(t) dt \right| \leq V(f) D_N(\omega).
\label{koksma}
\end{equation}

Following Kuipers and Niederreiter \cite[p.~122, Lemma 3.2]{kuipers}, we can bound the discrepancy of the sequence $nx$ in terms of the sum on the left-hand side
of  Theorem \ref{hhalpha}

\begin{lemma}
There is some $C>0$ such that for all  $x \in \Omega$,  $\omega=(nx)$, and for all positive integers $m$
we have
\[
D_N(\omega) < C\left(\frac{1}{m}+\frac{1}{N}\sum_{j=1}^m \frac{1}{j \norm{jx}}\right).
\]
\label{erdosturan}
\end{lemma}
\begin{proof}
We shall use the following inequality, which lets us bound the discrepancy of a sequence in terms of exponential sums formed from the elements of the sequence.
The Erd\H os-Tur\'an theorem \cite[p.~114, Eq. 2.42]{kuipers} states that there is some constant $C>0$ such that for any sequence
 $\omega=(x_n)$ of real numbers, any positive integer $N$, and any positive integer $m$ we have
\begin{equation}
 D_N(\omega) \leq C\left(\frac{1}{m}+\sum_{j=1}^m \frac{1}{j} \left| \frac{1}{N} \sum_{n=1}^N e^{2\pi i j x_n} \right| \right).
 \label{erdosturantheorem}
\end{equation}

Take $x_n=nx$. For each $j \geq 1$, by \eqref{irratunif} we have
\[
\left| \sum_{n=1}^N e^{2\pi i jnx} \right| 
\leq \frac{1}{| \sin \pi j x|}
=\frac{1}{\sin(\pi \norm{jx})}.
\]
But $\sin t \geq \frac{2}{\pi}t$ for $0 \leq t \leq \frac{\pi}{2}$, so 
\[
\frac{1}{\sin(\pi \norm{jx})} \leq \frac{1}{2\norm{jx}}<\frac{1}{\norm{jx}}.
\]
Using this in \eqref{erdosturantheorem} gives us
\[
D_N(\omega) < C\left(\frac{1}{m}+\sum_{j=1}^m \frac{1}{j} \cdot \frac{1}{N} \cdot \frac{1}{\norm{jx}}\right)
= C\left(\frac{1}{m}+\frac{1}{N} \sum_{j=1}^m \frac{1}{j\norm{jx}}\right),
\]
which is the claim.
\end{proof}

It follows from Theorem \ref{hhalpha} and Lemma \ref{erdosturan} (taking $m=N$) that if $x$ is of type $<K(\log h)^{1+\epsilon}$ then
\begin{equation}
D_N(\omega)=O\left( \frac{(\log N)^{2+\epsilon}}{N} \right).
\label{DNepsilon}
\end{equation}
Lemma \ref{aatype} tells us that for  almost all $x \in \Omega$ there is some $K>0$ such that $x$ is of type $<K(\log h)^{1+\epsilon}$, so for almost
all $x \in \Omega$, the bound \eqref{DNepsilon} is true. 
It likewise  follows that if $x$ has bounded partial quotients then 
\begin{equation}
D_N(\omega)=O\left( \frac{(\log N)^2}{N} \right).
\label{DNbounded}
\end{equation}
In fact, it can be proved that if $x \in \mathcal{B}_K$ then, for $g = \frac{1+\sqrt{5}}{2}$ \cite[p.~125, Theorem 3.4]{kuipers},
\[
D_N(\omega) \leq  3N^{-1} + \left( \frac{1}{\log g} + \frac{K}{\log(K+1)} \right) N^{-1} \log N.
\]

We use the above bounds in the proof of the following theorem.

\begin{theorem}
Let $\epsilon>0$. For almost all $x$ we have
\[
\sum_{n=1}^N \norm{nx}=\frac{N}{4}+O((\log N)^{2+\epsilon}),
\]
while if $x$ has bounded partial quotients then
\[
\sum_{n=1}^N \norm{nx}=\frac{N}{4}+O((\log N)^2).
\]
\label{normsum}
\end{theorem}
\begin{proof}
Let $f(t)=\norm{t}$. Then $V(f)=1$ and $\int_0^1 f(t) dt=\frac{1}{4}$, so we get from Koksma's inequality \eqref{koksma} that
\[
\left| \frac{1}{N} \sum_{n=1}^N \norm{nx} - \frac{1}{4} \right| \leq D_N(\omega),
\]
thus
\[
\sum_{n=1}^N \norm{nx}=\frac{N}{4}+O(ND_N(\omega)).
\]
The claims then follow respectively from \eqref{DNepsilon} and \eqref{DNbounded}.
\end{proof}

Like we mentioned at the beginning of \S \ref{reciprocalsection},
because $|\sin(\pi x)|=\sin(\pi \norm{x}) \leq \pi \norm{x}$ and $|\sin(\pi x)| = \sin(\pi \norm{x}) \geq \frac{2}{\pi} \cdot \pi \norm{x}=2\norm{x}$, we have
\[
2\sum_{n=1}^N \norm{nx} \leq   \sum_{n=1}^N |\sin(\pi nx)|  \leq \pi \sum_{n=1}^N \norm{nx}.
\]
Thus Theorem \ref{normsum} also gives estimates for  $\sum_{n=1}^N |\sin(\pi nx)|$.

We can investigate the sum $\sum_{n=1}^N R(nx)$ rather than $\sum_{n=1}^N \norm{nx}$; see Lang \cite[p.~37, Theorem 1]{MR0209227}, who proves that
for almost all $x \in \Omega$,
\[
\sum_{n=1}^N R(nx)=\frac{N}{2}+O((\log N)^{2 + \epsilon}).
\]

For $x \in \Omega$, let $q_n=q_n(x)$, the denominator of the $n$th convergent of the continued fraction expansion of $x$,
 and let $a_n=a_n(x)$, the $n$th partial quotient of the continued fraction expansion of $x$. For $m \geq 1$, one can prove
 \cite[p.~211, Proposition~1]{berthe} that $m$ can be written in one and only one way in the form
\begin{equation}
 m=\sum_{k=1}^\infty z_k q_{k-1}=\sum_{k=1}^t z_k q_{k-1},
 \label{ostexpansion}
\end{equation}
 where (i) $0 \leq z_1 \leq a_1-1$, (ii) $0 \leq z_k \leq a_k$ for $k \geq 2$, (iii) for $k \geq 1$, if $z_{k+1}=a_{k+1}$ then
 $z_k=0$, and (iv) $z_t \neq 0$ and $z_k=0$ for $k>t$. The expression \eqref{ostexpansion} is called the \textbf{Ostrowski expansion} of $m$. We emphasize that
 this expansion depends on $x$. Berth\'e \cite{berthe} surveys applications of this numeration system in combinatorics.
 For $n \geq 0$, define $d_{2n}=q_{2n}x-p_{2n}$ and $d_{2n+1}=p_{2n+1}-q_{2n+1}x$. 
 Brown and Shiue \cite[p.~184, Theorem~1]{MR1316813} prove that for $x \in \Omega$,
\begin{equation}
\sum_{k=1}^m \left(R(kx)-\frac{1}{2}\right)=\sum_{k=1}^t (-1)^k z_k\left( \frac{1}{2}-d_{k-1}\left(m_{k-1}+\frac{1}{2}z_k q_{k-1}+\frac{1}{2}\right) \right),
\label{Cxformula}
\end{equation}
where $m_0=0$ and if $k \geq 1$ then $m_k=\sum_{j=1}^k z_j q_{j-1}$.
If $k \geq 0$, then by \eqref{upperbound} we have $0<d_k<\frac{1}{q_{k+1}}$.
For $k \geq 1$, using the fact that $q_k \geq m_k+1$ (for the same reason that if the highest power of $2$ appearing in a number's binary expansion is $2^{k-1}$, then the number
is $\leq 2^k-1$),
\begin{eqnarray*}
m_{k-1}+\frac{1}{2}z_k q_{k-1}+\frac{1}{2}&=&m_k-\frac{1}{2}z_kq_{k-1} + \frac{1}{2}\\
&\leq&q_k-1-\frac{1}{2}z_kq_{k-1} + \frac{1}{2}\\
&=&q_k-\frac{1}{2}-\frac{1}{2}z_kq_{k-1}\\
&<&q_k.
\end{eqnarray*} 
Using \eqref{Cxformula}, this inequality, and the inequality $0<d_k<\frac{1}{q_{k+1}}$, we obtain
\begin{eqnarray*}
\left| \sum_{k=1}^m \left(R(kx)-\frac{1}{2}\right) \right| &\leq& \sum_{k=1}^t z_k \left|\frac{1}{2}-d_{k-1}\left(m_{k-1}+\frac{1}{2}z_k q_{k-1}+\frac{1}{2}\right) \right|\\
&<&\frac{1}{2}\sum_{k=1}^t z_k\\
&<&\frac{1}{2}\sum_{k=1}^t a_k.
\end{eqnarray*}
If the continued fraction expansion of $x$ has bounded partial quotients, say $a_k \leq K$ for all $k$, we obtain from the above that
\[
\left| \sum_{k=1}^m \left(R(kx)-\frac{1}{2}\right) \right| < \frac{Kt}{2}.
\]
It can be proved \cite[p.~185, Fact~2]{MR1316813} that $t<3\log m$. Thus, if $a_k \leq K$ for all $k$, then for $m \geq 1$,
\[
\left| \sum_{k=1}^m \left(R(kx)-\frac{1}{2}\right) \right| < \frac{3K\log m}{2}.
\]
This is Lerch's claim stated in \S \ref{background}.
For example, if  $x=\frac{-1+\sqrt{5}}{2} \in \Omega$ then $a_k(x)=1$ for all $k \geq 1$. We compute that 
\[
\sum_{k=1}^{1000000} \left(R(kx)-\frac{1}{2}\right)=0.941799\ldots;
\]
on the other hand, we compute that
\[
\sum_{k=1}^{1000000} \left( R(k\pi)-\frac{1}{2} \right)=19.223414\ldots.
\]

Brown and Shiue \cite[p.~185, Fact~1]{MR1316813}  use \eqref{Cxformula} to obtain the result of Sierpinski stated in \S \ref{background} that
for all $x \in \Omega$,
\[
\sum_{k=1}^m R(kx)=\frac{m}{2}+o(m).
\]
They also prove \cite[p.~188, Theorem~4]{MR1316813} that for $A>0$, there exists some $d_A>0$ such that for $x \in \Omega$, if there
are infinitely many  $t$ such that $\sum_{k=1}^t a_k \leq At$ (which happens in particular if $x$ has bounded partial quotients), then there are infinitely many $m$ such that
\[
\sum_{k=1}^m \left(R(kx)-\frac{1}{2}\right)>d_A \log m,
\]
and there are infinitely many $m$ such that
\[
\sum_{k=1}^m \left(R(kx)-\frac{1}{2}\right)< -d_A \log m.
\]

It can be shown \cite[p.~44, Theorem 4]{MR0209227} that if $k$ is a positive integer 
and $\epsilon>0$, then
for almost all $x$ we have 
\[
\left| \sum_{n=1}^N e^{2\pi i n^k x} \right| = O\left( N^{\frac{1}{2}+\epsilon}\right).
\]
Lang attributes this result to Vinogradov. But it is not so easy to obtain a bound on this exponential sum for specific $x$. For $k=2$, one can prove
\cite[p.~45, Lemma]{MR0209227}  that for any $x \in \Omega$,
\[
\left| \sum_{n=1}^N e^{2\pi i n^2 x} \right|^2 \leq N + 4\sum_{n=1}^N \frac{1}{|\sin 4\pi n x|}<
N+4\sum_{n=1}^{4N}  \frac{1}{|\sin \pi n x|};
\]
cf. Steele \cite[Problem 14.2]{steele}.
By \eqref{sinebound} this gives us
\[
\left| \sum_{n=1}^N e^{2\pi i n^2 x} \right|^2 < N+2\sum_{n=1}^{4N} \frac{1}{\norm{jx}}.
\]
If $x$ has bounded partial quotients, it follows from Theorem \ref{reciprocaltheorem} that
\[
\left| \sum_{n=1}^N e^{2\pi i n^2 x} \right| = O\left(N^{1/2}(\log N)^{1/2}\right).
\]

Hardy and Littlewood \cite[p.~28, Theorem B5]{latticeI} prove  that if $x \in \Omega$ is an algebraic number, then
there is some $0<\alpha(x)<1$ such that $\sum_{n=1}^N R(nx) = \frac{N}{2} + O(N^{\alpha(x)})$.
Pillai \cite{pillai1}  gives a different proof of this. 

\begin{theorem}[Hardy and Littlewood, Pillai]
For $\tau>2$, if $x \in \mathcal{D}(\tau)$ then for $\alpha = \frac{\tau-2}{\tau-1}$,
\[
\sum_{n=1}^N R(nx) = \frac{N}{2} + O(N^\alpha).
\]
\end{theorem}

Pillai \cite{pillai2} proves other identities and inequalities for  $\sum_{n=1}^N R(nx)$, some for all $x \in \Omega$ and some for all algebraic $x \in \Omega$.

For $\omega=(x_n)$, $n \geq 1$ and for $E \subset [0,1)$, we remind ourselves that
$A(E;M;\omega)$ denotes the number of $x_n$, $1 \leq n \leq M$, such that $R(x_n) \in E$.
Define for $M \geq 1$,
\[
D_M^*(\omega) = \sup_{0<\beta \leq 1} \left| \frac{A([0,\beta);M;\omega)}{M}-\beta \right|.
\]
It is straightforward to prove that $D_M^* \leq D_M \leq 2D_M^*$ \cite[p.~91, Theorem 1.3]{kuipers}.
For $N \geq 1$ write $r=\log N$. Let $\epsilon_0>0$ and let
$N^{1/2} \leq \tau \leq N \exp(-r^{\epsilon_0})$. Suppose that $\alpha \in \mathbb{R}$ and 
\[
\alpha = \frac{a}{q} + \frac{\theta}{q\tau},\quad \gcd(a,q)=1,\quad \exp(r^{\epsilon_0}) \leq q \leq \tau,\quad |\theta| \leq 1.
\]
For $0<\beta<1$ denote by $H_\beta(N)$  the number of primes $p \leq N$ such that $R(\alpha p) \leq \beta$. 
Vinogradov \cite[p.~177, Chapter XI, Theorem]{vinogradov} proves that for $\epsilon>0$,
\[
H_\beta(N) = \beta \pi (N) + O(N(q^{-1}+qN^{-1})^{\frac{1}{2}-\epsilon} + N^{\frac{4}{5}+\epsilon}),
\qquad N \to \infty.
\]
Let $\omega=(p_n \alpha)$, $n \geq 1$, where $p_n$ is the $n$th prime. Using Vinogradov's estimate,
one proves that for $\alpha \in \mathbb{R} \setminus \mathbb{Q}$,
$D_N^*(\omega) \to 0$ as $N \to \infty$, which implies that the sequence $(p_n \alpha)$ is uniformly distributed modulo $1$.
A clean proof of this is given by Pollicott \cite[p.~200, Theorem 1]{pollicott},
and this is also
proved by Vaaler \cite{vaaler} using a Tauberian theorem. See also the early survey by 
Hua \cite[pp.~98--99, \S 38]{hua}.

Defining $S_\alpha(n)=\sum_{k=1}^n \left(R(k\alpha)-\frac{1}{2}\right)$, 
 Beck \cite[p.~14, Theorem 3.1]{beck} proves that there is some $c>0$ such that for every
 $\lambda \in \mathbb{R}$, 
 \[
 \frac{1}{N} \left| \left\{ 1 \leq n \leq N: \frac{S_{\sqrt{2}}(n)}{c\sqrt{\log n}} \leq \lambda 
 \right\} \right| \to \frac{1}{\sqrt{2\pi}} \int_{-\infty}^\lambda \exp\left(-\frac{u^2}{2}\right) du
 \]
 as $N \to \infty$. 
Beck \cite[p.~20, Theorem 1.2]{beck2014} further proves that if $\alpha$ is a quadratic irrational, there are
 $C_1=C_1(\alpha) \in \mathbb{R}$ and $C_2=C_2(\alpha) \in \mathbb{R}_{>0}$ such that for $A,B \in \mathbb{R}$, $A<B$,
\[
\begin{split}
&\frac{1}{N} \left| \left\{ 0 \leq n < N : A \leq \frac{S_\alpha(n)-C_3 \log N}{C_4 \sqrt{\log N}} \leq B\right\}\right|\\
=&(2\pi)^{-1/2} \int_A^B e^{-u^2/2} du + O((\log N)^{-1/10} \log \log N).
\end{split}
\]

Beck and Chen \cite{irregularities}

\section{Dirichlet series}
\label{dirichlet}
The result of de la Vall\'ee-Poussin \cite{valleepoussin} stated in \S \ref{background} implies that there is no $s$ such that for all irrational $x$ the Dirichlet series
$\sum_{n=1}^\infty \frac{n^{-s}}{|\sin n\pi x|}$ converges. It follows from this fact that there is no $s$ such that for all irrational $x$ the Dirichlet series
$\sum_{n=1}^\infty \frac{n^{-s}}{\norm{nx}}$ converges.

Lerch \cite{comptes} in 1904 gives some statements without proof about the series
\[
\sum_{\nu=1}^\infty \frac{\cot \nu \omega \pi}{(2\nu \pi)^{2m+1}}.
\]
He states that if $\omega$ is a real algebraic number that does not belong to $\mathbb{Q}$, then for sufficiently
large $m$ this series converges, and states, for example, that
with $\omega=\frac{1+\sqrt{5}}{2}$,
\[
\sqrt{5} \sum_{\nu=1}^\infty \frac{\cot \nu \omega \pi}{(2\nu \pi)^7} = -\frac{8}{10!}.
\]
Writing $c_r(\theta)=\sum_{n=1}^\infty \frac{\cot(\pi n \theta)}{n^{2r-1}}$, Berndt \cite[p.~135, Theorem 5.1]{berndt}
proves that if $\theta$ is a real algebraic number of degree $d$ and $1<d<2r-1$ (to say that $d>1$ is to say that $\theta$ is irrational), then
$c_r(\theta)$ converges.

For a Dirichlet series $\sum_{n=1}^\infty a_n n^{-s}$, one can show \cite[pp.~289--290, \S 9.11]{titchmarsh} that if the series is convergent at $s=\sigma_0+it_0$, then for $\sigma>\sigma_0$ and any  $t$ the series is convergent at $s=\sigma+it$. It follows that there is some $\sigma_0 \in [-\infty,\infty]$ such if $\sigma<\sigma_0$ then the series diverges at $s=\sigma+it$, and if $\sigma>\sigma_0$ then the series converges at $s=\sigma+it$. We call $\sigma_0$ the \textbf{abscissa of convergence} of the Dirichlet series. If each $a_n$ is a nonnegative real number, then 
the function
\[
F(s)=\sum_{n=1}^\infty a_n n^{-s}, \qquad \Re s>\sigma_0,
\]
cannot be analytically continued to any domain that includes $s=\sigma_0$ \cite[p.~101, Proposition~18]{schoiss}.

Let $a_n$ be a sequence of complex numbers.
It can be  shown \cite[pp.~292--293, \S 9.14]{titchmarsh} that if $s_n=a_1+\ldots+a_n$ and the sequence $s_n$ diverges, then the abscissa of convergence of the Dirichlet series $\sum_{n=1}^\infty a_n n^{-s}$ is given by
\[
\sigma_0=\limsup_{n \to \infty} \frac{\log |s_n|}{\log n}.
\]
By Theorem \ref{reciprocaltheorem} and Theorem \ref{omegabound} (taking, say, $\epsilon=1$), for almost all $x \in \Omega$ there are $C_1,C_2$ such that for all positive integers $n$ we have
\[
C_1 n \log n< \sum_{j=1}^n \frac{1}{\norm{jx}} < C_2 n(\log n)^2.
\]
Thus, if $a_n=\frac{1}{\norm{nx}}$, then
\[
\frac{\log C_1}{\log n} + 1 + \frac{\log \log n}{\log n} < \frac{\log s_n}{\log n} < \frac{\log C_2}{\log n} + 1 + \frac{2 \log \log n}{\log n},
\]
and hence
\[
\lim_{n \to \infty} \frac{\log s_n}{\log n}=1.
\]
It follows that for almost all $x \in \Omega$ the abscissa of convergence of the Dirichlet series $\sum_{n=1}^\infty \frac{1}{\norm{nx}} n^{-s}$ is $\sigma_0=1$.

Likewise, by Theorem \ref{normsum} (taking $\epsilon=1$), we get for almost all $x \in \Omega$ that $\sum_{j=1}^n \norm{jx}=\frac{n}{4}+O\left((\log n)^3\right)$.
We can then check that $\lim_{n \to \infty} \frac{\log s_n}{\log n}=1$, and hence that the abscissa of convergence of the Dirichlet series
$\sum_{n=1}^\infty \norm{nx} n^{-s}$ is $\sigma_0=1$.

A 1953 result of Mahler \cite[pp. 107--108]{feldman} implies that if $\alpha \in \mathbb{R}$ is
an algebraic number of degree $d$, then, for $m=[20\cdot 2^{\frac{5(d-1)}{2}}]$, the Dirichlet series
\[
\sum_{n=1}^\infty \frac{n^{-s}}{\sin n\alpha}
\]
has abscissa of convergence $\sigma_0 \leq d(m+1)\log(m+1)$, and the power series
\[
\sum_{n=1}^\infty \frac{z^n}{\sin n\alpha}
\]
has radius of convergence 1.

Rivoal \cite{MR2928508} presents later work on similar Dirichlet series. See also Queff\'elec and Queff\'elec \cite{queffelec}.
Lal\'in, Rodrigue and Rogers \cite{secant} prove results about Dirichlet series of the form
$\sum_{n=1}^\infty \frac{n^{-s}}{\cos(n \pi z)}$.
Duke and Imamo{\=g}lu \cite{MR2127435} review Hardy and Littlewood's work on estimating lattice points in triangles, and 
prove results about lattice points in cones.

For $\theta \in \Omega$,  write
\[
R_{r,\theta}(\zeta,m)=\sum_{n=0}^m \frac{1}{r!} P_r(\zeta+n\theta),
\]
where $P_r$ is a periodic Bernoulli function.
Spencer \cite{spencer} proves that for any $\epsilon>0$ and almost all $\theta \in \Omega$, 
\[
R_{1,\theta}(\zeta,m) = O\left( (\log m)^{1+\epsilon}\right).
\]
Another result Spencer proves in this paper is that if $q_n(\theta) = O(q_{n-1}^h)$, then $R_{r,\theta}(\zeta,m)=O(m^{1-\frac{r}{h}})$ for $1 \leq r < h$.
Schoi{\ss}engeier \cite{MR1012966} gives an explicit formula for $\sum_{k=0}^{N-1} P_2(k\alpha)$.

\section{Power series}
\label{powersection}
For a power series $\sum a_n z^n$ with radius of convergence $0 \leq R \leq \infty$, the \textbf{Cauchy-Hadamard formula} \cite[p.~111,
Chapter 4, \S 1]{remmert}
states
\begin{equation}
R =\frac{1}{\limsup_{n \to \infty} |a_n|^{1/n}} =  \liminf_{n \to \infty} |a_n|^{-1/n}.
\label{cauchyhadamard}
\end{equation}
The radius of convergence $R$ is equal to the supremum of those $t \geq 0$ for which 
 $|a_n| t^n$ is a bounded sequence.

\begin{lemma}
If $x \in \Omega$, then
the power series
\[
\sum_{n=1}^\infty \frac{z^n}{\norm{nx}}
\]
and 
\[
\sum_{n=1}^\infty \frac{z^n}{|\sin \pi n x|},
\]
have the same radius of convergence.
\label{sameradius}
\end{lemma}
\begin{proof}
The radii of convergence of these power series are respectively
\[
\liminf_{n \to \infty} \norm{nx}^{1/n} \quad \textrm{and} \quad  \liminf_{n \to \infty} |\sin(\pi n x)|^{1/n}.
\]
On the one hand, 
\[
|\sin(\pi nx)|=\sin(\pi \norm{nx}) \leq \pi \norm{nx}
\]
Therefore, since $\lim_{n \to \infty} \pi^{1/n}=1$,
\[
\liminf_{n \to \infty} |\sin(\pi nx)|^{1/n} \leq \liminf_{n\to \infty} \left(\pi\norm{nx}\right)^{1/n}
=\liminf_{n \to \infty} \norm{nx}^{1/n}
\]

On the other hand, since $\sin t \geq \frac{2}{\pi}t$ for $t \geq 0$,
\[
|\sin(\pi n x)|=\sin(\pi\norm{nx}) \geq \frac{2}{\pi} \pi\norm{nx}=2\norm{nx}.
\]
Therefore, using $\lim_{n \to \infty} 2^{1/n}=1$, we have
\[
\liminf_{n \to \infty} |\sin(\pi n x)|^{1/n} \geq \liminf_{n \to \infty} \left(2\norm{nx}\right)^{1/n}
=\liminf_{n \to \infty} \norm{nx}^{1/n},
\]
showing that the two power series have the same radius of convergence.
\end{proof}

We  show in the following theorem that for almost all $x$, the power series $\sum_{n=1}^\infty \frac{z^n}{\norm{nx}}$ has radius of convergence $1$.

\begin{theorem}
For almost all $x \in \Omega$, the power series 
\begin{equation}
\sum_{n=1}^\infty \frac{z^n}{\norm{nx}}
\label{roc}
\end{equation}
has radius of convergence $1$.
\label{roclemma}
\end{theorem}
\begin{proof}
For $x \in \Omega$, let $R_x$ be the radius of convergence of the power series \eqref{roc}.
We have $0 < \norm{nx} < \frac{1}{2}$, so $\frac{1}{\norm{nx}} > 2$. Therefore $\sum_{n=1}^N \frac{1}{\norm{nx}} \to \infty$ as $N \to \infty$, and so
the power series \eqref{roc} diverges at $z=1$. Therefore $R_x \leq 1$ for all $x \in \Omega$.

We shall use Lemma \ref{khinchin} to get a lower bound on $R_x$ that holds for almost all $x \in \Omega$.
Let $A=\{x \in \Omega: R_x<1\}$,
let $A_m=\{x \in \Omega:R_x<1-\frac{1}{m}\}$, and let
$B_m$ be those $x \in \Omega$ such that $\norm{nx}^{1/n}<1-\frac{1}{m}$ infinitely often.
If $x \in A_m$, then $R_x=\liminf_{n \to \infty} \norm{nx}^{1/n}<1-\frac{1}{m}$, and this implies that there are infinitely many $n$ such that
$\norm{nx}^{1/n}<1-\frac{1}{m}$, so $x \in B_m$, i.e. 
$A_m \subset B_m$.
But let $f_m(n)=\left(1-\frac{1}{m}\right)^n$. Then $\sum_{n=1}^\infty f_m(n)$ converges, since $1-\frac{1}{m}<1$, so, by Lemma \ref{khinchin}, for almost all $x \in \Omega$ there are only
finitely many $n$ such that $\norm{nx} < f_m(n)$. Thus $\mu(B_m)=0$. Hence $\mu(A_m)=0$, and since $A=\bigcup_{m=2}^\infty A_m$ we
get that
\[
\mu(A) \leq \sum_{m=2}^\infty \mu(A_m)=0,
\]
that is, $R_x \geq 1$ for almost all $x \in \Omega$. In conclusion, $R_x=1$ for almost all $x \in \Omega$.
\end{proof}

In fact, we can prove the above theorem using the bounds we obtained in Theorem \ref{reciprocaltheorem}.
By Theorem \ref{reciprocaltheorem}, for almost all $x \in \Omega$ we have that $\sum_{j=1}^m \frac{1}{\norm{j x}}=O(m^2)$. (Here we will merely need the fact that the sum is 
subexponential in $m$.) For such an $x$, take $0<r<1$.
Let $a_n=r^n$, let $s_n=\sum_{j=1}^n \frac{1}{\norm{j x}}$, let $b_1=0$, and let $b_n=s_{n-1}$ for $n  \geq 2$. Using summation by parts,
namely
\[
\sum_{n=1}^N a_n(b_{n+1}-b_n)=a_{N+1}b_{N+1}-a_1 b_1 - \sum_{n=1}^N b_{n+1}(a_{n+1}-a_n),
\]
we get
\[
\sum_{n=1}^N \frac{r^n}{\norm{nx}} = r^{N+1} s_N + \sum_{n=1}^N s_n  r^n (1-r).
\]
Therefore
\[
\sum_{n=1}^N \frac{r^n}{\norm{nx}}  = O\left(r^{N+1} N^2\right) + O\left( \sum_{n=1}^N n^2 r^n \right)=O(1).
\]
Since $\sum_{n=1}^N \frac{r^n}{\norm{nx}}$ is increasing in $N$ (being a sum of positive terms), we obtain that the series
$\sum_{n=1}^\infty \frac{r^n}{\norm{nx}}$ converges. Since this is true for all $r$ with $0<r<1$, it follows that $R_x \geq 1$.

For $x \in \Omega$ let $R_x$ be the radius of convergence of the power series $\sum_{q=1}^\infty \frac{z^q}{\norm{qx}}$. 
We proved in Theorem \ref{roc} that for almost all $x \in \Omega$, $R_x=1$. 

\begin{theorem}
For  $x \in \Omega$, let $R_x$ be the radius of convergence of the power series $\sum_{q=1}^\infty \frac{z^q}{\norm{qx}}$, and let
$a_n=a_n(x)$ and $q_n=q_n(x)$. Then
\[
R_x = \liminf_{n \to \infty} a_{n+1}^{-1/q_n}.
\]
For any $0 \leq R \leq 1$ there is some $x \in \Omega$ such that $R_x=R$. 
\end{theorem}
\begin{proof}
From the Cauchy-Hadamard formula \eqref{cauchyhadamard},
\[
R_x =\liminf_{q \to \infty} \norm{qx}^{1/q}.
\]
Then $R_x \leq \liminf_{n \to \infty} \norm{q_n x}^{1/q_n}$. 
On the one hand, by \eqref{upperbound}, 
$\norm{q_n x}<q_{n+1}^{-1}$, 
and $q_{n+1} = a_{n+1} q_n + q_{n-1} > a_{n+1} q_n$ hence
$\norm{q_n x} < a_{n+1}^{-1} q_n^{-1}$, and using that
$\lim_{n \to \infty} q_n^{-1/q_n}=1$,
\[
R_x \leq \liminf_{n \to \infty} a_{n+1}^{-1/q_n} q_n^{-1/q_n}
=\liminf_{n \to \infty} a_{n+1}^{-1/q_n}.
\]
On the other hand, let $q \geq 2$ and take $q_n \leq q < q_{n+1}$. 
Applying \eqref{lowerbound},
\[
\norm{q_n x} > \frac{1}{q_{n+1}+q_n} 
> \frac{1}{2q_{n+1}} = 
\frac{1}{2(a_{n+1}q_n+q_{n-1})}
> \frac{1}{4a_{n+1} q_n}.
\]
Then applying Theorem \ref{bestapproximation}, and using that $0<\norm{q_n x}<1$ and
$q \geq q_n$,
\[
\norm{q x}^{1/q} \geq \norm{q_n x}^{1/q}
\geq \norm{q_n x}^{1/q_n} >
\left( \frac{1}{4a_{n+1} q_n} \right)^{1/q_n}.
\]
As $(4q_n)^{-1/q_n} \to 1$ as $n \to \infty$, this implies
\[
R_x = \liminf_{q \to \infty} \norm{q x}^{1/q} \geq \liminf_{n \to \infty} a_{n+1}^{-1/q_n}.
\]

For $0<R<1$, let $R=e^{-r}$ for $r>0$. Define $a \in \mathbb{N}^\mathbb{N}$ as follows.
Define $a_1=1$. Suppose for $n \geq 1$ that we have defined $a_1,\ldots,a_n$
and thus $p_1,\ldots,p_n$ and $q_1,\ldots,q_n$. 
Define
$a_{n+1} = [e^{r q_n}]$.
Then 
$a_{n+1}^{1/q_n} \leq e^r$, so $a_{n+1}^{-1/q_n} \geq e^{-r}$. Therefore for $x=v(a)$, $R_x \geq e^{-r}$. 
Now, $e^{rq_n}>1$ so $a_{n+1}=[e^{rq_n}] \geq 2^{-1} e^{rq_n}$. Then $a_{n+1}^{1/q_n} \geq 2^{-1/q_n}
e^r$, hence
\[
R_x = \liminf_{n \to \infty} a_{n+1}^{-1/q_n} \leq  \liminf_{n \to \infty} 2^{1/q_n} e^{-r} = e^{-r}.
\]
We have therefore established that when $x=v(a)$, $R_x = e^{-r}$. 

For $R=0$, define $a \in \mathbb{N}^\mathbb{N}$ by
$a_1=1$ and $a_{n+1} = [e^{nq_n}]$, which satisfies
$a_{n+1} \geq 2^{-1} e^{nq_n}$. For
$x=v(a)$,
\[
R_x = \liminf_{n \to \infty} a_{n+1}^{-1/q_n} \leq 
\liminf_{n \to \infty} 2^{1/q_n} e^{-n} = 0.
\]

For $R=1$, define $a \in \mathbb{N}^\mathbb{N}$ by $a_n=1$ for all $n \geq 1$.
Namely, $v(a) = \frac{-1+\sqrt{5}}{2} \in \Omega$. 
For $x=v(a)$ it is immediate that $R_x \geq 1$.
\end{proof}

Since $R(nx) <1$, of course the power series $\sum_{n=1}^\infty R(nx) z^n$ has radius of convergence $\geq 1$.  
The following result, for which P\'olya and Szeg\H o \cite[p.~280, Part II, No. 168]{polyaI}  cite Hecke, shows in particular that the radius of convergence of this power series is
$\leq 1$ for $x \in \Omega$ and is thus equal to $1$.

\begin{theorem}
For $x \in \Omega$, let
\[
f(z)=\sum_{n=1}^\infty R(nx) z^n,\qquad |z|<1.
\]
We have
\[
\lim_{r \to 1^-} (1-r) f(re^{2\pi i x}) = \frac{1}{2\pi i }.
\]
\end{theorem}
\begin{proof}
Since $x \in \Omega$, the sequence $nx$ is uniformly distributed modulo $1$. Therefore, with
$f(t)=te^{2\pi i t}$ we have by \eqref{riemann} that
\[
\lim_{N \to \infty} \frac{1}{N} \sum_{n=1}^N R(nx) e^{2\pi inx} =
\lim_{N \to \infty} \frac{1}{N} \sum_{n=1}^N f(R(nx))=
 \int_0^1 te^{2\pi i t} dt=\frac{1}{2\pi i}.
\]

We will use the following result \cite[p.~21, Part I, No. 88]{polyaI}.
If a sequence of complex numbers $a_n$  satisfies
$\lim_{N \to \infty} \frac{1}{N} \sum_{n=1}^N a_n=s$, then
\[
\lim_{t \to 1^-} (1-t)\sum_{n=1}^\infty a_n t^n=s.
\]
Let $a_n=R(nx) e^{2\pi inx}$, and we thus have
\[
\lim_{t \to 1^-} (1-t)\sum_{n=1}^\infty R(nx) e^{2\pi inx} t^n=\frac{1}{2\pi i}.
\]
\end{proof}

It follows from the above theorem that if $x \in \Omega$ then
$|z|=1$ is a natural boundary of the function $f$ defined on the open unit disc by $f(z)=\sum_{n=1}^\infty R(nx)z^n$; cf. Segal \cite[p.~255, Chapter 6]{segal}, who writes about this
power series, and who gives a thorough introduction to natural boundaries in the same chapter.
Breur and Simon \cite{breuer} prove a generalization of this result.

Hata \cite[p.~173, Problem 12.6]{hata} 
mentions the appearance of the function  $f$ from the above theorem in the study of the Caianiello neuron equations.

\section{Product}
We will use the following lemma proved by Hardy and Littlewood \cite[p.~89]{XXIV},  whose brief proof we expand.

\begin{lemma}
Let $\psi:(0,\infty) \to \mathbb{R}$ be positive and nondecreasing.
If 
\[
\sum_{k=1}^\infty \frac{1}{k\psi(k)}<\infty,
\]
then
for almost all $x \in \Omega$, there exists some $H$ such that for all $n \geq 1$ and for all real $h \geq H$, there are at most $\max\{\frac{n\psi(h)}{h},1\}$ integers $m \in \{1,\ldots,n\}$
that satisfy $\norm{mx} < \frac{1}{h}$.
\label{hllemma}
\end{lemma}
\begin{proof}
By Lemma~\ref{khinchin},
for almost all $x \in \Omega$ there is some $K$ such that if $k \geq K$ then
\begin{equation}
\norm{kx} \geq \frac{2}{k\psi(k)}.
\label{Kcontradiction}
\end{equation}
Let $H$ be  large enough  so that
\[
\min_{k < K} \norm{kx} \geq \frac{2}{H};
\]
also let $\psi(H) \geq 1$.
Now suppose by contradiction that there is some $n \geq 1$ and some $h \geq H$ such that there are more
than $\max\{\frac{n\psi(h)}{h},1\}$ integers $m \in \{1,\ldots,n\}$
that satisfy $\norm{mx} < \frac{1}{h}$.
Then there are some $1 \leq m_1 < m_2 \leq n$ satisfying $\norm{m_1 x}<\frac{1}{h}$ and $\norm{m_2 x}<
\frac{1}{h}$ and such that
\[
\mu=m_2-m_1< \frac{n}{\frac{n\psi(h)}{h}}=\frac{h}{\psi(h)},
\]
so $\mu \psi(h)<h$. 
On the other hand,
\[
\norm{\mu x} \leq \norm{m_1 x}+\norm{m_2 x}<\frac{1}{h}+\frac{1}{h}=\frac{2}{h},
\]
so $h<\frac{2}{\norm{\mu x}}$. Thus $\mu \psi(h)<\frac{2}{\norm{\mu x}}$, i.e.,
\[
\norm{\mu x}< \frac{2}{\mu \psi(h)} \leq \frac{2}{\mu \psi(\mu)};
\]
$\mu < h$ because $\mu<\frac{h}{\psi(h)}$ and $\psi(h) \geq \psi(H) \geq 1$.
Moreover, since $\norm{\mu x}<\frac{2}{h} \leq  \frac{2}{H}$, we have $\mu \geq K$. This contradicts \eqref{Kcontradiction}.
\end{proof}

Hardy and Littlewood \cite[p.~89, Theorem~4]{XXIV} prove the following theorem
that  gives us the conclusion \eqref{riemann} for certain functions that are not Riemann integrable on $[0,1]$.

\begin{theorem}
Let $f:(0,1) \to \mathbb{R}$ be  nonnegative, let $f$ be nonincreasing on $(0,\frac{1}{2})$ and nondecreasing on $(\frac{1}{2},1)$, and let
\[
\int_0^1 f(t) dt < \infty.
\] 
Let $\psi:(1,\infty) \to \mathbb{R}$ be a positive and nondecreasing function such that
\[
\sum_{k=2}^\infty \frac{1}{k\psi(k)} < \infty.
\]
If
\[
\int_0^1 f(t) \left( \psi\left(\frac{1}{t}\right)+\psi\left(\frac{1}{1-t}\right) \right) dt<\infty,
\]
then for almost all $x \in \Omega$,
\[
\lim_{n \to \infty} \frac{1}{n} \sum_{m=1}^n f(R(mx))= \int_0^1 f(t)dt.
\]
\label{fpsi}
\end{theorem}
\begin{proof}
For $0<\delta<\frac{1}{2}$, define 
\[
f_\delta(t)=\begin{cases}f(t),&\delta\leq t \leq 1-\delta,\\
0,&0<t < \delta \quad \textrm{or} \quad 1-\delta < t<1.
\end{cases}
\]
From Lemma~\ref{hllemma}, for almost all $x \in \Omega$ there is some $C$ such that for all $n$ and $h$ there are at most
$C\frac{n\psi(h)}{h}$ integers $m \in \{1,\ldots,n\}$ satisfying $\norm{mx} < \frac{1}{h}$.
Let
\[
S_n=\frac{1}{n}\sum_{m=1}^n f(R(mx))=S_n^1(\delta)+S_n^2(\delta),
\]
where
\[
S_n^1(\delta)=\frac{1}{n}\sum_{m=1}^n f_\delta(R(mx))
\]
and
\begin{eqnarray*}
S_n^2(\delta)&=&\frac{1}{n}\sum_{m=1}^n f(R(mx))-f_\delta(R(mx))\\
&=&\frac{1}{n}\sum_{\stackrel{1 \leq m \leq n}{\norm{mx} < \delta}} f(R(mx))\\
&=&\frac{1}{n}\sum_{k=0}^\infty \sum_{\stackrel{1 \leq m \leq n}{\frac{\delta}{2^{k+1}} \leq \norm{mx} < \frac{\delta}{2^k}}} f(R(mx))\\
&=&\frac{1}{n} \sum_{k=0}^\infty T_{k,n}(\delta).
\end{eqnarray*}
There are at most $C\frac{n\psi\left( \frac{2^k}{\delta} \right)}{\frac{2^k}{\delta}}$ integers $m \in \{1, \ldots, n\}$ that satisfy 
$\norm{mx} < \frac{\delta}{2^k}$, thus, as $\psi$ is nondecreasing, there are at most $2C\frac{n\delta\psi\left( \frac{2^k}{\delta} \right)}{2^{k+1}} \leq 2C \frac{n\delta \psi\left( \frac{2^{k+1}}{\delta} \right)}{2^{k+1}}$ 
terms in $T_{k,n}(\delta)$. 
For each term $f(R(mx))$ in $T_{k,n}(\delta)$, since $\frac{\delta}{2^{k+1}} \leq \norm{mx}$ we have, by assumption on $f$, either $f(R(mx)) \leq f\left( \frac{\delta}{2^{k+1}} \right)$ or
$f(R(mx)) \leq f\left(1-  \frac{\delta}{2^{k+1}} \right)$, and hence 
\[
f(R(mx)) \leq f\left( \frac{\delta}{2^{k+1}} \right) + f\left(1-  \frac{\delta}{2^{k+1}} \right).
\]
Therefore,
\begin{eqnarray*}
S_n^2(\delta)&\leq&\frac{1}{n}\sum_{k=0}^\infty  2C \frac{n\delta \psi\left( \frac{2^{k+1}}{\delta} \right)}{2^{k+1}} \left(f\left( \frac{\delta}{2^{k+1}} \right) + f\left(1-  \frac{\delta}{2^{k+1}} \right) \right)\\
&=&4C\sum_{k=0}^\infty  \frac{\delta}{2^{k+2}} \psi\left( \frac{2^{k+1}}{\delta} \right) f\left( \frac{\delta}{2^{k+1}} \right) 
+4C\sum_{k=0}^\infty \frac{\delta}{2^{k+2}} \psi\left( \frac{2^{k+1}}{\delta} \right) f\left(1-  \frac{\delta}{2^{k+1}} \right)\\
&\leq&4C\int_0^{\frac{\delta}{2}} \psi\left(\frac{1}{t}\right) f(t) dt + 4C\int_{1-\frac{\delta}{2}}^1 \psi\left(\frac{1}{1-t}\right) f(t) dt.
\end{eqnarray*}
Let $\epsilon>0$. Because $\int_0^1 f(t) \left( \psi\left(\frac{1}{t}\right)+\psi\left(\frac{1}{1-t}\right) \right) dt<\infty$, there exists a $\delta_1$ such that if
$\delta \leq \delta_1$ then $S_n^2(\delta)<\epsilon$.

On the other hand, since $f_\delta$ is Riemann integrable on $[0,1]$ and because, by \eqref{irratunif}, the sequence $mx$ is uniformly distributed modulo $1$, we obtain from
\eqref{riemann} that
\[
\lim_{n \to \infty} S_n^1(\delta)=\int_0^1 f_\delta(t) dt=\int_\delta^{1-\delta} f(t)dt.
\]
As $\int_0^1 f(t)dt < \infty$, there exists a $\delta_2$ such that for $\delta \leq \delta_2$ and for sufficiently large $n$,
\[
\left| S_n^1(\delta)-\int_0^1 f(t) dt\right| \leq \left|S_n^1(\delta) - \int_\delta^{1-\delta} f(t) dt \right| +
\left| \int_0^\delta f(t) dt \right| + \left| \int_{1-\delta}^1 f(t) dt \right| < 3\epsilon.
\]
Therefore, for sufficiently large $n$ and for sufficiently small $\delta$,
\[
\left|S_n - \int_0^1 f(t)dt \right| \leq \left|S_n^1(\delta) - \int_0^1 f(t)dt  \right| + \left|S_n^2(\delta) \right|
<
 4\epsilon.
\]
Thus for sufficiently large $n$,
\[
\left|S_n - \int_0^1 f(t)dt \right| \leq 4\epsilon.
\]
\end{proof}

By the Birkhoff ergodic theorem \cite[p.~44, Theorem~2.30]{einsiedler}, if $f \in L^1[0,1]$ and $x \in \Omega$, then for almost all
$\alpha \in [0,1]$,
\[
\lim_{n \to \infty} \frac{1}{n} \sum_{j=1}^n f(R(\alpha+jx))=\int_0^1 f(t) dt.
\]
This equality holding for $\alpha=0$ is the conclusion of Theorem~\ref{fpsi}. 

Baxa \cite{MR2189068} reviews further results that give conditions when a function $f:[0,1] \to \mathbb{R} \cup \{+\infty\}$ that is not Riemann integrable on $[0,1]$ nevertheless satisfies
\[
\lim_{n \to \infty} \frac{1}{n} \sum_{m=1}^n f(R(mx))= \int_0^1 f(t)dt
\]
for certain $x \in \Omega$. Oskolkov \cite[p.~170, Theorem~1]{MR1044053}  shows that if $f:(0,1) \to \mathbb{R}$ satisfies
$\lim_{t \to 0+} f(t)=+\infty$ and $\lim_{t \to 1-} f(t)=+\infty$, and also the improper Riemann integral of $f$ on $[0,1]$ exists, then, for $x \in \Omega$,
\[
\lim_{n \to \infty} \frac{1}{n} \sum_{m=1}^n f(R(mx))= \int_0^1 f(t)dt
\]
if and only if 
\[
\lim_{n \to \infty} \frac{1}{q_n(x)} f(R(q_n(x) x))=0,
\]
where $q_n(x)$ is the denominator of the $n$th convergent of the continued fraction expansion of $x$.

Driver, Lubinsky, Petruska and Sarnak \cite{sarnak}

Using Theorem \ref{fpsi} we can now prove the following theorem of Hardy and Littlewood \cite[p.~88, Theorem~2]{XXIV}.

\begin{theorem}
For almost all $x \in \Omega$, 
\[
\lim_{n \to \infty} \left( \prod_{k=1}^n |\sin k\pi x| \right)^{1/n}=\frac{1}{2}.
\]
\label{bigtheorem}
\end{theorem}
\begin{proof}
Let $f(t)=-\log \sin \pi t$. Using $\cos(t-\frac{\pi}{2})=\sin t$ and $\sin 2t=2\sin t \cos t$, one can check that  $\int_0^1 \log \sin \pi t dt=-\log 2$. (The earliest evaluation of this integral of which we are aware
is by Euler \cite{E393}, who gives two derivations, the first using the Euler-Maclaurin summation formula, the power
series expansion for $\log\left( \frac{1+z}{1-z} \right)$, and the power series expansion of $z\cot(z)$,
and the second using the Fourier series of
$\log | \sin t|$.) Thus, $\int_0^1 f(t) dt=\log 2<\infty$. So $f$ satisfies the conditions of Theorem \ref{fpsi}.

Let $\psi(t)=(\log t)^2$. First, upper bounding the series by an integral,
\[
\sum_{k=2}^\infty \frac{1}{k(\log k)^2}\leq \frac{1}{2(\log 2)^2}+\int_2^\infty \frac{1}{t(\log t)^2} dt=
\frac{1}{2(\log 2)^2}+\log 2 <\infty.
\]
Second, 
\begin{eqnarray*}
\int_0^1 f(t) \left( \psi\left(\frac{1}{t}\right)+\psi\left(\frac{1}{1-t}\right) \right) dt&=&\int_0^1 f(t) \left(\psi(t)+\psi(1-t) \right) dt\\
&=&\int_0^{\frac{1}{2}} -2\log \sin (\pi t)\\
&& \left((\log t)^2+(\log(1-t))^2\right) dt \\
&\leq&\int_0^{\frac{1}{2}} -2\log (2t)  \left((\log t)^2+(\log(1-t))^2\right) dt \\
&<&\infty.
\end{eqnarray*}
Therefore by Theorem \ref{fpsi},
for almost all $x \in \Omega$,
\[
\lim_{n \to \infty} \frac{1}{n} \sum_{m=1}^n -\log \sin(\pi R(mx))= \int_0^1 -\log \sin \pi t dt,
\]
i.e.
\[
\lim_{n \to \infty} \frac{1}{n} \sum_{m=1}^n \log |\sin \pi m x|= \log \frac{1}{2}.
\]
\end{proof}

Hardy and Littlewood give another proof \cite[p.~86, Theorem~1]{XXIV} of the above theorem, which we now work out.
This proof is complicated and we greatly expand on the abbreviated presentation of Hardy and Littlewood.

We remind ourselves that the Cauchy-Hadamard formula states that the radius of
convergence $R$ of a  power series $\sum a_n z^n$ satisfies
\[
R =\frac{1}{\limsup_{n \to \infty} |a_n|^{1/n}} =  \liminf_{n \to \infty} |a_n|^{-1/n}.
\]

\begin{theorem}
Fix $x \in \Omega$ and write $q_0=e^{2\pi ix}$. Let $\rho$ and $R$ respectively be   the radii of convergence of the power series
\[
f(z) = \sum_{n=1}^\infty \frac{z^n}{n(1-q_0^n)},\qquad F(z) = 1 + \sum_{n=1}^\infty \frac{z^n}{(1-q_0) (1-q_0^2)\cdots (1-q_0^n)}.
\]
Then $R=\rho$, and if $|z|<\rho$ then $F(z) = e^{f(z)}$. 
\label{efz}
\end{theorem}

The functions $f:D(0,\rho) \to \mathbb{C}$ and $F:D(0,R) \to \mathbb{C}$ are analytic \cite[p.~69, Theorem 2.16]{titchmarsh}.

Using the Cauchy-Hadamard formula we have
\[
\rho = \liminf_{n \to \infty} n^{1/n} |1-q_0^n|^{1/n}
\leq \liminf_{n \to \infty} (2n)^{1/n} = 1
\]
and
\begin{equation}
R = \liminf_{n \to \infty} (|1-q_0| |1-q_0^2| \cdots |1-q_0^n|)^{1/n}
\leq 2\rho.
\label{R2rho}
\end{equation}

\begin{lemma}
For $|u| = 1$ and for $0 \leq r < 1$,
\[
\left| \frac{1-u}{1-ru} \right| \leq 2.
\]
\label{unimodular}
\end{lemma}
\begin{proof}
\[
|1-u| \leq |1-ru| + |ru-u| = |1-ru| + 1-r.
\]
Because $\Re u \leq 1$ we have $1-r \leq 1-r \Re u = \Re(1-ru) \leq |1-ru|$. 
Hence $|1-u| \leq 2|1-ru|$, from which the claim follows.
\end{proof}

We assert the following as a common fact in complex analysis.

\begin{lemma}
For $|w|<1$ define
\[
L(w) = \sum_{n=1}^\infty \frac{w^n}{n}.
\]
If $|w|<1$, then $e^{L(w)} = (1-w)^{-1}$. 
\label{logarithmlemma}
\end{lemma}

For $|q|, |z|<1$, we define
\[
f(z,q) = \sum_{n=1}^\infty \frac{z^n}{n(1-q^n)},\qquad F(z,q) = 1 + \sum_{n=1}^\infty \frac{z^n}{(1-q)(1-q^2)\cdots (1-q^n)}.
\]
For $|q|<1$, define $c_0(q)=0$ and for $n \geq 1$,
\[
c_n(q)=\frac{1}{n(1-q^n)},
\]
and thus for $|z|<1$,
\[
f(z,q) = \sum_{n=0}^\infty c_n(q) z^n.
\]
Furthermore define $\gamma_0=0$ and for $n \geq 1$,
\[
\gamma_n = \frac{1}{n(1-q_0^n)},
\]
and thus for $|z|<\rho$,
\[
f(z) = \sum_{n=0}^\infty \gamma_n z^n.
\]
For $|q|<1$, define $C_0(q)=1$ and for $n \geq 1$,
\[
C_n(q) = \frac{1}{(1-q)(1-q^2) \cdots (1-q^n)},
\]
and thus for $|z|<1$,
\[
F(z,q) = \sum_{n=0}^\infty C_n(q) z^n.
\]
Furthermore define $\Gamma_0=1$ and for $n \geq 1$,
\[
\Gamma_n = \frac{1}{(1-q_0)  (1-q_0^2) \cdots (1-q_0^n)},
\]
and thus for $|z|<R$,
\[
F(z) = \sum_{n=0}^\infty \Gamma_n z^n. 
\]

We prove directly the following, which is an instance of the \textbf{$q$-binomial formula}  \cite[p.~17, Theorem 2.1]{andrews}.

\begin{proposition}
If $|q|, |z|<1$ then
\[
F(z,q) = e^{f(z,q)}.
\]
\label{qbinomial}
\end{proposition}
\begin{proof}
By Lemma \ref{logarithmlemma},
\[
f(z,q)=\sum_{n=1}^\infty \frac{z^n}{n} \left( \sum_{m=0}^\infty q^{nm}\right)
=\sum_{m=0}^\infty \left( \sum_{n=1}^\infty \frac{(zq^m)^n}{n} \right)
=\sum_{m=0}^\infty L(zq^m).
\]
Because $e^{L(w)}=(1-w)^{-1}$ for $|w|<1$,
\[
e^{f(z,q)} = \prod_{m=0}^\infty e^{L(zq^m)}
=\prod_{m=0}^\infty (1-zq^m)^{-1}.
\]
Define
\[
G(z,q) = \prod_{m=0}^\infty (1-zq^m)^{-1} = \sum_{n=0}^\infty g_n(q) z^n.
\]
On the one hand, $g_0(q) = G(0,q) = 1$. On the other hand,
\[
G(qz,q) = (1-z) G(z,q),
\]
thus
\[
\sum_{n=0}^\infty g_n(q) z^{n+1} = \sum_{n=0}^\infty g_n(q)(1-q^n) z^n,
\]
and therefore for $n \geq 1$ we have 
$g_n(q) = (1-q^n)^{-1} g_{n-1}(q)$. 
Thus by induction, for $n \geq 1$,
\[
g_n(q) = \frac{1}{(1-q)(1-q^2)\cdots (1-q^n)}.
\]
Hence
\[
e^{f(z,q)} = 1 + \sum_{n=1}^\infty  \frac{z^n}{(1-q)(1-q^2)\cdots (1-q^n)} = F(z,q).
\]
\end{proof}

\begin{proposition}
$R \geq \rho$, and if $|z| < \rho$ then $e^{f(z)} = F(z)$. 
\label{Rgeqrho}
\end{proposition}
\begin{proof}
If $\rho=0$ then the claim is immediate. Otherwise, $0 < \rho \leq 1$. Let
 $0<t<\rho$, $0<r<1$, and  define $G(\theta) = F(te^{i\theta},rq_0)$.
On the one hand,
\[
G(\theta) = \sum_{n=0}^\infty C_n(rq_0) (te^{i\theta})^n = \sum_{n=0}^\infty C_n(rq_0) t^n e^{in\theta},
\]
which implies that $\widehat{G}(n)=C_n(rq_0) t^n$ for $n \geq 0$ and $\widehat{G}(n)=0$ for $n<0$. 
On the other hand,
for $n \in \mathbb{Z}$, using 
Proposition \ref{qbinomial},
\begin{align*}
\widehat{G}(n) &= \frac{1}{2\pi} \int_0^{2\pi} G(\theta) e^{-in\theta} d\theta\\
&=\frac{1}{2\pi} \int_0^{2\pi} F(te^{i\theta},rq_0) e^{-in\theta} d\theta\\
&=\frac{1}{2\pi} \int_0^{2\pi} e^{f(te^{i\theta},rq_0)} e^{-in\theta} d\theta.
\end{align*}
By Lemma \ref{unimodular},
\[
|f(te^{i\theta},rq_0)| \leq \sum_{n=1}^\infty \frac{t^n}{n|1-r^n q_0^n|}
\leq 2 \sum_{n=1}^\infty \frac{t^n}{n|1-q_0^n|} = M(t),
\]
and because
$t<\rho$ it is the case that $M(t)<\infty$. 
Then for $n \geq 0$,
\[
|C_n(rq_0) t^n| = |\widehat{G}(n)| \leq  \frac{1}{2\pi} \int_0^{2\pi} |e^{f(te^{i\theta},rq_0)}| d\theta
\leq e^{M(t)}.
\]
$C_n(rq_0) \to \Gamma_n$ as $r \to 1$, and therefore
\[
|\Gamma_n t^n| \leq e^{M(t)},\qquad 0<t<\rho,\qquad n \geq 0.
\]
Now fix $|z|<\rho$, take $t$ such that $|z|<t<\rho$, and write $0 \leq \delta=\frac{|z|}{t}<1$
and $M_n=e^{M(t)} \delta^n$.
Because $|\Gamma_n t^n| \leq e^{M(t)}$,
\[
\sum_{n=0}^\infty |\Gamma_n z^n| = \sum_{n=0}^\infty  \delta^n |\Gamma_n t^n| \leq \sum_{n=0}^\infty  M_n
=\frac{e^{M(t)}}{1-\delta}<\infty,
\]
which implies that $|z| \leq R$. Therefore $R \geq \rho$.
Furthermore, because
\[
|C_n(rq_0) z^n| = \delta^n |C_n(rq_0| t^n| = M_n,\qquad r \in (0,1),
\]
by the Weierstrass $M$-test \cite[p.~148, Theorem 7.10]{rudin}, the sequence $\sum_{n=0}^N C_n(rq_0) z^n$ converges uniformly for $r \in (0,1)$ and therefore \cite[p.~149, Theorem 7.11]{rudin}
\begin{align*}
\lim_{r \to 1} F(z,rq_0)&=\lim_{r \to 1} \lim_{N \to \infty} \sum_{n=0}^N C_n(rq_0)z^n\\
&=\lim_{N \to \infty} \lim_{r \to 1}  \sum_{n=0}^N C_n(rq_0)z^n\\
&=\lim_{N \to \infty} \sum_{n=0}^N \Gamma_n z^n\\
&=F(z).
\end{align*}
Now, for $0<r<1$, by Lemma \ref{unimodular},
\[
|c_n(rq_0)z^n| = \left| \frac{z^n}{n(1-r^n q_0^n)} \right|
\leq 2 \left| \frac{z^n}{n(1-q_0^n)} \right|
=m_n.
\]
Because $|z|<\rho$, the series $\sum_{n=0}^\infty m_n$ converges, and therefore by the Weierstrass $M$-test, 
the sequence $\sum_{n=0}^N c_n(rq_0) z^n$ converges uniformly for $r \in (0,1)$. Then
\begin{align*}
\lim_{r \to 1} f(z,rq_0)&=\lim_{r \to 1} \lim_{N \to \infty} \sum_{n=0}^N c_n(rq_0)z^n\\
&=\lim_{N \to \infty} \lim_{r \to 1}  \sum_{n=0}^N c_n(rq_0)z^n\\
&=\lim_{N \to \infty} \sum_{n=0}^N \gamma_n z^n\\
&=f(z).
\end{align*}
Then using Lemma \ref{qbinomial},
\begin{align*}
\exp(f(z))&=\exp\left( \lim_{r \to 1} f(z,rq_0) \right)\\
&=\lim_{r \to 1} \exp(f(z,rq_0))\\
&=\lim_{r \to 1} F(z,rq_0)\\
&=F(z),
\end{align*}
completing the proof.
\end{proof}

\begin{lemma}
If $|u|=1$ and $u \neq 1$ then
\[
\left| \sum_{n=1}^N u^n \right|  \leq \frac{2}{|1-u|}.
\]
\label{geometricbound}
\end{lemma}
\begin{proof}
\[
\left| \sum_{n=1}^N u^n \right| = \left| \frac{u-u^{N+1}}{1-u} \right| = \left| \frac{1-u^N}{1-u} \right| \leq \frac{2}{|1-u|}.
\]
\end{proof}

\begin{lemma}
Let $0<a<1$, let $x \in \Omega$, and suppose that there is some $C$ such that
$\frac{a^n}{|\sin n\pi x|} \leq C$ for all $n$. Then
\[
\sum_{n=1}^N \frac{a^{2n}}{\sin^2 n \pi x} = o(N),\qquad N \to \infty.
\]
\label{axC}
\end{lemma}
\begin{proof}
Take $0<\epsilon<\frac{1}{2}$ and let
\[
E_N = \{n: 1 \leq n \leq N, \norm{nx}<\epsilon\}.
\]
Thus if $1 \leq n \leq N$ and $n \not \in E_N$ then $\norm{nx} \geq \epsilon$. 
Write $S_N=\sum_{n=1}^N \frac{a^{2n}}{\sin^2 n \pi x}$.
Because $|\sin n\pi x| =\sin(\pi \norm{nx}) \geq 2\norm{nx}$,
\begin{align*}
S_N&=\sum_{n \in E_N} \frac{a^{2n}}{\sin^2 n\pi x} + \sum_{n \not \in E_N, 1 \leq n \leq N} 
\frac{a^{2n}}{\sin^2 n \pi x}\\
&\leq C^2 |E_N| + \sum_{n \not \in E_N, 1 \leq n \leq N} \frac{a^{2n}}{4 \epsilon^2}\\
&\leq C^2 |E_N| + \frac{1}{4\epsilon^2(1-a^2)}. 
\end{align*}
Because $x$ is irrational, by Theorem \ref{uniformtheorem} the sequence $nx$ is uniformly distributed modulo $1$.
Therefore
\[
\frac{|E_N|}{N} \to 2\epsilon,\qquad N \to \infty,
\] 
and this implies
\[
\limsup_{N \to \infty} \frac{S_N}{N} \leq 2\epsilon \cdot C^2.
\]
Because this is true for each $0<\epsilon<\frac{1}{2}$ it follows that $\lim_{N \to \infty} \frac{S_N}{N} = 0$, proving the claim.
\end{proof}

\begin{proposition}
$R \leq \rho$.
\label{Rleqrho}
\end{proposition}
\begin{proof}
We have found in \eqref{R2rho} that  $R \leq 2\rho$, i.e. $\rho \geq \frac{R}{2}$.
Assume by contradiction that $R>\rho$; in particular $R>0$, which implies $\rho \geq \frac{R}{2} > 0$. 
By Proposition \ref{Rgeqrho}, we have $F(z)=e^{f(z)}$ for $|z|<\rho$ and so $F(z) \neq 0$ for $z \in D(0,\rho)$. 

Let $u_1 \rho, \ldots, u_s \rho$ be the distinct zeros of $F$ on $|z|=\rho$, with respective multiplicities $p_1,\ldots,p_s$;
if there is none, take $s=0$, and use $\sum_{\emptyset}=0$ and $\prod_{\emptyset} =1$.
Define
\[
G(z) = F(z)  \prod_{j=1}^s \left(1-\frac{z}{u_j \rho}\right)^{-p_j},\qquad z \in D(0,R).
\]
Because
$G:D(0,R) \to \mathbb{C}$ is analytic and
$G(z) \neq 0$ for $z \in \overline{D(0,\rho)}$,  there is some $T$, $\rho<T \leq R$, such that 
$G(z) \neq 0$ for $z \in D(0,T)$ \cite[p.~208, Theorem 10.18]{rudincomplex}. As $D(0,T)$ is simply connected, there is an analytic function $g:D(0,T) \to \mathbb{C}$ such that
$G(z) = e^{g(z)}$ for $z \in D(0,T)$ \cite[p.~274, Theorem 13.11]{rudincomplex}. Thus
\[
F(z) = e^{g(z)}  \prod_{j=1}^s  \left(1-\frac{z}{u_j \rho}\right)^{p_j},\qquad z \in D(0,T).
\]
For $z \in D(0,\rho)$ we have $\left|\frac{z}{u_j \rho} \right| = \frac{|z|}{\rho}<1$, and then by
Lemma \ref{logarithmlemma}, $e^{L\left(\frac{z}{u_j \rho}\right)} = \left(1-\frac{z}{u_j \rho}\right)^{-1}$.
Therefore for $z \in D(0,\rho)$,
\[
e^{f(z)} = F(z) =  e^{g(z)}  \prod_{j=1}^s e^{-p_j L\left(\frac{z}{u_j \rho}\right)}
= \exp\left( g(z) - \sum_{j=1}^s p_j L\left(\frac{z}{u_j \rho}\right)\right),
\]
i.e.
\[
\exp\left(f(z)-g(z) - \sum_{j=1}^s p_j L\left(\frac{z}{u_j \rho}\right)\right)=1,\qquad z \in D(0,\rho).
\]
But the image of a continuous function $D(0,\rho) \to 2\pi i \mathbb{Z}$ is connected, and because 
$2\pi i \mathbb{Z}$ has the discrete topology
it follows that the image is a singleton,  thus
there is some $\nu \in \mathbb{Z}$ such that 
\[
f(z) - g(z)  + \sum_{j=1}^s p_j L\left(\frac{z}{u_j \rho}\right) = 2\pi i\nu,\qquad z \in D(0,\rho).
\]
But $e^{g(0)} = G(0) = F(0) \cdot 1 = 1$, so $g(0)=0$, and $f(0)=0$, hence $\nu=0$. 
Therefore
\[
f(z) = g(z) - \sum_{j=1}^s p_j L\left(\frac{z}{u_j \rho}\right),\qquad z \in D(0,\rho).
\]
Now, for $z \in D(0,\rho)$, 
\[
\sum_{j=1}^s p_j L\left(\frac{z}{u_j \rho}\right) =\sum_{j=1}^s p_j \sum_{n=1}^\infty \frac{\left(\frac{z}{u_j \rho}\right)^n}{n}
=\sum_{n=1}^\infty  \left(  \frac{1}{n} \sum_{j=1}^s p_j (u_j \rho)^{-n} \right) z^n.
\]
Let $g_n = \frac{g^{(n)}(0)}{n!}$. 
For $z \in D(0,\rho)$,
\[
\sum_{n=1}^\infty \gamma_n z^n = \sum_{n=1}^\infty g_n z^n - \sum_{n=1}^\infty  \left(  \frac{1}{n} \sum_{j=1}^s p_j (u_j \rho)^{-n} \right) z^n,
\]
so for $n \geq 1$,
\[
 \frac{1}{n(1-q_0^n)} = \gamma_n = g_n -  \frac{1}{n} \sum_{j=1}^s p_j (u_j \rho)^{-n}.
\]
Then
\begin{equation}
\frac{\rho^n}{1-q_0^n} = n \rho^n g_n - \sum_{j=1}^s p_j u_j^{-n}.
\label{pjsum}
\end{equation}
Cauchy's integral formula \cite[p.~82, Theorem 2.41]{titchmarsh} tells us that for $0<V<T$, $C(t) = V e^{it}$, $0 \leq t \leq 2\pi$,
\[
g_n =  \frac{g^{(n)}(0)}{n!} = \frac{1}{2\pi i} \int_C \frac{g(w)}{w^{n+1}} dw,
\]
whence, as the length of $C$ is $2\pi V$,
\[
|g_n| \leq \frac{1}{V^n} \max_{|w|=V} |g(w)|.
\]
Fix $\rho<V<T$, with which
\[
|n \rho^n g_n| \leq n \left(\frac{\rho}{V}\right)^n \max_{|w|=V} |g(w)|,
\] 
and for $\frac{\rho}{V} < \delta < 1$ we have
\[
|n\rho^n g_n| = O(\delta^{n}).
\] 
Using this and
\[
\left| \sum_{j=1}^s p_j u_j^{-n} \right| \leq \sum_{j=1}^s p_j = O(1),
\]
 \eqref{pjsum} yields
 \[
 \frac{\rho^n}{1-q_0^n} = O(1).
 \] 
As 
\[
\frac{\rho^n}{1-q_0^n} = \frac{\rho^n}{1-e^{2\pi inx}}
=\frac{\rho^n}{-2i e^{\pi inx} \sin \pi nx}
=\frac{i e^{-\pi inx} \rho^n}{2 \sin \pi nx},
\]
we get 
\[
\frac{\rho^{2n}}{\sin^2 \pi nx} = O(1).
\]
For any $M$, there is some $n_M$ such that $\sin^2 \pi n_M x \geq M$, and thus the above estimate is contradicted
if $\rho=1$; hence $0<\rho<1$. (We emphasize that $\rho<1$ is deduced from the assumption 
$R>\rho$, which we are showing to imply a contradiction.) 
By Lemma \ref{axC} we then get that 
\begin{equation}
\sum_{n=1}^N \frac{\rho^{2n}}{\sin^2 n\pi x} = o(N).
\label{sumoN}
\end{equation} 

Now, multiplying each side of  \eqref{pjsum} by its complex conjugate and using
that $|n\rho^n g_n| = O(\delta^n)$ and that
$\left|\sum_{j=1}^s p_j u_j^{-n} \right| = O(1)$,
\[
\frac{\rho^{2n}}{4\sin^2 \pi nx} = \left( \sum_{j=1}^s p_j u_j^{-n} \right)\left( \sum_{k=1}^s p_k u_k^{n} \right)
+O(\delta^n),
\]
i.e.
\begin{equation}
\frac{\rho^{2n}}{4\sin^2 \pi nx} = \sum_{j=1}^s p_j^2 + \sum_{j \neq k} p_j p_k (u_j^{-1} u_k)^n + O(\delta^n).
\label{pjsquared}
\end{equation}
Let $E=\{(j,k): 1 \leq j,k \leq s, j \neq k\}$ and let
$P=\sum_{j=1}^s p_j^2 > 0$.
Then summing \eqref{pjsquared} for $n=1,\ldots,N$,
using $\sum_{n=1}^N \delta^n = \delta\cdot \frac{1-\delta^N}{1-\delta} = O(1)$, 
\[
\sum_{n=1}^N \frac{\rho^{2n}}{4\sin^2 \pi nx} = NP + \sum_{(j,k) \in E} p_j p_k \left( \sum_{n=1}^N (u_j^{-1} u_k)^n \right) + O(1).
\]
Because $u_j \neq u_k$ for $(j,k) \in E$, we have according to Lemma \ref{geometricbound} that
\[
\left| \sum_{n=1}^N (u_j^{-1} u_k)^n \right| \leq \frac{2}{|1-u_j^{-1} u_k|} =O(1).
\]
Thus
\[
\sum_{n=1}^N \frac{\rho^{2n}}{\sin^2 \pi nx} = 4NP + O(1),
\]
and because $P>0$ this contradicts \eqref{sumoN}. Therefore, it is false that $R>\rho$, which means that
 $R \leq \rho$, proving the claim.
\end{proof}

\begin{theorem}
Let $x \in \Omega$ and let $\rho_1$ and $R_1$  respectively be the radii  of convergence of the power
series
\[
\sum_{n=1}^\infty \frac{z^n}{\sin n\pi x},\qquad \sum_{n=1}^\infty \frac{z^n}{\sin \pi x \cdot
\sin 2\pi x  \cdots \sin n\pi x}.
\]
Then $R_1=\frac{\rho_1}{2}$. 
\label{HLsineseries}
\end{theorem}
\begin{proof}
\[
|1-q_0^n| = |1-e^{2\pi inx}| = 2 |\sin \pi nx|.
\]
By the Cauchy-Hadamard formula, 
\[
R_1 = \liminf_{n \to \infty} |\sin \pi x\cdots \sin n\pi x|^{1/n} = 
\frac{1}{2} \cdot \liminf_{n \to \infty} |(1-q_0)\cdots (1-q_0^n)|^{1/n}
=\frac{R}{2}
\]
and 
\[
\rho_1 = \liminf_{n \to \infty} |\sin n\pi x|^{1/n} = 
\rho.
\]
Theorem \ref{efz} says $R=\rho$, hence $R_1 = \frac{R}{2} = \frac{\rho}{2}=\frac{\rho_1}{2}$.  
\end{proof}

By Lemma \ref{sameradius} and Theorem \ref{roclemma}, for almost all $x$, the power series $\sum_{n=1}^\infty \frac{z^n}{\sin n\pi x}$ has radius of convergence $1$.
Then using Theorem \ref{HLsineseries}, for
almost all $x$  the power series $\sum_{n=1}^\infty \frac{z^n}{\sin \pi x \cdot
\sin 2\pi x  \cdots \sin n\pi x}$ has radius of convergence $\frac{1}{2}$, and thus for almost all $x \in \Omega$,
\[
\liminf_{n \to \infty} \left( \prod_{k=1}^n |\sin k\pi x| \right)^{1/n} = \frac{1}{2}.
\]
Hardy and Littlewood give a separate argument \cite[p.~88, Eq.~4.3]{XXIV} proving that for almost all $x \in \Omega$,
\[
\limsup_{n \to \infty} \left( \prod_{k=1}^n |\sin k\pi x| \right)^{1/n} = \frac{1}{2},
\]
and combining the formulas for the limit inferior and limit superior yields Theorem \ref{bigtheorem}.

Lubinsky \cite{MR1684685} proves more results about products of the form $\prod_{k=1}^n (1-e^{2\pi i k\theta})$. For example,
Lubinsky \cite[p.~219, Theorem 1.1]{MR1684685} proves that for all $\epsilon>0$, for almost all
$\theta \in \Omega$ we have
\[
\left|\log \left| \prod_{k=1}^n (1-e^{2\pi ik\theta})\right| \right| = O((\log n)(\log \log n)^{1+\epsilon}).
\]

The  author \cite[p.~532, Theorem 2]{MR3061031} gives asymptotic expressions for  the $L^p[0,1]$ norm
of
$\prod_{k=1}^n (1-e^{2\pi ik\theta})$ as $n \to \infty$, for $1 \leq p \leq \infty$.

For $\omega \in  \mathbb{R}$ let $P_n(\omega) = \prod_{r=1}^n |2\sin \pi r \omega|$. 
Let $F_n$ be the $n$th Fibonacci number and let
 $g = \frac{-1+\sqrt{5}}{2}$.
 Verschueren and Mestel \cite[p.~204, Theorem 2.2]{verschueren} prove that
 there is some $c = 2.407\ldots$ such that
\[
P_{F_n}(g) \to c,\qquad n \to \infty
\]
and
\[
\frac{P_{F_{n-1}}(g)}{F_n} \to \frac{c}{2\pi\sqrt{5}},\qquad n \to \infty,
\]
and that there are $C_1 \leq 0$ and $C_2 \geq 1$ such that for all $n$,
\[
n^{C_1} \leq P_n(g) \leq n^{C_2}.
\]

Let $X$ be a measure space with probability measure $\lambda$. Following  \cite[p.~21, Definition 3.6]{EMS}, we say that a measure preserving map $T:X \to X$ is \textbf{$r$-fold mixing}
if for all $g,f_1,\ldots,f_r \in L^{r+1}(X)$ we have 
\begin{equation}
\begin{split}
&\lim_{m_1 \to \infty,\ldots,m_r \to \infty} \int_X g(t) \cdot \prod_{k=1}^r f_k\left(T^{\sum_{j=1}^k m_j}(t)\right) d\lambda(t)\\
=&\left(\int_X g(t) d\lambda(t)\right) \prod_{k=1}^r \left( \int_X f_k(t) d\lambda(t) \right).
\end{split}
\label{rmixing}
\end{equation}
If for each $r$ the map $T$ is $r$-fold mixing, we say that $T$ is \textbf{mixing of all orders}.

Let $q \geq 2$ be an integer, and define $T_q:[0,1] \to [0,1]$ by 
$T_q(t)=R(qt)$. We assert that $T_q$ is mixing of all orders. This can be proved by first showing that
the dynamical system $([0,1],\mu,T_q)$ is  isomorphic to a Bernoulli shift (cf. \cite[p.~17, Example 2.8]{einsiedler}). This implies that if the Bernoulli shift is $r$-fold mixing then
$T_q$ is $r$-fold mixing. One then shows that a Bernoulli shift is mixing of all orders \cite[p.~53, Exercise 2.7.9]{einsiedler}. Using that $T_q$ is mixing of all orders gets us the following result.

\begin{theorem}
Let $q \geq 2$ be an integer. For each $n \geq 1$ we have
\[
\lim_{m \to \infty} \int_0^1 |\sin(2\pi t)| \cdot \prod_{k=1}^n \left| \sin\left(2\pi q^{km} t\right) \right|  dt
= \left( \frac{2}{\pi} \right)^{n+1}.
\]
\end{theorem}
\begin{proof}
Define
$g(t)=f_1(t)=\cdots=f_n(t)=|\sin(2\pi t)|$. For any nonzero integer $N$ we have
\[
\int_0^1 |\sin(2\pi N t)| dt=\frac{2}{\pi},
\]
and it follows from \eqref{rmixing}, using $m_1=m, \ldots, m_n=m$, that
\[
\lim_{m \to \infty} \int_0^1 |\sin(2\pi t)| \cdot \prod_{k=1}^n \left| \sin\left(2\pi q^{km} t\right) \right|  dt
= \left( \frac{2}{\pi} \right)^{n+1}.
\]
\end{proof}

See Sinai \cite{Sinai1994}.

Write
$S_k(\alpha)=\sum_{j=1}^k X_j(\alpha)$, where $X_j(\alpha)=\log|2-2\cos(2\pi j\alpha)|$ and
$\alpha=\frac{-1+\sqrt{5}}{2}$. 
Knill and Tangerman \cite{tangerman} 
talk about  motivations from KAM theory for caring about these sums. See Lagarias \cite{lagarias} and Ghys \cite{ghys} for more on small divisors in Hamiltonian dynamics, 
and Carleson and Gamelin \cite[p.~48, Theorem 7.2]{carleson} and Yoccoz \cite{yoccoz} for Arnold's theorem on analytic circle diffeomorphisms.

Marmi and Sauzin \cite{sauzin}

\section{Conclusions}
Kac and Salem \cite{MR0089284} prove the following. Let $c_k$ be a sequence of nonnegative real numbers for which 
$\sum_{k=1}^\infty c_k<\infty$. If the series
 \[
 \sum_{k=1}^\infty c_k \frac{1}{|\sin kx|}
 \]
 converges in a set of positive measure, then
 \[
 \sum_{k=1}^\infty c_k \log\left( \frac{1}{c_k} \right)<\infty,
 \]
 and if this condition is satisfied then $ \sum_{k=1}^\infty c_k \frac{1}{|\sin kx|}$ converges for almost all $x$.

Muromskii \cite[p.~54, Theorem 1]{MR0161081} proves that if $c_k$ is a sequence of nonnegative real numbers,  if $\alpha>1$,
 and if there is a set of positive measure on which the series
 \[
 \sum_{k=1}^\infty c_k \frac{1}{|\sin kx|^\alpha}
 \]
 converges, then for any $\delta>0$, the series
 \[
 \sum_{k=1}^\infty c_k^{\frac{4}{\alpha+3}+\delta}
 \]
 converges. 

Let $X$ be a random variable that is uniformly distributed on $[0,1]$. Kesten \cite[p.~111, Theorem 1 ]{MR0104640} proves that
 if $\sum_{k=0}^\infty |c_k|<\infty$, then the series
 \[
 \sum_{k=0}^\infty \frac{c_k}{\sin (2\pi 2^k X)}
 \]
 converges with probability $1$. Stated using measure theory, the conclusion is that for almost all $x \in \Omega$, the series 
 \[
  \sum_{k=0}^\infty \frac{c_k}{\sin (2\pi 2^k x)}
 \]
 converges. Kesten \cite[p.~114, Theorem 3]{MR0104640} also proves if $a_n \in \mathbb{R}$ and $a_n \to \infty$, then
 \[
 \frac{1}{a_n} \sum_{k=0}^{n-1} \frac{1}{\sin(2\pi 2^k X)} \to 0
 \]
 in probability. Stated using measure theory, the conclusion is that for each $\epsilon>0$,
 \[
 \lim_{n \to \infty} \mu\left\{x \in \Omega: \left|  \frac{1}{a_n} \sum_{k=0}^{n-1} \frac{1}{\sin(2\pi 2^k x)} \right| \geq \epsilon \right\} =0.
 \]
 
For $\mathbb{T}^d = \mathbb{R}^d / \mathbb{Z}^d = (\mathbb{R}/\mathbb{Z})^d$, 
let $\sigma$ be Haar measure on $\mathbb{T}$ and let $\sigma_d=\bigotimes_{1 \leq j \leq d} \sigma$ be Haar measure on $\mathbb{T}^d$, with $\sigma(\mathbb{T})=1$.
A sequence $t(n) \in \mathbb{T}^d$, $n \geq 1$,
is said to be \textbf{uniformly distributed} if for any arcs $I_1,\ldots,I_d$ in $\mathbb{T}$, with $I=\prod_{j=1}^d I_j$,
\[
\lim_{N \to \infty} \frac{|\{t(n) \in I : 1 \leq n \leq N\}|}{N} = \sigma_d(I).
\]
\textbf{Kronecker's approximation theorem} \cite[p.~108, Theorem 6.3]{travaglini} states that if
$\alpha_1,\ldots,\alpha_d \in \mathbb{R}$ and $\{1,\alpha_1,\ldots,\alpha_d\}$ is linearly independent over $\mathbb{Q}$,
then the sequence $(n\alpha_1+\mathbb{Z},\ldots,n\alpha_d+\mathbb{Z})$, $n \geq 1$, is uniformly distributed in $\mathbb{T}^d$.
Meyer \cite{meyer} is
a thorough presentation of
multidimensional Diophantine approximation and Diophantine approximation with locally compact abelian groups, and harmonic
analysis involving sets satisfying various Diophantine properties.

See Rudin \cite{rudinfourier}.

See Abraham and Marsden \cite[p. 395, Proposition 5.2.23]{abraham} and \cite[pp. 818-820, Proposition 5.2.23]{abraham}.

Measure theoretic results in Diophantine approximation are presented in Khinchin \cite{MR1451873}, 
Einsiedler and Ward \cite[Chapter 3]{einsiedler},
Rockett and Sz\"usz \cite[Chapters V and VI]{rockett},   Billingsley \cite[pp.~13--15, 319--326]{billingsley},  Kac \cite[Chapter 5]{kac},  Bugeaud \cite{bugeaud}, and Kesseb\"ohmer,
Munday and Stratmann \cite{kessebohmer}.
The significance of continued fraction expansions of irrational numbers in the early history of axiomatic probability theory is described
by Barone and Novikoff \cite{barone}, Durand and Mazliak \cite{MR2859995}, and von Plato \cite{vonplato}.
Veech \cite{veech} presents material on Diophantine approximation in the setting of topological dynamics.

We have been interested in results about almost all $x \in \Omega$, using Lebesgue measure $\mu$ on $[0,1]$. If $E \subset \Omega$ has $\mu(E)=0$, one can 
ask what the Hausdorff dimension $\dim_H E$ of the set $E$ is. Let $\mathcal{B}_K$ be the set of those $x \in \Omega$ such that $a_n(x) \leq K$ for all
$n \geq 1$, and let $\mathcal{B} = \bigcup_{K \geq 1} \mathcal{B}_K$, the set of those $x \in \Omega$ with bounded partial quotients.
We have  already stated that $\mu(\mathcal{B})=0$ \cite[p.~60, Theorem 29]{MR1451873}, and Jarn\'ik  \cite[Theorem~4.3]{MR2112110} proves that
the Hausdorff dimension of $\mathcal{B}$ is in fact 1; cf. Falconer \cite[p.~155, Theorem 10.3]{falconer} and Wolff \cite[p.~67, Chapter 9]{wolff} on Hausdorff dimension.
Hensley \cite{hensley} proves 
\[
\dim_H \mathcal{B}_K = 1 - \frac{6}{\pi^2} K^{-1} - \frac{72}{\pi^4} K^{-2} \log K + O(K^{-2}),\qquad K \to \infty.
\]
Dodson and Kristensen \cite{MR2112110} give a survey of results on the Hausdorff dimension of various sets that
appear in Diophantine approximation.

For a decreasing positive function $\psi$, the set
\[
W(\psi) = \{ x \in [0,1] : \textrm{$\norm{qx} < q \psi(q)$ for infinitely many  $q \in \mathbb{N}$} \}
\]
can be written as the limsup of a sequence of sets, 
\[
W(\psi) = \limsup_{n \to \infty} W(\psi, n) = \bigcap_{N=1}^\infty \bigcup_{n=N}^\infty W(\psi,n),
\]
where 
\[
W(\psi, n) = \bigcup_{2^{n-1} < q \leq 2^n} \bigcup_{0\leq p \leq q} \left( \frac{p}{q} - \psi(q), \frac{p}{q} + \psi(q) \right) \cap [0,1].
\]
One can exploit nice properties of limsup sets, such as the Borel-Cantelli lemma and invariance under ergodic transformations, to prove fundamental results in Diophantine approximation.  Beresnevich, Dickinson and Velani \cite{limsup} use this motiviation of Diophantine approximation to build a framework for a natural class of limsup sets on compact metric spaces.  Their general results readily imply the divergent  case of  Khinchin's theorem: $\mu(W(\psi)) = 0$ if $\sum q \psi(q) < \infty$  (Lemma \ref{khinchin}) and 
$\mu(W(\psi)) = 1$ if $\sum q \psi(q) = \infty$.  Their framework also establishes the divergent case of Jarn\'ik's  theorem: the $f$-Hausdorff dimension of $W(\psi)$ is $0$ if $\sum q f(\psi(q)) < \infty$ and is infinity if $\sum q f(\psi(q)) = \infty$, where $f$ is a
\textbf{dimension function} such that $r^{-1} f(r) \to \infty$ as $r \to 0$, and $r^{-1} f(r)$ is decreasing.

As well, rather than making statements about subsets of $\Omega$ of measure $1$, we can talk about sets whose complements are meager. (Measure theoretically,
the notion of a negligible set is made precise as a set of measure $0$, and topologically the notion of a negligible set is made precise as a meager set.)
Some results of this type are proved in Oxtoby \cite[Chapter 2]{oxtoby}.

Let $p$ be prime, let $N_p=\{0,\ldots,p-1\}$, and let $\mathbb{Q}_p \subset \prod_{\mathbb{Z}} N_p$ be the $p$-adic numbers. For $x \in \mathbb{Q}_p$
let
\[
v_p(x) = \inf\{k \in \mathbb{Z}: x(k) \neq 0\},\qquad |x|_p = p^{-v_p(x)},
\]
and let $\mathbb{Z}_p = \{x \in \mathbb{Q}_p: v_p(x) \geq 0\}$, the $p$-adic integers. 
Let $\mu_p$ be the Haar measure on the additive locally compact abelian group $\mathbb{Q}_p$ with
$\mu_p(\mathbb{Z}_p)=1$. 
We call $\lambda \in \mathbb{Z}_p$ a \textbf{$p$-adic Liouville number} if $\nu(\lambda)=\liminf_{n \to \infty} |n-\lambda|_p^{1/n}=0$, and let
$\mathscr{L}_p$ be the set of $p$-adic Liouville numbers. One checks that if $\lambda \in \mathbb{Z}_{\geq 0}$ then
$\nu(\lambda)=1$  \cite[p.~201, Exercise 66.A]{schikhof}.
It can be proved that $\mathscr{L}_p$ is a dense $G_\delta$ set in $\mathbb{Z}_p$ \cite[p.~204,
Theorem 67.3]{schikhof} and that
$\mu_p(\mathscr{L}_p)=0$ \cite[p.~205, Theorem 67.4]{schikhof}.
One reason for caring about $p$-adic Liouville numbers is that
if $x \in \mathbb{Z}_p$ is algebraic over $\mathbb{Q}$ then $\nu(x)=1$, and hence
a $p$-adic Liouville number is transcendental over $\mathbb{Q}$ \cite[p.~203, Theorem 67.2]{schikhof}.

Unlike in estimating exponential sums, the sums that we have been estimating
in this paper do not have cancellation. Instead we have estimated
them by showing that the terms are only occasionally large. 
For $A_n(t)=\sum_{k=1}^n \frac{\sin kt}{k}$, it can be proved \cite[p.~74, no.~25]{polyaII} that
\[
|A_n(t)| < \int_0^\pi \frac{\sin \theta}{\theta} d\theta=1.8519\ldots
\]
On the other hand, let $M_n$ be the maximum of $\sum_{k=1}^n \frac{|\sin kt|}{k}$. It can be proved \cite[p.~77, no.~38]{polyaII} that
\[
M_n=\frac{2}{\pi}\log n +O(1).
\]

Walfisz \cite{walfisz} presents results of his, of Oppenheim, and of Chowla  on sums
\[
\sum_{j \leq n} g(j) e^{2\pi i jx}
\]
 for $g(j)=r_k(j)$, the number of ways to write
$j$ as a sum of $k$ squares,  and for $g(j)=d(j)$, the number of positive divisors of $j$. One of the results of Chowla is that  if $x \in \Omega$ has bounded
partial quotients, then
\[
\sum_{j=1}^n d(j) e^{2\pi i jx}=O\left(n^{\frac{1}{2}} \log n \right).
\]
One of the results Walfisz proves is that if $\epsilon>0$, then for almost all $x \in \Omega$,
\[
\sum_{j=1}^n d(j) e^{2\pi ijx}= O\left( n^{\frac{1}{2}} (\log n)^{2+\epsilon} \right).
\]
Wilton \cite{wilton} proves some similar results. For example, Wilton proves that for any $x \in \Omega$, 
\[
\sum_{j=1}^n \frac{d(j)}{j} \cos 2\pi jx = o((\log n)^2),
\]
See Jutila's book on exponential sums \cite{jutila}. 

Let $f(t)=P_1(t)$ for $t \not \in \mathbb{Z}$ and $\{t\}=0$ for $t \in \mathbb{Z}$, where $P_1$ is a periodic
Bernoulli function. Namely, for all $t \in \mathbb{R}$,
\[
f(t) = - \frac{1}{\pi} \sum_{m=1}^\infty \frac{\sin 2\pi mt}{m}.
\]
Define $S_N(\theta) = \sum_{n=1}^N \frac{\mu(n)}{n} f(n\theta)$, where
$\mu$ is the \textbf{M\"obius function}. 
Davenport \cite[p.~11, Theorem 2]{davenport1937} proves that there is some $C$ such that for all $N$ 
and for all $\theta$,
\[
\left| \sum_{n=1}^N \frac{\mu(n)}{n} f(n\theta) \right| \leq C.
\]
Davenport \cite[p.~13, Theorem 4]{davenport1937}  also proves that for almost all $\theta$,
\[
\sum_{n=1}^N \frac{\mu(n)}{n} f(n\theta) \to -\frac{1}{\pi} \sin 2\pi \theta.
\]
See Jaffard \cite{jaffard}.

It would be a useful project to give an organized presentation of Hardy and Littlewood's results on Diophantine approximation. Their papers in this area are all
included in Hardy's collected works \cite{collectedpapers}. Hardy and Littlewood proved many pleasant results on various sums and series with coefficients related
to $\sin(n \pi x)$ and $R(nx)$. It would be desirable to streamline and systematically prove these results, to let a modern reader to be able to understand them without having to 
read the whole series of papers to figure out what results are being tacitly used from earlier work or assumed as general knowledge.
There is only a bare summary of Hardy and Littlewood's work in the commentary in Hardy's collected papers. Hardy's work on Diophantine approximation is briefly summarized by
Mordell \cite{mordell}. See also lecture V of Hardy's lectures on Ramanujan \cite{ramanujan}.

\section*{Acknowledgments}
The author thanks Hervé Queffélec (Université de Lille) and Martine Queffelec (Université de Lille) for a long correspondence on Hardy and Littlewood's
``curious power-series'' \cite{XXIV}. 

The author thanks Jeremy Voltz (University of Toronto) for discussions on Beresnevich, Dickinson and Velani's monograph \cite{limsup} on limsup sets.

\bibliographystyle{plain}
\bibliography{diophantine}

\end{document}